\documentclass{article}
\usepackage[margin=1in]{geometry}
\DeclareMathAlphabet{\mathcal}{OMS}{cmsy}{m}{n}
\input xy
\xyoption{all}

\usepackage{cite}
\usepackage{tabu}
\usepackage{amsmath,amsfonts,amsthm,amssymb,amscd,amsopn}
\usepackage{graphics}
\usepackage{mathdots}
\usepackage{mathrsfs}
\usepackage{latexsym}
\usepackage{graphics}
\usepackage{color}
\usepackage{hyperref}
\usepackage{enumerate}
\usepackage{lscape}
\usepackage{mathtools}
\usepackage{tikz-cd}
\usepackage{todonotes,comment}
\usepackage{longtable}

\allowdisplaybreaks
 
\theoremstyle{plain}
  \newtheorem{thm}{Theorem}[section]
  \newtheorem*{thm*}{Theorem}
  
  \newtheorem{conj}[thm]{Conjecture}
  \newtheorem{prop}[thm]{Proposition}
  \newtheorem{cor}[thm]{Corollary}
  \newtheorem{lem}[thm]{Lemma}

\theoremstyle{definition}
  \newtheorem{defn}[thm]{Definition}
  \newtheorem{exmp}[thm]{Example}

  \newtheorem{rem}[thm]{Remark}

\theoremstyle{remark}

\DeclareMathOperator{\Cl}{Cl}

\DeclareMathOperator{\Sym}{Sym}

\DeclareMathOperator{\disc}{disc}

\DeclareMathOperator{\Mat}{Mat}
\DeclareMathOperator{\Tr}{Tr}

\def\Z{{\mathbb Z}}
\def\T{{\mathcal T}}

\def\W{{\mathcal W}}

\def\irr{{\rm irr}}

\def\GL{{\rm GL}}
\def\SL{{\rm SL}}
\def\Cl{{\rm Cl}}

\def\Sym{{\rm Sym}}

\def\P{{\mathbb P}}

\def\disc{{\rm disc}}

\def\Vol{{\rm vol}}

\def\R{{\mathbb R}}
\def\C{{\mathbb C}}
\def\F{{\mathbb F}}

\def\FF{{\mathcal F}}

\def\I{{\mathcal I}}
\def\Q{{\mathbb Q}}

\def\V{{\mathcal V}}
\def\Z{{\mathbb Z}}
\def\P{{\mathbb P}}
\def\F{{\mathbb F}}
\def\Q{{\mathbb Q}}
\def\O{{\mathcal O}}

\def\factor{(-1)^{(n-1)/2}}
\def\factoralt{(-1)^{(n-1)/2}}
\def\SS{\Sigma}
\def\diag{{\mathrm{diag}}}

\newcommand{\SO}{\mathrm{SO}}

\newcommand{\E}{\mathcal{E}}
\newcommand{\J}{\mathcal{J}}

\newcommand{\nc}{\newcommand}
\nc{\on}{\operatorname}
\nc{\renc}{\renewcommand}
\nc{\wt}{\widetilde}
\nc{\defeq}{\vcentcolon=}
\nc{\eqdef}{=\vcentcolon}
\nc{\Spec}{\on{Spec}}
\nc{\ol}{\overline}
\renc{\d}{\partial}

\title{Monogenicity and $2$-torsion in the class group of number fields of odd degree}

\author{Artane Siad}
\date{June 2, 2026}

\begin{document}

\maketitle

\begin{abstract}
We study the average $2$-torsion in the class group of monogenised fields of odd degree. Bhargava--Hanke--Shankar \cite{BhargavaHankeShankar} have recently shown that for a fixed signature, the average number of non-trivial $2$-torsion elements in the class group of monogenised cubic fields is exactly twice the value predicted by the Cohen--Lenstra--Martinet--Malle heuristic over the full $S_3$ family. For any odd degree $n \ge 3$ and signature, we prove that the average number of non-trivial $2$-torsion elements in the class group of monogenised fields is at most twice the value predicted by the Cohen--Lenstra--Martinet--Malle heuristic over the full $S_n$ family. Conditional on a tail estimate for $n \ge 5$, this establishes that the doubling phenomenon discovered by Bhargava--Hanke--Shankar persists across all odd degrees and signatures.
\end{abstract}

\section{Introduction}

The Cohen--Lenstra heuristics, and their generalisation by Cohen--Martinet, form our best conjectural description of the distribution, for ``good'' primes $p$, of the $p$-primary part of the class group over families of fields of a fixed degree, signature, and Galois group (of the Galois closure), ordered by reasonably natural heights. The heuristics have undergone waves of corrections, extensions, and modifications since their introduction in the 1980s thanks to the collective efforts of several authors, including notably, but certainly not exhaustively, Gerth, Martinet, Malle, Wood, Liu, Zureick-Brown, Wang, Bartel, and Johnston (see \cite{MR756082, MR866103, MR1037430, MR1226813, MR759260, MR887792, MR2778658, MR3858473, 1907.05002, MR4277275, MR4105790, 2005.11533}). Despite these significant advances, few of its predictions have been verified over number fields beyond the seminal work of Davenport--Heilbronn \cite{MR491593} on average $3$-torsion in the class group of quadratic fields and Bhargava \cite{MR3272925} on average $2$-torsion in the class group of cubic fields, both ordered by discriminant, and their generalisations \cite{MR936994, MR3369305,  MR3471250,1512.03035,2110.07712}. Additional work includes the conditional determination by Ho--Shankar--Varma \cite{MR3782066} of the average $2$-torsion in the class group of odd degree number fields associated to binary forms ordered by box height. At the same time, a string of strong results has been obtained over function fields \cite{MR2184814,MR3488737,MR3705261,MR3858473,MR3909900,1907.05002,LandesmanLevy}. 

In all the absolute cases over $\Q$ where the heuristics have been verified, they exhibit a remarkable stability under local conditions --- the averages remain the same if, instead of averaging over the full family of fields of a fixed degree and signature, we average over subfamilies defined by finitely (or, in certain cases, infinitely) many local conditions \cite{MR3471250,MR3369305}. Naturally, this raises the question of the sensitivity of the heuristics to global conditions.

Bhargava--Hanke--Shankar \cite{BhargavaHankeShankar} have recently shown that one such global condition, monogenicity --- the property of having a ring of integers generated by one element as a $\Z$-algebra --- causes the average number of non-trivial $2$-torsion elements in the class group of cubic fields of a fixed signature to double compared to the value expected by the Cohen--Lenstra--Martinet--Malle heuristics for the corresponding full $S_3$ family. In this paper, we generalise both their techniques and results to prove the following theorem, a precise version of which will be stated in \S \ref{subsec: statement of main theorem}.
\begin{thm*}
Granted a tail estimate for the lower bound in degrees $\ge 5$, the mean number of non-trivial $2$-torsion elements in the class groups of monogenised fields of any given odd degree and signature is twice the value predicted by the Cohen--Lenstra--Martinet--Malle heuristic over the corresponding full $S_n$ family.
\end{thm*}
\noindent These conditional averages (and unconditional upper bounds) stay the same if we average over subfamilies of monogenised fields defined by finitely (or, in certain cases, infinitely) many local conditions. We also obtain upper bounds and conditional averages for the narrow class group.

This result offers a new perspective on class group heuristics by raising an intriguing possibility: that global conditions skew the random behaviour of the class group in subtle, predictable ways, which encode the special arithmetic features of the conditions. Indeed, it suggests that monogenicity does not simply shift the entire distribution of the $2$-ranks of the class group by $1$, but, instead, has the more subtle effect of doubling the expected number of \emph{non-trivial} $2$-torsion elements in the class group, in every odd degree. 

The proof involves a delicate interplay between local objects at the $2$-adic and archimedean places of $\Q$, highlighting the importance of these primes for understanding the $2$-primary part of the class group, and echoing recent theoretical developments in the modelling of the $2$-Selmer group of number fields in \cite{MR3873132} and \cite{Breen}. Together with \cite[Theorem 2]{MR3782066}, our statement is among the first to deal uniformly with class group averages in infinitely many degrees and signatures.

The method developed in this paper, and detailed in \S \ref{subsec: main ideas}, is a robust tool which opens the door to studying the behaviour of $2$-torsion in the class group of several families of interest to arithmetic statistics. For example, in \cite{Siadeven} we build on the work of this paper to prove the existence in any even degree at least $4$ of infinitely many number fields with odd class number, closing the circle on a line of results on the parity of class numbers dating back to Gauss \cite{MR0197380}, while in the forthcoming work \cite{Siadrelative} we establish qualitative Cohen--Lenstra--Gerth statements on the existence of infinitely many fields with prescribed $2\Cl[2^\infty]$ in any even degree at least $6$.

\subsection{Statement of main theorem}
\label{subsec: statement of main theorem}

A ring $R$ is an $n$-ic ring if it is a free $\Z$-module of rank $n$. An $n$-ic ring $R$ is said to be monogenic if it is generated by a single element as a $\Z$-algebra. A monogenised $n$-ic ring is a pair $(R,\alpha)$ comprising a monogenic $n$-ic ring $R$ and a monogeniser $\alpha$ --- an $\alpha \in R$ such that $R = \mathbb{Z}[\alpha]$. Two monogenised $n$-ic rings $(R,\alpha)$ and $(R,\alpha')$ are said to be isomorphic if $R$ and $R'$ are isomorphic through a ring isomorphism sending $\alpha$ to $\alpha' + m$ for some $m \in \mathbb{Z}$. \footnote{The results of this paper are unchanged if $\pm$ is included in front of the $\alpha'$; omitting it makes the exposition clearer.} If a monogenised ring $R$ is an integral domain, then it is an order in its field of fractions $K$, and we say that $R$ is a monogenised order. If it is also maximal in $K$, we say that $(K,\alpha) := (\O_{K},\alpha)$ is a monogenised field.  \footnote{Results of Gy\H{o}ry \cite{MR437489, MR437490} and then Evertse--Gy\H{o}ry \cite{MR1117339, MR3586280, MR4394329}, generalising work of Birch--Merriman \cite{MR306119}, show that the number of monogenisers, up to transformations of the form $\alpha \mapsto \pm \alpha + m$ for $m \in \Z$, is effectively bounded in terms of the degree.} 

Fix a degree $n$ and a signature $(r_1,r_2)$ such that $r_1+2r_2 = n$. We write $\mathfrak{R}_{1}^{n}$ for the set of isomorphism classes of monogenised $n$-ic rings and $\mathfrak{R}_{1,{\rm order}}^{r_{1},r_{2}}$ (resp. $\mathfrak{R}_{1,{\rm max}}^{r_{1},r_{2}}$) for the set of isomorphism classes of monogenised orders (resp. fields) with field of fractions having signature $(r_{1},r_{2})$ (resp. of signature $(r_{1},r_{2})$). We define a naive height on $\mathfrak{R}_{1}^n$ (and thus on $\mathfrak{R}_{1,{\rm order}}^{r_{1},r_{2}}$ and $\mathfrak{R}_{1,{\rm max}}^{r_{1},r_{2}}$). Each class in $\mathfrak{R}_{1}^n$ contains a unique representative $(R,\alpha_0)$ with the property that $0 \le {\rm tr}(\alpha_0) < n$. Writing $f(x) = x^n+a_1 x^{n-1}+\ldots+a_n \in \Z[x]$ for the characteristic polynomial of $\alpha_0$ we define the naive height of the class to be $H_{\text{naive}}([(K,\alpha_0)]) = \max_i \{|a_i|^{1/i}\}$. This is equivalent to the box height on the roots of the characteristic polynomial of $\alpha_0$ as we show in \S \ref{sec: background and preliminaries}. 

To state our main theorem, we will also need the notions of \textit{acceptable} and \textit{very large} families of monogenised orders, which we now explicate. For a ring $T$, let $U(T) = {\rm Sym}^{n}(T^{2})$ denote the space of binary $n$-ic forms with coefficients in $T$. As will be explained in \S \ref{sec: background and preliminaries}, $\mathfrak{R}_{1}^{n}$ can be identified with $\bigsqcup_{0 \le b < n} U_{1,b}(\Z)$ where $U_{1,b}$ denotes the subspace of $U$ consisting of forms with $x^{n}$ coefficient $1$ and $x^{n-1}y$ coefficient equal to $b$. Under this identification, $\mathfrak{R}_{1, {\rm order}}^{r_{1},r_{2}}$ corresponds to $\bigsqcup_{0 \le b < n} U_{1,b}^{r_{2}}(\Z)^{{\rm irr}}$ where $U_{1,b}^{r_{2}}(\Z)^{{\rm irr}}$ consists of forms which are irreducible and have $r_{2}$ pairs of complex roots. Let $\widehat{\Z}$ denote the profinite integers $\prod_{p} \Z_p$. Collections of local conditions are subsets $S = \bigsqcup_{0 \le b < n} \prod_{p} S_{b,p} \subset \bigsqcup_{0 \le b < n} U_{1,b}(\widehat{\Z})$ with the property that each $S_{b,p}$ is non-empty, clopen, and has boundary with measure $0$. Such collections are said to be \textit{acceptable} if for large enough $p$, $S_{b,p}$ includes all elements of $U_{1,b}(\Z_p)$ with discriminant indivisible by $p^{2}$, and \textit{very large} if for large enough $p$, $S_{b,p}$ is equal to $U_{1,b}(\Z_p)$. Subfamilies $\SS \subset \mathfrak{R}_{1, {\rm order}}^{r_{1},r_{2}}$ whose set of associated forms is the preimage under the diagonal embedding $\bigsqcup_{0 \le b < n} U_{1,b}^{r_{2}}(\Z)^{{\rm irr}} \hookrightarrow \bigsqcup_{0 \le b < n} U_{1,b}(\widehat{\Z})$ of an acceptable (resp. very large) collection of local conditions are said to be acceptable (resp. very large). 

For an order $\O$ in a number field $K$, let $\I(\O)$ denote the group of invertible ideals of $\O$, $P(\O)$ the subgroup of principal fractional ideals of $\O$, and $P^+(\O)$ the subgroup of totally positive principal fractional ideals of $\O$ (i.e. ideals of the form $y\O$ for some $y \in K^\times$ with the property that $\sigma(y)$ is positive for any real embedding $\sigma$ of $K$). Let $\Cl(\O) = \I(\O)/P(\O)$ and $\Cl^+(\O) = \I(\O)/P^+(\O)$ denote the class group and the narrow class group of $\O$ respectively, and $\Cl(\O)[2]$ and $\Cl^+(\O)[2]$ denote their $2$-torsion subgroups. We denote by $\I_2(\O)$ the $2$-torsion subgroup of the ideal group $\I(\O)$. Note that $\I_2(\O)$ is trivial if $\O$ is maximal, but may be non-trivial otherwise. 

For a subfamily $\SS \subset \mathfrak{R}_{1, {\rm order}}^{r_1,r_2}$, we define
\[ 
{\rm Avg}(\SS,\Cl(\cdot)[2]) := \lim_{X \rightarrow \infty} \tfrac{\sum\limits_{\substack{\mathcal{O} \in \SS \\ H_{\text{naive}}(\mathcal{O}) < X}} \left| \Cl(\mathcal{O})[2] \right|}{\sum\limits_{\substack{\mathcal{O} \in \SS \\ H_{\text{naive}}(\mathcal{O}) < X}} 1}. \] 
and ${\rm Avg}(\SS,\Cl^{+}(\cdot)[2])$ analogously.

\begin{thm}[Main theorem for monogenised fields and orders]\label{main theorem}
Let $n \ge 3$ be an odd integer and $(r_1,r_2)$ a choice of signature. Let $\SS \subset \mathfrak{R}_{1, \max}^{r_1,r_2}$ be an acceptable family. The average number of $2$-torsion elements in the class group and narrow class group of fields in $\SS$ satisfy 
\begin{equation*}\label{main equation}
{\rm Avg}(\SS,\Cl(\cdot)[2]) \le 1+ \frac{2}{2^{r_1+r_2-1}} \hspace{5ex} \text{ and } \hspace{5ex} {\rm Avg}(\SS,\Cl^+(\cdot)[2]) \le 1 + \frac{1}{2^{\frac{n-1}{2}}}+\frac{1}{2^{r_2}}
\end{equation*}
with equalities conditional on the tail estimate \ref{conj: conjectural tail estimate} for $n \ge 5$. For $n =3$, these upper bounds are also unconditional equalities. For $\SS \subset \mathfrak{R}_{1,{\rm order}}^{r_1,r_2}$ a very large family with the property that the local conditions at $2$ are given modulo $2$, the average over $\O \in \SS$ of
\[
\left|\Cl(\mathcal{O})[2]\right|-\frac{1}{2^{r_1+r_2-1}}\left| \mathcal{I}_2(\mathcal{O}) \right| \hspace{5ex} \text{ and } \hspace{5ex} \left|\Cl^+(\mathcal{O})[2]\right|-\frac{1}{2^{r_2}}\left| \mathcal{I}_2(\mathcal{O}) \right|
\]
are unconditionally equal to $1+\frac{1}{2^{r_1+r_2-1}}$ and $1+\frac{1}{2^{\frac{n-1}{2}}}$ respectively. \footnote{The same holds if we also add the condition that $\SS$ contains only orders in $S_n$-fields.}
\end{thm}

In particular, Theorem \ref{main theorem} gives ${\rm Avg}(\mathfrak{R}_{1,\max}^{1,1},\Cl(\cdot)[2]) = 2$ for complex cubic fields and ${\rm Avg}(\mathfrak{R}_{1,\max}^{3,0}$, $\Cl(\cdot)[2]) = 3/2$ and ${\rm Avg}(\mathfrak{R}_{1,\max}^{3,0},\Cl^{+}(\cdot)[2])\allowbreak = 5/2$ for real cubic fields. These values match those of \cite[Theorem 4]{BhargavaHankeShankar}, although averages there are taken with respect to a slightly different height. 
The Cohen--Lenstra--Martinet--Malle heuristics predict a value of $1+1/2^{r_1+r_2-1}$ for the average number of $2$-torsion elements in the class group of all $S_n$-fields of odd degree $n$ and signature $(r_1,r_2)$. Thus, the conditional value of ${\rm Avg}(\SS,\Cl(\cdot)[2])-1$ in Theorem \ref{main theorem} is precisely twice this expected value. 
The conditional equality ${\rm Avg}(\SS,\Cl(\cdot)^+[2]) =  1 + 1/2^{\frac{n-1}{2}}+1/2^{r_2}$ in Theorem \ref{main theorem} suggests that monogenicity also has an increasing effect on the size of the narrow class group. This was already observed for cubic fields in \cite[Theorem 7]{BhargavaHankeShankar}. 

Combining our proof with the reasoning of \cite[Proof of Theorem 6.8]{MR3782066} shows that there exists an infinite family $ \Sigma' \subset \mathfrak{R}_{1, \max}^{r_1,r_2}$ (not necessarily defined by local conditions) such that any two fields in $\Sigma'$ are not only non-isomorphic as monogenised fields, but also non-isomorphic as fields, and for which Theorem \ref{main theorem} holds. This strongly suggests that the increase in the size of the $2$-torsion in the class group and narrow class groups in Theorem \ref{main theorem} is genuinely due to monogenicity and not some hypothetical artefact of seeing fields with large class numbers many times due to inequivalent monogenisers.

From this family $\Sigma'$, we also deduce that \cite[Theorem 2 and Theorem 4]{MR3782066} hold for monogenic fields, albeit with slightly weaker bounds. 

\begin{cor}\label{main corollary}
Let $n \ge 3$ be an odd integer and $(r_1,r_2)$ a choice of signature. Let $\SS \subset \mathfrak{R}^{r_1,r_2}_{1,\max}$ be an acceptable family. Then the following lower bounds hold:
\begin{enumerate}[1)]
\item The proportion of $S_n$-fields in $\SS$ which have odd class number is at least $1-\frac{2}{2^{r_1+r_2-1}}$. 	
\item If $(r_1,r_2) \neq (1,1)$ we have $\# \left\{R \in \SS \colon \left| {\rm Disc} \left( R \right) \right| < X \, \mathrm{and} \, 2 \nmid \left| \Cl \left(R\right) \right| \right\} \gg X^{\frac{1}{2}+\frac{1}{n}}.$
\item The proportion of $S_n$-fields in $\SS$ which have odd narrow class number is at least $1-\frac{1}{2^{\frac{n-1}{2}}}-\frac{1}{2^{r_2}}$.
\item If $r_2 \neq 0$ and $(r_1,r_2) \neq (1,1)$ we have $\# \left\{R \in \SS \colon \left| {\rm Disc} \left( R \right) \right| < X \, \mathrm{and} \, 2 \nmid \left| \Cl^+ \left(R\right) \right| \right\} \gg X^{\frac{1}{2}+\frac{1}{n}}.$
\end{enumerate}
In particular, we conclude that for any $(r_1,r_2) \neq (1,1)$ there are infinitely many monogenic $S_n$-fields of degree $n$ and signature $(r_1,r_2)$ which have odd class number and, for any $(r_1,r_2) \neq (1,1)$ with $r_2 \neq 0$, infinitely many which have units of every signature. 
\end{cor}

\subsection{Fields associated to polynomials of arbitrary fixed leading coefficient} \label{subsec: intro N-monogenic}

Fix an odd degree $n$, a signature $(r_1,r_2)$ such that $r_1+2r_2 = n$, and a leading coefficient $1 \le N \in \mathbb{N}$. 

For a non-zero integral binary $n$-ic form $f \in \Sym^n(\Z^2)$ which is primitive \footnote{A form is primitive if the $\gcd$ of its coefficients is $1$.}, let $R_f$ denote the coordinate ring of the subscheme of $\P_\Z^1$ described by $f$. We say that $R_f$ is equivalent to $R_{f'}$ if $f(x,y) = f'(x+m,y)$ for some $m \in \Z$. Denote by $\mathfrak{R}_{N,\max}^{r_1,r_2}$ the family of equivalence classes of rings $R_f$ corresponding to binary $n$-ic forms $f \in \Sym_n(\Z^2)$ with $r_1$ real roots and $2r_2$ complex roots, leading coefficient $N$, and with the property that: 1) $f$ is primitive; 2) $R_f$ is the maximal order in its fraction field $K_f$. \footnote{We call such an equivalence class, and also $K_{f}$, an $N$-monogenised field.} An equivalence class in $\mathfrak{R}_{N,\max}^{r_1,r_2}$ contains a unique representative $R_{f_0}$ with the $x^{n-1}y$ coefficient of $f_0$ contained in the interval $[0,Nn)$. Writing $f_0(x,y)= Nx^n+a_1 x^{n-1}y+\ldots+a_ny^n$, we define the height of the equivalence class as $H_{\text{naive}}([R_{f_0}]) := \max_i |a_i|^{1/i}$ and we order $\mathfrak{R}_{N,\max}^{r_1,r_2}$ using $H_{\text{naive}}$. Elements of $\mathfrak{R}_{N,\max}^{r_1,r_2}$ are called \textit{isomorphism classes of $N$-monogenised fields}. 

We have a notion of an \textit{acceptable family} of $N$-monogenised fields $\SS \subset \mathfrak{R}_{N,\max}^{r_1,r_2}$ associated to an \textit{acceptable collection} of local specifications $(\SS_p)_p$, see \S \ref{sec: averages for $N$ monogenised fields}. For $p|N$, the \emph{even ramification density} of $\SS$ at $p$, $\square_p(\SS)$, is the density in $\Sigma_p$ of $f_p \in \Sigma_p$ with the property that $f_p(x,1)$ is a square mod $p$ up to multiplication by $(\Z/p)^\times$, that is $f_p(x,1) = u g_p(x,1)^2$ for some polynomial $g_p(x,1)$ and some unit $u \in (\Z/p)^\times$. 

In \S \ref{sec: averages for $N$ monogenised fields}, we use the tools developed for the proof of Theorem \ref{main theorem} to prove the following statement.
 
\begin{thm}[Main theorem for $N$-monogenised fields] \label{thm: main N monogenic theorem}
Let $n \ge 3$ be an odd integer, $(r_1,r_2)$ a choice of signature and $N = m^2k$ with $k$ squarefree.  Let $\SS \subset \mathfrak{R}_{N, \max}^{r_1,r_2}$ be an acceptable family of isomorphism classes of $N$-monogenised fields of odd degree, corresponding to an acceptable collection of local specifications $(\SS_p)_p$, with the property that local specification at $2$, $\SS_2$, is given modulo $2$ and only contains forms in $U_N^{r_2}(\Z_2)$ whose mod $2$ reduction is unramified. For $p|k$, let $\square_p(\SS)$ denote the even ramification density of forms in $\SS$. Then, the average number of $2$-torsion elements in the class group and narrow class group of fields in $\SS$ satisfy
\[{\rm Avg}(\SS,\Cl(\cdot)[2]) \le 1+ \frac{1+ \prod_{p|k}\square_p(\SS)}{2^{r_1+r_2-1}} \hspace{5ex} \text{ and } \hspace{5ex} {\rm Avg}(\SS,\Cl^+(\cdot)[2]) \le 1 + \frac{\prod_{p|k}\square_p(\SS)}{2^{\frac{n-1}{2}}}+\frac{1}{2^{r_2}}\]
with equalities conditional on a tail estimate of the form \ref{conj: conjectural tail estimate}.
\end{thm}

Theorem \ref{thm: main N monogenic theorem} is due to \cite{BhargavaHankeShankar} in degree $n=3$ and to \cite{Ashvin} for $\Cl(\cdot)[2]$ in all odd $n > 3$, while it was unknown in odd $n>3$ for $\Cl^+(\cdot)[2]$. In \cite{Ashvin}, Swaminathan first uses a new parametrisation \footnote{This new parametrisation is roughly a quadratic twist of the parametrisation of \cite{MR3187931}.} to reduce the problem of computing ${\rm Avg}(\SS,\Cl(\cdot)[2])$ for subfamilies $\Sigma \subset \mathfrak{R}_{N,\max}^{r_1,r_2}$ to the monogenic problem together with a cuspidal correction factor, before applying the techniques and results of the present paper. Our proof of Theorem \ref{thm: main N monogenic theorem} is different and gives a new perspective on the leading constants in the asymptotic used to estimate ${\rm Avg}(\SS, \Cl(\cdot)[2])$ and ${\rm Avg}(\SS,\Cl^{+}(\cdot)[2])$, distinct from that of either \cite{BhargavaHankeShankar}or \cite{Ashvin}.

\subsection{Leitfaden} \label{subsec: 
main ideas}

The main idea is to adapt the strategy of \cite{BhargavaHankeShankar} to the setting of Wood's parametrisation \cite{MR2763952, MR3187931}. We prove Theorem \ref{main theorem} by taking the ratio of two asymptotic formulas --- the first, a formula for the number of $\SL_n(\Z)$ orbits of pairs of integral symmetric $n \times n$ matrices lying above \footnote{This is understood via the resolvent map defined below.} elements of $\Sigma$ of height at most $X$, and the second, a formula for the number of elements of $\Sigma$ of height at most $X$ --- and letting $X \rightarrow \infty$.

The second formula is known: parametrising isomorphism classes of monogenised $n$-ic orders, $\mathfrak{R}_{1,{\rm order}}^{r_{1},r_{2}}$, in terms of elements of $\bigsqcup_{0 \le b < n} U_{1,b}(\Z)$, we get asymptotics for the number of elements of height at most $X$ in acceptable, and thus also in very large, subfamilies of $\mathfrak{R}_{1,{\rm order}}^{r_{1},r_{2}}$ from the work of \cite{MR4419629}. 

The first formula was previously unknown and is the focus of this paper's work. Wood's parametrisation \cite{MR3187931} allows us to interpret $2$-torsion ideal classes of rings in $\mathfrak{R}_{1,{\rm order}}^{r_1,r_2}$ in terms of $\SL_n(\mathbb{Z})$-orbits on the space $V(\Z)(\det(A) = \factoralt)$, consisting of pairs of integral symmetric $n \times n$ matrices $(A,B)$ subject to the constraint $\det(A) = \factoralt$. This space has a resolvent map to binary forms, $\pi(A,B):= \factoralt\det(Ax-By)$, allowing us to lift the naive height.

The constraint $\det(A) = \factoralt$, which comes from asking for monogenicity in Wood's parametrisation, considerably complicates direct application of Bhargava's geometry of numbers method developed in \cite{MR2183288,MR3272925,MR3275847}. This is because $V(\Z)(\det(A) = \factoralt)$ is a \emph{hypersurface} in the affine space of pairs $\{ (A,B) 
\}$ of symmetric $n \times n$ matrices, obstructing the direct use of Davenport's Lemma, its most basic tool. 

We get around this complication by ``linearising'' the problem, following an idea of \cite{BhargavaHankeShankar}. Denote by $\mathscr{L}_{\mathbb{Z}}$ the collection of $\SL_n(\mathbb{Z})$-equivalence classes of integral symmetric matrices of determinant $\factoralt$. This collection is finite. Counting $\SL_n(\mathbb{Z})$-orbits on $V(\Z)(\det(A) = \factoralt)$ is equivalent to counting, for each $A_0 \in \mathscr{L}_\mathbb{Z}$, integral $\SO_{A_0}(\mathbb{Z})$-orbits on the slice consisting of pairs $(A_0,B)$. Precisely, we count integral orbits in $\bigsqcup_{A_0 \in \mathscr{L}_{\Z},0 \le b < n} \SO_{A_0}(\Z) \backslash V_{A_0,b}^{r_2,\delta}(\R)$, where $V_{A_0,b}^{r_2,\delta}$ denotes the space of pairs of symmetric matrices of the form $(A_{0},B)$ with a resolvent polynomial in $U_{1,b}^{r_{2}}$ and where $\delta$ is a parameter indexing certain archimedean local conditions. 

Despite some technical complexities, Bhargava's geometry of numbers method is capable of handling the count corresponding to $2$-torsion ideal classes in orders. It gives asymptotic bounds for the number of absolutely irreducible orbits of height at most $X$ in subfamilies $\mathcal{V}(\Lambda_{A,b}^\delta)$ of $V_{A_0,b}^{r_2,\delta}(\Z)$ defined by congruence conditions. These bounds are equalities for very large subfamilies, and conditional equalities for acceptable subfamilies, granted the tail estimate \ref{conj: conjectural tail estimate}.

The asymptotic bounds obtained for each $A_0$-slice must be summed over all the elements $A_0 \in \mathscr{L}_\Z$. The evaluation of the leading coefficient of this sum is delicate and relies on understanding non-trivial $2$-adic and archimedean masses at each $A_0$-slice. We find their values by proving ``equidistribution'' of local mass results. In doing so, we introduce a flexible argument allowing us to prove average equidistribution of $2$-adic mass in families defined by local condition modulo $2$ by using the local representation densities at $2$ which appear in the Smith--Minkowski--Siegel mass formula --- these densities were first correctly computed by Conway--Sloane in \cite{MR965484}. The same argument works for non-maximal orders, and also in the even degree and relative cases of the sequels \cite{Siadeven} and \cite{Siadrelative}.  

The proof of Theorem \ref{thm: main N monogenic theorem} follows this pattern, but uses an argument based on the behaviour of Hasse--Witt invariants under orthogonal direct sum to establish the required, much more intricate, equidistribution of local mass results at all primes $p \mid N$. 

The paper is organised as follows. In \S \ref{sec: background and preliminaries}, we set notation and gather background that we will need for the proofs of Theorem \ref{main theorem} and Theorem \ref{thm: main N monogenic theorem}. In \S \ref{sec: parametrisations}, we parametrise the $2$-torsion ideal classes of the rings in $\mathfrak{R}_1^{n}$ in terms of integral $\SL_n(\mathbb{Z})$-orbits on the space of pairs of symmetric $n \times n$ matrices $(A,B)$ subject to $\det(A) = \factoralt$. In \S \ref{sec: reduction theory} we construct fundamental domains for the action of $\SO_{A_0}(\Z)$ on $V_{A_0,b}^{r_2,\delta}(\R)$ for $A_0 \in \mathscr{L}_\Z$ with good boundedness properties. In \S \ref{sec: averaging and cusp}, we cut off the cusp to obtain Theorem \ref{thm: rough count}: a count of absolutely irreducible integral orbits on these slices. In \S \ref{sec: sieves}, we sieve down this count to projective orbits lying above very large and, conditional on the tail estimate \ref{conj: conjectural tail estimate}, acceptable families of binary forms. In \S \ref{sec: product of local volumes} - \S \ref{sec: archimedean distribution}, we establish the equidistribution of local mass results we need. In \S \ref{sec: proof of main theorem}, we prove Theorem \ref{main theorem} and Corollary \ref{main corollary}, and in \S \ref{sec: averages for $N$ monogenised fields} Theorem \ref{thm: main N monogenic theorem}.

\subsection{Previous work}

Our method blends ideas from \cite{MR3782066} and \cite{BhargavaHankeShankar}. We, however, overcome several difficulties which were not present in their setting.

In \cite{MR3782066}, both the geometry of numbers and the evaluation of the adelic volume appearing in the asymptotic formula obtained after its application are more straightforward. In \cite{MR3782066}, one cuspidal part needs to be cut off for each degree, while we need to cut off, in each degree, a growing number of them: one for each $A_0$-slice. In \cite{MR3782066}, evaluating the adelic volume produced from the geometry of numbers reduces to understanding $R^\times / (R^\times)^2$ for $n$-ic rings $R$ over $\Z_p$ and $\R$, and is thus relatively easy to compute. In this paper, we are faced with sums of adelic volumes over genera, the evaluation of each of which requires precise equidistribution results on the relevant local masses. 

In \cite{BhargavaHankeShankar}, the parametrisation of quartic rings in \cite{MR2113024} and class field theory leads to the use of $\SL_3(\Z)$ orbits on pairs of binary quadratic forms (halves-in!). Such parametrisation is not available in higher degrees, and we have to use Wood's parametrisation and, with it, $\SL_n(\Z)$ orbits on pairs of symmetric bilinear forms (halves-out!), whose $2$-adic theory is considerably (and (in)famously!) more complicated. As a result, in addition to the difficulties arising from cutting off a growing number of cuspidal parts in higher degrees, we also have to work much harder to obtain equidistribution of local mass results at the $2$-adic and the archimedean places, even in the cubic case. 

Lastly, as mentioned in \S \ref{subsec: intro N-monogenic}, the work in this paper is critical to \cite{Ashvin}.

\subsection{Acknowledgements} 

The author is grateful to his advisor Arul Shankar for suggesting the project and for his constant encouragement. He also thanks Manjul Bhargava, Jonathan Hanke, Ashvin Swaminathan and Akshay Venkatesh for their helpful comments on earlier versions of this paper. The author would like to thank the anonymous referees for their careful reading and many helpful comments. 

This work formed part of the author's Ph.D. thesis at the University of Toronto, and he warmly thanks the members of his thesis committee --- Florian Herzig, Stephen Kudla, Jacob Tsimerman, Ila Varma, and Melanie Wood --- for a careful reading of an earlier version and for offering many valuable comments.

The author thanks Be\~nat Mencia-Uranga and Simon Lieu of the Cavendish Laboratory for their hospitality. This research was supported by a McCuaig–Throop Bursary and an OGS at the University of Toronto, a Postdoctoral Fellowship from NSERC, the NSF under Grant No. DMS-1926686 at the Institute for Advanced Study, and by Princeton University.

\section{Background and preliminaries}\label{sec: background and preliminaries}

We record a natural correspondence between monogenised rings and monic binary forms. This allows us to define acceptable and very large families of monogenised orders from \S \ref{subsec: statement of main theorem}. We then state the asymptotic formula for the number of elements of bounded height in acceptable families of monogenised rings. Lastly, we recall general facts regarding our heights and genera of symmetric bilinear forms (integer-matrix quadratic forms), which we will use in the coming sections.

\subsection{Monogenised rings in terms of monic polynomials}
For a ring $T$, let $U(T) = {\rm Sym}^{n}(T^{2})$ be the space of binary $n$-ic forms with coefficients in $T$. Let $U_{1} \subset U$ be the subspace of $U$ consisting of elements of $U$ with leading (i.e. $x^{n}$) coefficient equal to $1$. That is $U_{1}$ consists of ``monic'' binary $n$-ic forms. For an element $b \in T$, let $U_{1,b} \subset U_{1}$ denote the subspace of $U_{1}$ consisting of elements of $U_{1}$ with sub-leading (i.e. $x^{n-1}y$) coefficient equal to $b$. The group ${\rm GL}_2$ acts on $U$ via \[\gamma \cdot f(x,y) := \frac{1}{\det \gamma}f((x,y)\cdot \gamma)\] for $\gamma \in {\rm GL}_2$ and $f \in U$. The action of the subgroup $M \subset \GL_{2}$ of lower triangular unipotent matrices \footnote{That is $M(R)$ consists of matrices of the form $\begin{pmatrix} 1 & 0 \\ k & 1\end{pmatrix}$ for $k \in R$.} preserves $U_{1}$ and thus descends to an action of $M$ on $U_{1}$. 

We note that $M(\Z) \setminus U_{1}(\Z) = \bigsqcup_{0 \le b < n} U_{1,b}(\Z)$. For an integer $0 \le r_2 \le \frac{n-1}{2}$, we define $U_{1,b}^{r_2}(\R)$ to be the set of monic polynomials in $U_{1,b}(\R)$ with non-zero discriminant, $r_1 := n-2r_2$ real roots and $r_2$ pairs of conjugate complex roots. For an integer $b$, set $U_{1,b}^{r_2}(\Z) = U_{1,b}(\Z) \cap U_{1,b}^{r_2}(\R)$ and let $U_{1,b}(\Z)^{{\rm irr}}$ (resp. $U_{1,b}^{r_{2}}(\Z)^{{\rm irr}}$) be the subset of $U_{1,b}(\Z)$ (resp. $U_{1,b}^{r_{2}}(\Z)$) consisting of irreducible forms. 

The naive height $H_{\text{naive}}$ of $f = a_0x^{n}+b x^{n-1}y + a_{2}x^{n-2}y^{2}+\cdots + a_{n}y^{n} \in U(\Z)$ or $U(\R)$ is defined as: $$H_{\text{naive}}(f) := \max \left\{|b|,|a_{2}|^{1/2},\ldots,|a_{n}|^{1/n} \right\}.$$ 

We have the following parametrisation of monogenised $n$-ic rings in terms of this orbit data. 

\begin{lem} \label{monogenic rings bijection}
There is a natural bijection between isomorphism classes of monogenised $n$-ic rings and $M(\mathbb{Z})$-orbits on $U_1(\mathbb{Z})$. In particular, \[\mathfrak{R}_1^{n} \cong \bigsqcup_{0 \le b < n} U_{1,b}(\Z),\] and under this bijection $\mathfrak{R}_{1,{\rm order}}^{r_{1},r_{2}} \cong \bigsqcup_{0 \le b < n} U_{1,b}^{r_{2}}(\Z)^{{\rm irr}}$.
\end{lem}
\begin{proof} The map sending $f(x,y) \in U_1(\mathbb{Z})$ to the monogenised $n$-ic ring $\left(\mathbb{Z}[x]/(f(x,1)),x \right)$ descends to a map $\Phi$ from $M(\mathbb{Z}) \backslash U_1(\mathbb{Z})$ to isomorphism classes of monogenised $n$-ic rings. 

Indeed, if $g = \gamma \cdot f$ for $\gamma = \bigl[ \begin{smallmatrix} 1&0\\ m&1 \end{smallmatrix} \bigr] \in M(\mathbb{Z})$, then $g(x,1) = f(x+m,1)$ and the monogenised ring $\left(\mathbb{Z}[x]/(f(x,1)),x \right)$ is isomorphic to the monogenised ring $\left(\mathbb{Z}[x]/(g(x,1)), x \right)$ through $x \mapsto x+m$. 

The inverse map is given by sending a monogenised $n$-ic ring $(R,\alpha)$ to the (binary form associated to the) characteristic polynomial of $(\times \alpha)$ on $R$. 
\end{proof}

Under the identification of Lemma \ref{monogenic rings bijection}, $H_{\textrm{naive}}$ matches the naive height $H_{\text{naive}}$ on $\mathfrak{R}_{1}^{n}$ defined in \S \ref{subsec: statement of main theorem}.

\subsection{Number of elements of bounded height in acceptable and very large families}

For $S_{b} = \prod_{p} S_{b,p} \subset U_{1,b}(\widehat{\Z})$ a collection of local conditions, we denote by $U_{1,b}(S_{b})^{{\rm irr}}$ the preimage of $S_{b}$ in $U_{1,b}(\Z)^{{\rm irr}}$ under the diagonal embedding $U_{1,b}(\Z)^{{\rm irr}} \hookrightarrow \prod_{p} S_{b,p} \subset U_{1,b}(\widehat{\Z})$. For any $X >0$, we write $U_{1,b}^{r_2}(S_{b})^{{\rm irr}}(X)$ for the elements of $U_{1,b}^{r_2}(S_{b})^{{\rm irr}}$ of height at most $X$. 

We have the following asymptotic formula for the size of $U_{1,b}^{r_2}(S_{b})^{{\rm irr}}(X)$ for acceptable (and thus also for very large) families of local conditions $S_{b}$ by \cite{MR4419629}. 

\begin{thm}[{\cite[Theorem 4.1 and Theorem 5.4]{MR4419629}}] \footnote{The authors also obtain a power-saving error term which we will not need.} Fix $ 0\le b < n$ and let $S_b = \prod_{p} S_{b,p}$ be an acceptable collection of local specifications. Then we have \[\left| U_{1,b}^{r_2}(S_{b})^{{\rm irr}}(X) \right| = \Vol(U_{1,b}^{r_{2}}(\R)_{H < X}) \prod_p \Vol(S_{b,p}) + o(X^{\frac{n(n+1)}{2}-1}).\]
\end{thm}

The main term dominates because $\Vol(S_{\infty,H<X})$ grows like $X^{\frac{n(n+1)}{2}-1}$. 

\subsection{Equivalence of height and box-height on the roots} 

We show that the naive height of a monic polynomial is equivalent to the box height of its roots over the complex numbers. This equivalence will be useful in our construction of fundamental regions with certain boundedness properties in \S \ref{sec: reduction theory}. \\

Let $f(x) = x^n + a_1x^{n-1} + \cdots + a_n \in \mathbb{R}[x]$ be monic of degree $n$. We define its naive height as the naive height of its associated binary $n$-ic form,  
$H_{\text{naive}}(f) := \max_i \, \{ |a_i|^{1/i} \} = \max \{ |a_1|, |a_2|^{1/2}, \ldots, |a_n|^{1/n} \}$. Let $\alpha_1, \ldots, \alpha_n$ denote the (not necessarily distinct) roots of $f$ over the complex numbers. We define the box-height on the roots of $f$ as $$H_{\text{roots}}(f) := \max_i \{ |\alpha_i| \} = \max \{ |\alpha_1|, |\alpha_2|, \ldots, |\alpha_n| \}.$$
The naive and box height on the roots are equivalent.

\begin{prop}
Let $f(x) = x^n+a_1x^{n-1}+\ldots+a_n \in \R[x]$ be a monic polynomial of degree $n$. Then \[ H_{\text{roots}}(f) = \Theta(H_{\text{naive}}(f)).\]
\end{prop}
\begin{proof}
We need to show $C_1 H_{\text{roots}}(f) \le H_{\text{naive}}(f) \le C_2 H_{\text{roots}}(f)$ for some constants $C_1$ and $C_2$ depending only on $n$. Each coefficient $a_i$ is the $i$th elementary symmetric polynomial in the roots of $f$ and we apply the triangle inequality to find that $|a_i|^{1/i} \le {n \choose i}^{1/i} H_{\text{roots}}(f)$. This proves the upper bound. The lower bound follows from Lagrange's bound $H_{\text{roots}}(f)/2 \le H_{\text{naive}}(f)$, see \cite[\S 2.1]{MR2996895}.
\end{proof}

We will use the following lemma in the construction of the fundamental domain in \S \ref{sec: reduction theory}. 

\begin{lem} Let $c \in \C$. Then $$H_{\text{roots}}(f(x)) - |c| \le H_{\text{roots}}(f(x+c)) \le  H_{\text{roots}}(f(x)) + |c|.$$
\end{lem}
\begin{proof}
    Immediate from the triangle inequality.
\end{proof}

\subsection{Invariants of symmetric bilinear forms of odd degree} \label{subsec: background on quadratic forms}

We recall background material on symmetric bilinear forms, which we will use in the calculations of \S \ref{sec: proof of main theorem} and \S \ref{sec: averages for $N$ monogenised fields}. The material is derived from \cite{MR12640} and \cite{MR522835}.  

Throughout the paper, quadratic forms will refer to what are variously called twos-in or classically-integral quadratic forms. \footnote{A quadratic form is called classically integral if it has the form $a_{11}x_{1}^2+\cdots+a_{nn}x_n^2 + \sum_{i<j} 2 a_{ij}x_ix_j$. In other words, classically integral quadratic forms over a ring $R$ have matrices with integer entries and are thus the same as symmetric bilinear forms over $R$.} These correspond to symmetric bilinear forms (think of writing them as a matrix), and we will freely call our symmetric bilinear forms quadratic forms to make the connection to the literature on quadratic forms tighter.

\begin{defn}
For a prime $p$ or the archimedean place, any symmetric bilinear form $A \in \Sym_{2}(\Q_{p}^{n})$ can be diagonalized via an $\SL_{n}(\Q_{p})$ change of variables, say $g^{t} A g = {\rm diag}(\alpha_{1},\ldots, \alpha_{n})$. The Hasse--Witt symbol of $A$, denoted by $e_{p}(A)$, is defined as the product of the inner products $(\alpha_{i}, \alpha_{j})_{p}$ for all $i < j$, where $(\cdot,\cdot)_{p}$ is the Hilbert symbol. \footnote{By definition, $(a,b)_{p} = \pm 1$ if and only if the conic $ax^{2}+by^{2}=z^{2}$ has a non-zero solution in $\Q_{p}^{3}$.}
\end{defn}

The Hasse--Witt symbol has the following behaviour under orthogonal direct sums of forms. 

\begin{prop}[{\cite[Chapter 4 Lemma 2.3]{MR522835}}] \label{HasseWittorthsums} Let $q_1$ and $q_2$ be two $n \times n$ symmetric bilinear forms over $\Q_p$. The Hasse--Witt symbol has the following behaviour under orthogonal direct sums: \[e_p(q_1 \perp q_2) = e_p(q_1) e_p(q_2) ( \det(q_1),\det(q_2))_p.\]
\end{prop}

Two integral quadratic forms are said to be in the same genus if they are $\SL_{n}(\Z_{p})$-equivalent for all primes $p$ and $\SL_{n}(\R)$ equivalent. For an integer $N \in \Z$, we let $\mathcal{G}_{\Z,N}$ denote the set of genera of integral quadratic forms with determinant $N$. 

All forms in a genus have the same determinant. Since the number of $\SL_{n}(\Z)$ classes of integral quadratic forms of a fixed determinant is finite, $\mathcal{G}_{\Z,N}$ is also finite.

Genera are classified by the determinant and collections of local Hasse--Witt symbols whose product is equal to $1$.

\begin{thm}[{\cite[Lemmas 1, 2 and 3]{MR12640}}, {\cite[Chapter 9 \S 5]{MR522835}}] \label{genera computation}

 Fix an odd integer $n$. Let $\mathcal{G}_{\Z,N}$ denote the collection of genera of integral quadratic forms of determinant $N$. Then\[\mathcal{G}_{\Z,N} \cong \left\{ \prod_{p} [q_{p}] \times [q_{\infty}]  \in \left(\prod_{p} \frac{\Sym_2(\mathbb{Z}_p^n)}{\SL_n(\mathbb{Z}_p)} \right) \times \frac{\Sym_2(\R^n)}{\SL_n(\R)} \colon \substack{ \forall p\left(\det(q_p)= N \right) \\ \det(q_\infty) = N \\ \left(\prod_{p}e_p(q_p)\right) e_\infty(q_\infty) = 1}  \right\} .\] Specialising to $N = \factor$ and using the fact that for all odd primes $p$, there is a unique $\SL_{n}(\Z_{p})$ equivalence class of $\Z_{p}$ quadratic forms of determinant $\factor$, we get
\[\mathcal{G}_{\Z,{\factor}} \cong \left\{ \Big([q_2],[q_\infty] \Big) \in \frac{\Sym_2(\mathbb{Z}_2^n)}{\SL_n(\mathbb{Z}_2)} \times \frac{\Sym_2(\R^n)}{\SL_n(\R)} \colon \substack{ \det(q_2)= \factor \\ \det(q_\infty) = \factor \\ e_2(q_2)e_\infty(q_\infty) = 1}  \right\} .\]
\end{thm}

\section{Parametrisations} \label{sec: parametrisations}

Let $f \in U_1(\Z)$ with $\Delta(f) \neq 0$ and $R_{f}$ be the monogenised ring $(\Z[x]/(f(x,1)),x)$ associated to $f$ under the bijection of Theorem \ref{monogenic rings bijection}. 
The goal of this section is to write formulas for the sizes of $\Cl(R_{f})[2]$ and $\Cl^{+}(R_{f})[2]$ in terms of the number of $\SL_{n}(\Z)$-orbits of pairs of symmetric integral matrices with resolvent form $f$. This reduces the problem of obtaining an asymptotic formula for $\sum_{\substack{\mathcal{O} \in \SS \\ H_{\text{naive}}(\mathcal{O}) < X}} | \Cl^{(+)}(\mathcal{O})[2] | - | \I_2(\mathcal{O}) |$ to the problem of finding the number of orbits with resolvent of height bounded by $X$. 

The section is organised as follows: 1) for an order $\O$ in a number field $K$, we define the spaces $H(\O)$ and $H^+(\O)$, which are close relatives of the $2$-Selmer group of $K$. These surject onto $\Cl(\O)[2]$ and $\Cl^{+}(\O)[2]$, with fibres of known size; 2) we recall Wood's parametrisation of $2$-torsion ideal classes in rings associated to binary forms with coefficients in a PID (we will need it again in \S \ref{sec: averages for $N$ monogenised fields}, so we consider general binary forms not just monic ones). Specializing to $\Z$ gives $H(R_{f})$ and $H^{+}(R_{f})$ in terms of $\SL_{n}(\Z)$ orbits of pairs of integral symmetric matrices with resolvent $f$; 3) Wood's parametrisation takes a simple form over $\Z_p$ and $\R$, and we record the orbit and stabilizer data for later use.

The notation and material of this section are adapted from \cite[\S 2]{MR3782066}. 

\subsection{The spaces $H(\O)$ and $H^{+}(\O)$}  Let $\mathcal{O}$ be an order in an $S_n$ number field of degree $n$. Let $H(\mathcal{O})$ denote the set of pairs $(I,\delta)$ consisting of a fractional ideal $I$ of $\O$ and an element $\delta \in K^\times$ such that $I^2 \subset (\delta)$, $N(I)^2 = N(\delta)$ and such that the ideal $I$ is projective (i.e. invertible as a fractional ideal or, equivalently, satisfies $I^2=(\delta)$). 

The set $H(\mathcal{O})$ has a natural composition law defined by component-wise multiplication. There is a map from $H(\mathcal{O})$ to the $2$-torsion of the class group of the order $\mathcal{O}$ given by forgetting about the $\delta$ component. This map is surjective since the norm map is multiplicative on invertible ideals, and the fibres only depend on the rank of the unit group of $\O$. 

It is also possible to relate the space $H(\mathcal{O})$ to $2$-torsion in the narrow class group of $\mathcal{O}$. Denote by $H^{+}(\O)$ the subgroup of $H(\O)$ consisting of pairs $(I,\delta)$ with $\delta$ having the property that its image under any real embedding of $K$ is positive. 

We have the following relation between the sizes of $H(\mathcal{O})$ and $H(\mathcal{O}^+)$ and those of $\Cl(\mathcal{O})[2]$ and $\Cl^+(\mathcal{O})[2]$.

\begin{lem}[{\cite[Lemma 2.3 and Lemma 2.4]{MR3782066}}]
Let $\mathcal{O}$ be an order in a number field of degree $n$ and signature $(r_1,r_2)$. Then: 
$$
    \lvert H(\mathcal{O}) \rvert = 2^{r_1+r_2-1}\lvert \Cl(\mathcal{O})[2]|,
$$
$$
    |H^+(\mathcal{O})| = 2^{r_2} |\Cl^+(\mathcal{O})[2]|.
$$

\end{lem}

\subsection{Parametrisation in terms of orbits of pairs of symmetric matrices} 

Fix a ring $T$. Let $V(T) = T^2 \otimes \Sym_2 (T^n)$ be the space of pairs of symmetric $n \times n$ matrices with coefficients in $T$. The group $\SL_n(T)$ acts on $V(T)$ by change of basis $\gamma \cdot (A,B) = (\gamma^t A \gamma, \gamma^t B \gamma)$ for where $\gamma^t$ is the transpose of $\gamma$. There is a natural $\SL_n(T)$ invariant resolvent map from $\pi \colon V(T) \rightarrow U(T)$ defined by \[ \pi(A,B) := (-1)^{\frac{n-1}{2}}\det(Ax-By).\] We write $\Delta(A,B) := \Delta(\pi(A,B))$ for the discriminant of the resolvent form of $(A,B)$. We say that a pair $(A,B) \in V(T)$ is non-degenerate if $\Delta(A,B) \neq 0$. When $T = \Z$ or $\R$, the height of a pair $(A,B)$ is defined to be the height of its resolvent binary form, $H((A,B)) := H(\pi(A,B))$. 

Let $T$ be a PID, and $K$ its field of fractions. Let $f$  be a binary $n$-ic form $f = f_{0}x^{n} + \cdots + f_{n}y^{n} \in U(T)$ with $f_0 \neq 0$ and $\Delta(f) \neq 0$. Define $K_{f} := K[\theta]/f(\theta,1)$. Define the free $T$-submodule $R_{f} = {\rm Span}_{T} \{1, \zeta_{1}, \ldots, \zeta_{n-1} \}$ where $\zeta_{i} := f_{0}\theta^{i}+ \cdots + f_{i-1} \theta \subset K_{f}$. Then $R_{f}$ is multiplicatively closed and thus forms a ring. In addition, $\disc(R_f) = \Delta(f)$. When $T = \Z$, $R_{f}$ is an order in $K_{f}$. We also have the fractional ideals of $R_{f}$ defined by $I_{f}(k) := {\rm Span}_{T} \{1, \theta, \ldots, \theta^{k}, \zeta_{k+1}, \ldots, \zeta_{n-1}\} = I_{f}(1)^{k}$. These are invertible if $f$ is primitive. If $I$ is an ideal given with a basis, its norm is the determinant of the linear map taking the basis of $I$ to the basis $\{1, \zeta_{1}, \ldots, \zeta_{n-1} \}$ of $R_{f}$. If $I$ is not based, the norm is understood up to multiplication by an element of $T^{\times}$. We have $N(I_{f}(k)) = 1/f_{0}^{k}$. 

We now present the parametrisation of $2$-torsion ideal classes in rings associated with binary forms, due to Bhargava \cite{MR2081442} in the case $n=3$ and to Wood \cite{MR2763952, MR3187931} for general $n \ge 3$.

\begin{thm}[{\cite[Theorem 6.3]{MR3187931}}] \label{Wood parametrisation}
Let $n$ be odd and $T$ a PID. Take a non-degenerate primitive binary $n$-ic form $f \in U(T)$. We have a bijection between $\SL_n(T)$-orbits of pairs $(A,B) \in V(T)$ with $\pi(A,B) = f$ and equivalence classes of pairs $(I,\delta)$ where $I \subset K_f$ is a fractional ideal of $R_f$ and $\delta \in K_f^\times$ such that $I^2 \subset (\delta) I_{f}(n-3)$ as ideals and $f_{0}^{n-3}N(I)^2 = N(\delta)$. The classes $(I,\delta)$ and $(I',\delta')$ are equivalent if there exists a $\kappa \in K_f^\times$ with the property that $I = \kappa I'$ and $\delta = \kappa^2 \delta$. Furthermore, the forms $A$ and $B$ have a simple description: They are, respectively, the bilinear forms $\check{\zeta}_{n-1}(xy/\delta)$ and $\check{\zeta}_{n-2}(xy/\delta)$ on $I_{f}(n-3)$ expressed in a common basis for $I$, where we write $ \{ \check{\zeta}_{0}, \check{\zeta}_{1}, \ldots, \check{\zeta}_{n-1} \}$ for the dual basis of $R_{f}$. 
\end{thm}

We can describe the stabilisers in terms of $R_f$. 

\begin{lem}[{\cite[Corollary 2.15]{MR3782066}}]
The stabiliser of $(A,B) \in \pi^{-1}(f)$  in the group $\SL_n(T)$, for $f \in U(T)$ a non-degenerate primitive binary $n$-ic form, corresponds to the norm $1$ elements of the $2$-torsion in the units of $R_f$: $R_f^\times[2]_{N \equiv 1}$.
\end{lem} 

We specialise to the case where $T = \Z$. An orbit $\SL_{n}(\Z) \cdot (A,B)$ is said to be \emph{projective} if $\pi(A,B)$ is primitive and the corresponding pair $(I,\delta)$ has $I$ projective as an $R_{f}$-module. Equivalently, we may simply ask that $I$ be invertible or that $I^{2} = (\delta) I_{f}(n-3)$. 

For an order $\O$ in a number field, we denote by $\I_2(\O)$ the $2$-torsion subgroup of the ideal group $\I(\O)$ of $\mathcal{O}$ --- that is the group of invertible fractional ideals of $\mathcal{O}$ with the property that $I^2 = \mathcal{O}$. If $\mathcal{O}$ is a maximal order, the only element in $\mathcal{I}_2(\mathcal{O})$ is the trivial ideal $\mathcal{O}$. 

Elements of $\I_2(\O)$ are related to reducible pairs in $V(\Z)$ under the parametrisation of Theorem \ref{Wood parametrisation}.

\begin{defn}(\cite[\S 2.4]{MR3782066}) \label{defn: reducible}
A pair $(A,B) \in V(\mathbb{Q})$ is said to be reducible if the quadrics corresponding to $A$ and $B$ have a common rational isotropic subspace of (affine) dimension $(n-1)/2$. \footnote{For instance, if $n=3$, pairs $(A,B)$ with $a_{11} = b_{11} = 0$ are reducible.}
\end{defn} 

Reducible pairs correspond to $(I,\delta)$ where $\delta$ is a square.

\begin{lem}[{\cite[Theorem 2.6]{MR3782066}}] \label{lem: reducible}
Let $(A, B)$ be a projective element $V(\mathbb{Z})$ with an irreducible and non-degenerate monic resolvent form. Let $(I,\delta)$ be the corresponding pair. Then $\delta$ is a square in $(R_f \otimes_{\mathbb{Z}} \mathbb{Q})^\times$ if and only if $(A,B)$ is reducible.
\end{lem}

This gives a parametrisation of $H(R_{f})$ and $H^{+}(R_{f})$ in terms of $\SL_n(\Z)$ orbits on $V(\Z)$.

\begin{lem}[{\cite[Proposition 2.5 and 2.12]{MR3782066}}] \label{integral parametrisation}
Take a non-degenerate irreducible primitive binary $n$-ic form $f \in U(\Z)$. There is a natural bijection between $H(R_{f})$ and projective $\SL_{n}(\Z)$ orbits on $V(\Z) \cap \pi^{-1}(f)$. There is a natural bijection between $H^{+}(R_{f})$ and projective $\SL_{n}(\Z)$ orbits on $V(\Z) \cap \pi^{-1}(f)$ with the property that in the associated pair $(I,\delta)$ we have that the image of $\delta$ under any real embedding of $K_{f}$ is positive. Furthermore, there is a natural bijection between $\I_2(\mathcal{O})$ and the set of projective reducible $\SL_n(\mathbb{Z})$-orbits on $V(\mathbb{Z}) \cap \pi^{-1}(f)$. 
\end{lem}

\subsection{Orbits over fields and $\mathbb{Z}_p$} Over a field or $\mathbb{Z}_p$ the parametrisation in Theorem \ref{Wood parametrisation} takes a particularly simple form. 

\begin{lem}[{\cite[Corollary 2.15]{MR3782066}}] \label{arithmetic rings}
Let $f$ be a separable non-degenerate primitive binary $n$-ic form with coefficients in $T$, for $T$ a field or $\mathbb{Z}_p$. The projective $\SL_n(T)$-orbits of $V(T)$ with resolvent $f$ are in bijection with the elements of $R_f^\times/(R_f^\times)^2$ of norm equal $1 \in T^\times/(T^\times)^2$, which we write $(R_f^\times/(R_f^\times)^2)_{N \equiv 1}$.
\end{lem}

We record some information about the real orbits in the following example.

\begin{exmp} \cite[Example 2.17]{MR3782066} \label{archimedean example}
For each $f \in U^{r_2}(\R)$, we have $R_{f} = K_{f} \cong \R[x]/(f(x)) \cong \R^{r_1} \times \C^{r_2}$. We have that $(R_f^\times/(R_f^\times)^2)_{N \equiv 1}$, and thus the set of $\SL_{n}(\R)$ orbits with resolvent $f$, is in natural bijection with $\mathcal{T}(r_{2})$, the set of tuples in $\{\pm 1\}^{r_1} \times \{1 \}^{r_2}$ with an even number of $-1$ entries. The set $\mathcal{T}(r_2)$ has cardinality $2^{r_{1}-1}$. The stabilisers correspond to $((\R^{\times})^{r_1} \times (\C^{\times})^{r_2})[2]_{N \equiv 1}$ and thus all have size $\sigma(r_{2}) := 2^{r_{1}+r_{2}-1}$. 
\end{exmp}

\subsection{Orbits and self-adjoint operators}   

We will find the following characterisation of $\SO_A$-orbits on $V_A \bigcap \pi^{-1}(f)$ for monic $f$ useful in \S \ref{sec: archimedean distribution}.

\begin{lem}\label{self-adjoint operators}
Let $A \in \mathscr{L}_{\mathbb{Z}}$. Let $\W_A$ be the bilinear space of rank $n$ with associated Gram matrix $A$ \footnote{$\W_{A}$ is defined as the space $\R^{n}$ with the bilinear form given in standard coordinates by the matrix $A$.}. Then, given $f \in U_1$, there is a natural correspondence between $\SO_A$-orbits on $V_A \bigcap \pi^{-1}(f)$ and $\SO_{\W_A}$-conjugacy classes of self-adjoint operators on $\W_A$ with characteristic polynomial $f$ .
\end{lem}
\begin{proof}
We note that for each $A \in \mathscr{L}_{\mathbb{Z}}$ we can represent the $\SO_A$ orbits on $V_A \cap \pi^{-1}(f)$, in terms of the orbit data $\SO_{\W_A}$ orbits on $\left\{M \in \Mat_n \, \vline \substack{AM = M^t A \\ f = \det(Ix-M)}\right\}$ where $\SO_{\W_A}$ acts on the set of $M$ via $\gamma \cdot M = \gamma M \gamma^{-1}$. The latter space consists of self-adjoint operators with characteristic polynomial $f$. Indeed, the maps $(A, B) \mapsto A^{-1}B$ and $M \mapsto (A, AM)$ are natural and give isomorphisms between these sets of orbits. 
\end{proof}

\subsection{The subspaces $V_{A,b}^{r_2}(\R)$ and $V_{A,b}^{r_2,\delta}(\R)$} \label{subsec: decorations on V}

Fix $A$ a symmetric bilinear form in $\mathscr{L}_{\mathbb{Z}}$ and $0 \le b < n$. Let $V_A \subset V$ denote the subspace of pairs $(A, B)$, where $B$ is arbitrary. Since the first component of the pairs in $V_{A}$ is fixed to be $A$, we will talk about elements $B \in V_{A}$. The resolvent map takes $V_A$ to $U_1$. Let $V_{A,b}$ denote the inverse image under the resolvent map of $U_{1,b}$. Plainly, $V_{A,b}$ is an affine subspace of $V_A$ of dimension $n(n+1)/2-1$. The algebraic group $\SO_{A}$ acts on $V_{A,b}$ and has the same stabilisers as $\SL_n$ acting on $V$. 

For $0 \le r_2 \le (n-1)/2$, write $V_{A,b}^{r_2}(\R)$ for the elements of $V_{A,b}(\R)$ whose resolvent lie in $U_{1,b}^{r_2}(\R)$. In particular, elements $v \in V_{A,b}^{r_2}(\R)$ have resolvents with no repeated roots because, by definition, elements of $U_{1,b}^{r_2}(\R)$ have non-zero discriminant.

Let $n = r_1 + 2r_2$. Denote by $\mathcal{T}(r_2)$ the set of tuples in $\{\pm 1\}^{r_1} \times \{1 \}^{r_2}$ with an even number of $-1$ entries. For $f \in U_{1,b}^{r_{2}}(\R)$, there is a canonical identification of $( R_{f}^{\times}/ (R_{f}^{\times})^{2})_{N \equiv 1}$ with $\mathcal{T}(r_2)$ by ordering the real roots of $f$ in strictly increasing order by Example \ref{archimedean example}. Since for each $f \in U_{1,b}^{r_{2}}(\R)$, the set $( R_{f}^{\times}/ (R_{f}^{\times})^{2})_{N \equiv 1}$ is identified with $\pi^{-1}(f)$ via \ref{arithmetic rings}, we get a map $V_{A,b}^{r_{2}} \rightarrow U_{1,b}^{r_{2}} \times \mathcal{T}(r_2)$. For $\delta \in \mathcal{T}(r_2)$, we define $V_{A,b}^{r_2,\delta}(\R)$ as the inverse image of $U_{1,b}^{r_{2}} \times \{ \delta \}$ under this map.

\subsection{Absolutely irreducible elements in $V(\Z)$ and definition of $N(L;X)$}
\label{subsec: absolutely irreducible elements and N(L;X)}

An element $v \in V(\mathbb{Z})$ is said to be absolutely irreducible if: 1) the resolvent of $v$ corresponds to an order in a number field; and, 2) $v$ is not reducible in the sense of Lemma \ref{lem: reducible} (thus does not correspond to a $2$-torsion element in the \emph{ideal group}, $\mathcal{I}_2(\O)$, under the identification of Lemma \ref{integral parametrisation}). \footnote{We could also ask for that order to be in an $S_n$-field, and everything in the rest of the paper still holds.}An element that is not absolutely irreducible is said to be \textit{bad}. Note that being ``bad'' is a property of the $\Q$-orbit of $v$.

Here is a sufficient criterion to determine if an element is ``bad''. 

\begin{thm}[{Criterion for badness, \cite[Theorem 2.6]{MR3782066}}]\label{thm: criterion for badness}
Let $(A, B) \in V(\mathbb{\Q})$ with $A = (a_{ij})$ and $B = (b_{ij})$ be such that all the variables in one of the following two sets vanish.
\begin{enumerate}
\item {\bf \emph{The squares:}} $\left\{a_{ij},b_{ij} | 1 \le i,j \le \frac{n-1}{2} \right\}$. 
\item {\bf \emph{The rectangles:}} $\left\{a_{ij},b_{ij} | 1 \le i \le k,  1 \le j \le n-k\right\}$ for some $1 \le k \le n-1$.
\end{enumerate}
Then $(A, B)$ is ``bad''. 
\end{thm}

Vanishing of the variables in the two different subsets in Theorem \ref{thm: criterion for badness} implies that $(A, B)$ is ``bad'' for different reasons. Vanishing of the variables in set 1) of Theorem \ref{thm: criterion for badness} implies that $(A, B)$ maps to a $2$-torsion element in the ideal group under the parametrisation of Theorem \ref{Wood parametrisation}, while vanishing of the variables in set 2) implies that $(A, B)$ has a resolvent with discriminant $0$.  

For any $L \subset V(\mathbb{Q})$, let $L^{{\rm irr}}$ denote the set of absolutely irreducible elements. 

\begin{defn}
Let $L \subset V_{A,b}^{r_2,\delta}(\mathbb{Z}) := V_{A,b}^{r_2,\delta} (\mathbb{\mathbb{R}}) \cap V_{A,b}(\mathbb{Z})$ be an $\SO_A(\mathbb{Z})$-invariant set. We define $N(L; X)$ to be the number of absolutely irreducible $\SO_A(\mathbb{Z})$-orbits on $L$ that have height bounded by $X$. 
\end{defn}

In particular, $N(L; X)$ counts a subset which is closely related to the non-trivial elements $H(\O)$ for monogenised orders $\O$ of height at most $X$: it corresponds to pairs $(I,\delta)$ where $I$ is not necessarily invertible and $\delta$ is not a square. In the same sense, subsets of $H^+(\O)$ are close to the being counted by $N(L; X)$ for $L \subset V_{A,b}^{r_2,\delta_{\gg0}}$ for $\delta_{\gg 0} := (1 1 \cdots 1 1)$. The asymptotic formulas for $N(L;X)$ obtained in the next sections can be sieved down to only count projective orbits, and thus elements in the $H(\O)$ and $H^+(\O)$ in very large families. We do so in \S \ref{sec: sieves}.

\section{Reduction theory} \label{sec: reduction theory}

Let $A \in \mathscr{L}_{\Z}$, $0 \le b < n$, and $0 \le r_2 < (n-1)/2$. In this section, we construct a fundamental domain for the action of $\SO_{A}(\Z)$ on $V_{A,b}^{r_2,\delta}(\R)$. We build it by taking the product of a fundamental domain for the action of $\SO_{A}(\Z)$ on $\SO_{A}(\R)$, which we call $\FF_{A}$ together with a fundamental domain for the action of $\SO_{A}(\R)$ on $V_{A,b}^{r_2,\delta}(\R)$ which we call $R_{A,b}^{r_{2},\delta}$. Lastly, for an $\SO_{A}(\Z)$-invariant set $L \subset V_{A,b}^{r_{2},\delta}$, we write $N(L; X)$ in terms of $\FF_{A} \cdot R_{A,b}^{r_{2},\delta}$.

\subsection{Fundamental domain for the action of $\SO_{A}(\R)$ on $V_{A,b}^{r_2,\delta}(\R)$}

We build a finite cover of the fundamental domain for the action of $\SO_A(\mathbb{Z})$ on $V_{A,b}^{r_2,\delta}(\mathbb{R})$ with a good boundedness property. Recall that the height of an element $(A, B) \in V_{A,b}^{r_2,\delta}(\R)$ is defined as the naive height of its resolvent form. For any subset $R \subset  V_{A,b}^{r_2,\delta}(\R)$, we write $R(X)$ for the subset of elements of $R$ of height at most $X$. 

\begin{prop} \label{prop: good fundamental domain}
There exists a (possibly empty) fundamental set $R_{A,b}^{r_2,\delta}$ for the action of $\SO_A(\mathbb{R})$ on $V_{A,b}^{r_2,\delta}(\mathbb{R})$ with the following properties:
\begin{enumerate}[(I)]
\item the set $R_{A,b}^{r_2,\delta}$ is semi-algebraic;
\item if $R_{A,b}^{r_2,\delta}(X)$ denotes the set of elements of height at most $X$, then the coefficients of elements $B \in R_{A,b}^{r_2,\delta}(X)$ are bounded by $O(X)$ where the implied constant does not depend on $B$.
\end{enumerate} 
\end{prop}
\begin{proof}
Many constructions are possible. 

A short one proceeds as follows. Remark that the output of the proposition is analogous to the output of \cite[\S 9.1]{MR3156850} save that there $A$ is taken to be totally split over $\Z$ and $b=0$. 

To remove the condition on $A$, use Lemma \ref{self-adjoint operators} to interpret $\SO_A(\mathbb{R})$-orbits in $V_A(\mathbb{R}) \bigcap \pi^{-1}(f)$ in terms of $\SO_{W_A}$-orbits on self-adjoint operators with characteristic polynomial $f$ on the bilinear space $W_A$ of rank $n$ with Gram matrix $A$. 
The construction of \textit{loc. cit.} with the obvious modifications then proves the existence of an $R_{A,0}^{r_2,\delta}$ satisfying properties (\textit{I}) and (\textit{II}). To deal with $b \neq 0$, note that the ``translation map'' $B \mapsto B - \frac{b}{n} A$ is an affine map from $V_{A,0}^{r_2,\delta}$ to $V_{A,b}^{r_2,\delta}$ which descends under $\pi$ to the map $f(x) \mapsto f(x+\frac{b}{n})$ from $U_{1,0}^{r_2} \rightarrow U_{1,b}^{r_2}$. To complete the proof, define $R_{A,b}^{r_2,\delta} := R_{A,0}^{r_2,\delta} - \frac{b}{n}A$, which is easily seen to be a semi-algebraic fundamental set for the action of $\SO_A(\mathbb{R})$ on $V_{A,b}^{r_2,\delta}(\mathbb{R})$, which is non-empty if and only if $R_{A,0}^{r_2,\delta}$ is. To see that property (\textit{II}) holds for $R_{A,b}^{r_2,\delta}$, we first note that our height has the following property $[H_{\text{naive}}(f(x+c)) \le X \Rightarrow H_{\text{naive}}(f) = O(X)]$ where the implied constant depends only on $c$ and not on $f$ (this follows from $H_{\text{roots}}(f)/2 \le H(f) \le cH_{\text{roots}}(f)$ and $H_{\text{roots}}(f(x+c)) \le H_{\text{roots}}(f(x))+c$). This implies that under the inverse of our ``translation'' map, $R_{A,b}^{r_2,\delta}(X)$ is mapped into $R_{A,0}^{r_2,\delta}(O(X))$ with the implied constant depending only on $b$. The coefficients of elements of $R_{A,0}^{r_2,\delta}(O(X))$ are bounded by $O(X)$. Translating back, we then see that the same is true for coefficients of elements of $R_{A,b}^{r_2,\delta}(O(X))$, completing the proof.
\end{proof}

\begin{rem} \label{homogeneously expanding}
By the proof, we also see that if $R_{A,b}^{r_2,\delta}$ is non-empty, it contains the translate of a region of $V_{A,0}$ whose coefficients are growing homogeneously in $X$. This point will be important in \S \ref{sec: averaging and cusp}.  
\end{rem}

It will be convenient to record whether $V_{A,b}^{r_2,\delta}(\mathbb{R})$ is empty or not with an indicator symbol.  \footnote{Equivalently whether $R_{A,b}^{r_1,\delta}$ is empty or not.}  

\begin{defn}
We define the indicator function
\[\chi_{A,b}^{r_2} (\delta):=
\begin{cases}
1 &\text{if } V_{A,b}^{r_2,\delta}(\mathbb{R}) \,\, (\text{or equivalently } R_{A,b}^{r_1,\delta}) \neq \emptyset\\
0 &\mbox{otherwise }
\end{cases}.\] 
\end{defn}

Note that $\chi_{A,b}^{r_2}(\cdot)$ is independent of $b$, so we drop it from the notation. Additionally, as $r_2$ will be clear from context, we will write $\chi_A(\cdot)$ from now on. We will use these indicators in \S \ref{sec: product of local volumes} to define the archimedean mass and organise the computations of \S \ref{sec: proof of main theorem}.

\subsection{Fundamental domain for the action of $\SO_A(\Z)$ on $V_{A,b}^{r_2,\delta}(\R)$}

We build a cover of a fundamental domain for the action of $\SO_A(\mathbb{Z})$ on $V_{A,b}^{r_2,\delta}(\mathbb{R})$. 

\begin{prop}
Let $\mathcal{F}_A$ be a fundamental domain for the action of $\SO_A(\mathbb{Z})$ on $\SO_A(\mathbb{R})$ and $\sigma(r_2) = 2^{r_1+r_2-1}$. Then
\begin{enumerate}
\item If $\chi_A(\delta) = 1$, $\mathcal{F}_A \cdot R_{A,b}^{r_2,\delta}$ is an $\sigma(r_2)$-fold cover of a fundamental domain for the action of $\SO_A(\mathbb{Z})$ on $V_{A,b}^{r_2,\delta}(\mathbb{R})$, where we regard $\mathcal{F}_A \cdot R_{A,b}^{r_2,\delta}$ as a multiset.
\item If $\chi_A(\delta) = 0$, then $\emptyset$ is a fundamental domain.
\end{enumerate} 
\end{prop}
\begin{proof}
The stabiliser in $\SO_A(\mathbb{R})$ of an element $B \in V_{A,b}^{r_2,\delta}(\mathbb{R})$ is the stabiliser in $\SL_n(\mathbb{R})$ of $(A,B)$. This stabiliser has size $\sigma(r_2) = 2^{r_1+r_2-1}$ by Example \ref{archimedean example}.
\end{proof}

Now, $\mathcal{F}_A \cdot R_{A,b}^{r_2,\delta}$ is a $\sigma(r_2)$-fold cover of a fundamental domain for the action of $\SO_A(\mathbb{Z})$ on $V_{A,b}^{r_2,\delta}(\mathbb{R})$ and, for $X>0$, recall that we had set $\mathcal{F}_A \cdot R_{A,b}^{r_2,\delta}(X) := \{v \in \mathcal{F}_A \cdot R_{A,b}^{r_2,\delta} \colon H(v)<X\}$. For $L \subset V_{A,b}^{r_2,\delta}(\mathbb{Z}) := V_{A,b}^{r_2,\delta} (\mathbb{\mathbb{R}}) \cap V_{A,b}(\mathbb{Z})$ an $\SO_A(\mathbb{Z})$-invariant set, we therefore have \[N(L;X) = \frac{1}{\sigma(r_2)} \#(\mathcal{F}_A \cdot R_{A,b}^{r_2,\delta}(X) \cap L^{{\rm irr }}).\]
In \S \ref{sec: averaging and cusp} and \ref{sec: sieves}, we obtain volume formulas for $N(L_{A,b} \cap V_{A,b}^{r_2,\delta}(\mathbb{Z}); X)$ for $L_{A,b}$ the restriction of lattice in $V_A(\Z)$ to $V_{A,b}(\Z)$.

\section{Averaging and cutting off the cusp}\label{sec: averaging and cusp}

In this section, we find an asymptotic formula for $N(L_{A,b} \cap V_{A,b}^{r_2,\delta}(\mathbb{Z}); X)$, for $L_{A,b}$ the restriction of lattice in $V_A(\Z)$ to $V_{A,b}(\Z)$, in terms of the volume of $\mathcal{F}_A \cdot R_{A,b}^{r_2,\delta}(X)$ by carrying out the averaging and cutting off the cusp parts of the geometry of numbers method.

\begin{thm}\label{thm: rough count}
For $A \in \mathscr{L}_{\mathbb{Z}}$ and $L_{A,b} \subset V_{A,b}(\Z)$ the restriction of a lattice in $V_{A}(\Z)$ to $V_{A,b}(\Z)$, we have \[N(L_{A,b} \cap V_{A,b}^{r_2,\delta}(\mathbb{Z});X) = \frac{1}{\sigma(r_2)}\frac{\Vol (\mathcal{F}_A \cdot R_{A,b}^{r_2,\delta}(X))}{\mathrm{co}\Vol(L_{A,b})} +  o(X^{n(n+1)/2-1}),\] where $\Vol$ is the Euclidean measure on $V_{A,b}(\R)$ normalised so that $V_{A,b}(\Z)$ has covolume $1$. \footnote{This theorem also holds if we assume that $\det A \neq \pm 1$ (but non-zero). We briefly mention the changes to the proof. In this case, by Witt's decomposition theorem, we may take instead of $A_{pq}$, a matrix of that has the same form as $A_{pq}$ over $\Q$ but where the identity block is replaced by an anisotropic quadratic form over $\mathbb{Q}$, in diagonal form, and of the same determinant as $A$. The rest of the arguments of the section then go through \emph{mutatis mutandis} to give the same theorem.}
\end{thm}

The essential difficulty in writing $N(L_{A,b} \cap V_{A,b}^{r_2,\delta}(\mathbb{Z});X)$ in terms of the volume of the region $\mathcal{F}_A \cdot R_{A,b}^{r_2,\delta}(X)$ is that when $A$ is isotropic over $\Q$, $\mathcal{F_A}$ has cusps going to infinity. This creates long thin regions in $\mathcal{F}_A \cdot R_{A,b}^{r_2,\delta}(X)$ where any approximation of the number lattice points in terms of the volume gets progressively worse as $X \rightarrow \infty$. \footnote{Thinking of a long thin rectangle around an axis in $\R^d$ gives a good mental image of this phenomenon.} Bhargava's averaging technique \cite{MR2183288} overcomes this obstacle by averaging the count of lattice points in $\mathcal{F}_A \cdot R_{A,b}^{r_2,\delta}(X)$ over many such translates $\mathcal{F}_A h \cdot R_{A,b}^{r_2,\delta}(X)$ of the fundamental domain $\FF_A$ by elements $h$ in some compact subset $G_0$ of $\SO_A(\R)$. A Fubini-type theorem allows us to switch the domain of integration in this average from $G_0$ to $\FF_A$, giving 
\begin{equation*}
    N(L_{A,b} \cap V_{A,b}^{r_2,\delta}(\mathbb{Z}); X) = \frac{1}{\sigma(r_{2})\Vol(G_{0})} \int_{h \in \FF_{A}} \#(hG_{0}R_{A,b}^{r_2,\delta}(X) \cap \mathcal{L}^{{\rm irr }}) \,\, dh,
\end{equation*}
This integral consists of two parts: a main body part where the integrand $\#(hG_{0}R_{A,b}^{r_2,\delta}(X) \cap \mathcal{L}^{{\rm irr }})$ is well approximated by the volume of $hG_{0}R_{A,b}^{r_2,\delta(X)}$ and a cuspidal part where the lattice points in $hG_{0}R_{A,b}^{r_2,\delta}(X)$ cluster onto coordinate hyperplanes in $V_{A,b}$. 

The goal of ``cutting off the cusp'' is to show that the cuspidal part of the integral is negligible in the sense that it is $o(X^{n(n+1)/2-1})$. This is achieved by a detailed analysis of the error term in Davenport's Lemma applied to $\#hG_{0}R_{A,b}^{r_2,\delta}(X) \cap \mathcal{L}^{\text{irr}}$ as $h$ travels up the cusps of $\FF_A$ until it hits the ``deep'' regions of the cusps where $hG_{0}R_{A,b}^{r_2,\delta}(X) \cap \mathcal{L}^{\text{irr}} = \emptyset$.

The proof of Theorem \ref{thm: rough count} is organised as follows. We begin by choosing fundamental domains that will be convenient for calculations. We then reduce the asymptotic of Proposition \ref{thm: rough count} to Lemma \ref{cusp cutting lemma}. Lastly, we prove Lemma \ref{cusp cutting lemma} by a combinatorial induction argument inspired by \cite{MR3719247}. The key technical result in the proof is the cusp cutting estimate Lemma \ref{cusp cutting lemma} which applies to all $A \in \mathscr{L}_{\Z}$ irrespective of the degree of anisotropy of $A$ over $\Q$. 
The notation and setup are a combination of that of \cite{MR3719247} and \cite{BhargavaHankeShankar}.

\subsection{Choice of $\FF_{A}$ and Haar measure on $\SO_{A}(\R)$} \label{subsec: choice of FA and haar measure on SOA}

Fix $A,b,r_{2},\delta$ and let $L_{A,b}$ be the restriction of a lattice $L \subset V_A(\Z)$ to $V_{A,b}(\Z)$.

\subsubsection{Choosing $\mathcal{F}_A$}
If $A$ is anisotropic over $\mathbb{Q}$, there is a compact fundamental domain for the action of $\SO_A(\mathbb{Z})$ on $\SO_A(\mathbb{R})$. Choose such a compact fundamental domain $\FF_{A}$ and choose any Haar measure $df$ on $\SO_{A}(\R)$. \\

If $A$ is isotropic over $\mathbb{Q}$, there exists an element $g_A \in \SL_n(\Q)$ and a unique pair $p,q$ such that $n = p+q$ and $\frac{n-1}{2} \equiv q \pmod 2$ satisfying $g_A^tAg_A = A_{pq}$ where \begin{equation*} 
A_{pq} := 
\left(\begin{array}{ccccccc}
& & & & & & 1 \\ 
& & & & & \iddots & \\ 
& & & & 1 & & \\ 
& & & \pm I_{|p-q|} & & & \\  
& & 1 & & & &  \\ 
& \iddots & & & & & \\ 
1 & & & & & &  
\end{array}\right),
\end{equation*} 
where the sign of $p-q$ matches the $\pm$ on the identity block. Define $m$ as the minimum of $p$ and $q$, $m = \min\{p,q\}$. 

If $K = \mathbb{R}$ or $\mathbb{Q}$, we define the maps $\sigma_V \colon V_{A,b}^{r_2,\delta} \rightarrow V_{A_{pq},b}^{r_2,\delta}$ and $\sigma_A \colon \SO_A(K) \rightarrow \SO_{A_{pq}}(K)$ by $\sigma_V(A,B) = (A_{pq},g_A^tBg_A)$ and $ \sigma_A(h) = g_A^{-1} h g_A$. Note that $H(A,B) = H(\sigma_V(A,B))$, that $\sigma_V(h \cdot v) = \sigma_A(h)\cdot \sigma_V(v)$, and that $\sigma_{V}$ is measure preserving if both $V_{A,b}^{r_2,\delta}$ and $V_{A_{pq},b}^{r_2,\delta}$ are equipped with their standard Euclidean measures. 

Let $\mathcal{L} \subset V_{A_{pq},b}^{r_2,\delta}(\mathbb{R})$ denote $\mathcal{L} := \sigma_{V}(L_{A,b} \cap  V_{A_,b}^{r_2,\delta}(\mathbb{R})) \subset \sigma_V(V_{A,b}^{r_2,\delta}(\mathbb{Z}))$ and $\mathcal{L}^{{\rm irr}} := \sigma_{V}(L_{A,b}^{{\rm irr}} \cap  V_{A,b}^{r_2,\delta}(\mathbb{R})))$. Because being ``bad'' is a property of $\Q$-orbits, Theorem \ref{thm: criterion for badness} gives a criterion on elements of $\mathcal{L}$ to be ``bad'' --- that is to say to come from a ``bad'' element of $L_{A,b}$ and thus to be contained in $\mathcal{L} \setminus \mathcal{L}^{{\rm irr}}$. 

Denote by $\Gamma \subset \SO_{A_{pq}}(\mathbb{R})$ the subgroup $\sigma_A(\SO_{A}(\mathbb{Z}))$. It is commensurable with $\SO_{A_{pq}}(\mathbb{Z})$ and therefore by \cite[Example 2.5]{MR147566} there exists a fundamental domain $\mathcal{F}_{A_{pq}}$ for the action of $\Gamma$ on $\SO_{A_{pq}}(\mathbb{R})$ which is contained in a finite union of $\SO_{A_{pq}}(\mathbb{Q})$ translates of a Siegel domain $\mathcal{S}$, say $\cup_\ell g_\ell \mathcal{S}$ for $g_\ell \in \SO_{A_{pq}}(\mathbb{Q})$.

We choose $\FF_{A} := \sigma_{A}^{-1}(\FF_{A_{pq}})$ as our fundamental domain. \\

With this notation and choices we have $$N(L_{A,b} \cap V_{A,b}^{r_2,\delta}(\mathbb{Z});X) = \frac{1}{\sigma(r_2)} \#(\mathcal{F}_{A_{pq}} \cdot \sigma_{V}(R_{A,b}^{r_2,\delta})(X) \cap \mathcal{L}^{{\rm irr }}).$$ 

\subsubsection{Choosing a Haar measure on $\SO_A(\R)$}

The choice of $A_{pq}$ above is not disinterested: it allows us to use a convenient Iwasawa decomposition. Let $\SO_{A_{pq}}(\R) = N'T'K'$ be the Iwasawa decomposition. We choose as our Siegel domain $\mathcal{S}$ the product $NTK$ where we choose $K$ to be compact, $N$ a bounded subset of the group of lower triangular matrices with $1$ on the diagonal and $T$ to be \[T := \left\{\begin{pmatrix} 
t_1^{-1} & & & & & & \\ 
& \ddots & & & & & \\ 
& & t_m^{-1} & & & & \\  
& & & I_{|p-q|} & & & \\ 
& & & & t_m & & \\ 
& & & & & \ddots & \\ 
& & & & & & t_1 \\  
\end{pmatrix} \colon t_1/t_2 > c, \ldots, t_{m-1}/t_m > c, t_m >c \right\}\] for some constant $c>0$. This can be found in many sources; see for instance, \cite[Example 2.5]{MR0148666}, \cite{MR1727346} or \cite{MR3452758}. Define $s_{i} = t_{i}/t_{i+1}$ for all $1 \le i \le m-1$ and $s_{m} = t_{m}$. \\
 
We choose a convenient Haar measure on $G = \SO_{A_{pq}}$. Let $dn$ denote a Haar measure on the unipotent group $N$, and $dk$ denote a Haar measure on the compact group $K$. For every $1 \le i \le m$ we write $d^{\times}t_i = \frac{dt_i}{t_i}$ and $d^{\times}s_i = \frac{ds_i}{s_i}$. Furthermore, we write $dt =\prod_{i=1}^m dt_i$, $d^{\times}t = \prod_{i=1}^m d^{\times}t_i$ and $ds = \prod_{i=1}^{m}ds_i$, $d^{\times}s = \prod_{i=1}^{m} d^{\times}s_i$. Changing variables between the $t$-coordinates and the $s$-coordinates gives us $dt = \left( \prod_{i=1}^{m} s_i^{i-1} \right) ds.$ We thus find 
\[d^\times t = \frac{1}{t_1 \cdots t_m} dt = \frac{1}{\prod_{i=1}^m s_i^i}\left( \prod_{i=1}^{m} s_i^{i-1} \right) ds = \frac{1}{s_1 \cdots s_m} ds = d^\times s. \]
We choose the Haar measure $dg$ on $G$ given in $NTK$-coordinates by 
\[dg = e^{-2\rho(H)} du \, d^{\times}t \, dk = \prod_{i=1}^{m} t_i^{2i-p-q} du \, d^{\times}t \, dk = \prod_{i=1}^{m} s_i^{i(i+1-p-q)} du \, d^{\times}s \, dk =  \prod_{i=1}^{m} s_i^{i(i+1-n)} du \, d^{\times}s \, dk .\]

\subsection{Coordinates on $V_{A_{pq}}$ and the partial order $(\mathcal{E},\prec)$}

The matrix entries $\{b_{ij}\}$ for $1 \le i \le j \le n$ and $(i,j) \neq ((n-1)/2,(n-1)/2)$ form a set of coordinates on $V_{A_{pq},b}$ which we denote $\mathcal{E}$. \footnote{We choose the notation $\mathcal{E}$ for ``entries''! The condition that the resolvent binary form has $b$ as $x^{n-1}y$ coefficient translates to a linear equation on $V_{A_{pq}}$ which always involves $b_{(n-1)/2\,\,(n-1)/2}$.} We will define a partial order on the $n(n+1)/2-1$ coordinates $\{b_{ij}\}$ in $\mathcal{E}$, recording scaling of the different entries of the matrix $B$ under the torus action. This scaling is the weight. 

\begin{defn}
The weight $w(b_{ij})$ of an element $b_{ij} \in \mathcal{E}$ is the factor by which $b_{ij}$ scales under the action of $(t_1^{-1},\ldots,t_m^{-1},1, \ldots,1,t_m,\ldots,t_1) \in T$. 
\end{defn}

The weights of the elements of $\mathcal{E}$ in the $t$ and in the $s$-coordinates on $T$ are as follows:
\begin{enumerate}[1)]
\item $w(b_{11}) = t_1^{-2} = s_1^{-2}\cdots s_m^{-2}$;
\item $w(b_{ij}) = t_i^{-1} t_j^{-1} = s_i^{-1} \cdots s_{j-1}^{-1} s_j^{-2} \cdots s_n^{-2}$ if $i \le m$ and $j \le m$;
\item $w(b_{ij}) = t_i^{-1} = s_i^{-1} \cdots s_n^{-1}$ if $i \le m$ and $m+1 \le j \le m+|p-q|$;  
\item $w(b_{ij}) = t_i^{-1} t_{n-j+1} = s_i^{-1} \cdots s_{n-j}^{-1}$ if $i \le m$ and $m+|p-q|+1 \le j \le n$;
\item $w(b_{ij}) = 1 $ if $m+1 \le i \le m+|p-q|$ and $m+1 \le j \le m+|p-q|$; 
\item $w(b_{ij}) = t_{n-j+1} = s_{n-j+1} \cdots s_n$ if $m+1 \le i \le m+|p-q|$ and $m+|p-q|+1 \le j \le n$;
\item $w(b_{ij}) = t_{n-i+1} t_{n-j+1} = s_{n-i+1} \cdots s_{n-j} s_{n-j+1}^2 \cdots s_n^2$ if $m+|p-q|+1 \le i \le n$ and $m+|p-q|+1 \le j \le n$.
\end{enumerate}

Note that the product of the weights of all the coordinates is $1$. 

The partial order on $\mathcal{E}$ is defined as follows. 

\begin{defn} Let $b,b' \in \mathcal{E}$. We say that $b \prec b'$ if in the expression for $w(b)$ in the $s$-coordinates, the exponents of the variables $s_1,\cdots,s_m$ are all smaller than or equal to the corresponding exponents appearing in the expression for $w(b')$. 
\end{defn}

The relation $\prec$ defines a partial order on $\mathcal{E}$. \footnote{Indeed, $\prec$ is the product order on the tuple of $s$-exponents in the weight.} To give an example, $b_{11} \prec b_{m+1\,m+1}$ because $w(b_{11}) = s_1^{-2} \cdots s_m^{-2}$ while $w(b_{m+1\,m+1}) = 1 = s_1^0 \cdots s_m^0$. On the other hand, $b_{1\,n-2}$ and $b_{2\,n-3}$ cannot be compared in $\prec$ because $w(b_{1\,n-2}) = s_1^{-1} s_2^{-1}$ while $w(b_{2\,n-3}) = s_2^{-1} s_3^{-1}$. \\

The main fact we will need about the partial order  $(\mathcal{E},\prec)$ is that if $i \le i'$ and $j \le j'$, then $b_{ij} \prec b_{i'j'}$. \footnote{In particular, $b_{11}$ is the minimum of $(\mathcal{E},\prec)$!}

\subsection{Proof of Theorem \ref{thm: rough count} up to a cusp cutting estimate} \label{subsec: proof up to cusp cutting estimate}

Fix $A,b,r_{2},\delta$, let $L_{A,b}$ be the restriction of a lattice in $L \subset V_A(\Z)$ to $V_{A,b}^{r_2,\delta}(\Z)$. Let $\mathcal{L}, \mathcal{L}^{\irr} \subset V_{A_{pq},b}^{r_2,\delta}(\R)$ be defined from $L_{A,b}$ as above.

\subsubsection{Averaging}

By Bhargava's averaging argument \cite[Theorem 2.5]{MR3272925}, we find that for any non-empty open bounded $K$-invariant subset $G_{0} \subset \SO_{A_{pq}}(\R)$, we have: 
\begin{equation} \label{eqn: averaged count}
    N(L_{A,b} \cap V_{A,b}^{r_2,\delta}(\mathbb{Z}); X) = \frac{1}{\sigma(r_{2})\Vol(G_{0})} \int_{h \in \FF_{A_{pq}}} \#(hG_{0}\sigma_{V}(R_{A,b}^{r_2,\delta})(X) \cap \mathcal{L}^{{\rm irr }}) \,\, dh,
\end{equation}
where $dh$ denotes a Haar measure on $\SO_{A}$. The formula is self-normalising due to the term $1/\Vol(G_{0})$ and is thus independent of the choice of $dh$. We choose $dh$ to be the Haar measure chosen in \S \ref{subsec: choice of FA and haar measure on SOA}.  

\subsubsection{Main body and cuspidal part of the averaged count}

Recall that $\FF_{A_{pq}}$ is contained in a finite union of $\SO_{A_{pq}}(\mathbb{Q})$ translates of a Siegel domain $\mathcal{S}$, $$\FF_{A_{pq}} \subset \cup_\ell g_\ell \mathcal{S}$$ for some $g_\ell \in \SO_{A_{pq}}(\mathbb{Q})$. 
For each $\ell$, let $\widetilde{C}_\ell$ denote the minimum absolute value of the non-zero entries of elements of $g_\ell^{-1} \sigma_{V}(V_{A,b}(\Z))$. This is non-zero because $g_{\ell}^{-1}\sigma_{V}(V_{A,b}(\Z))$ is a lattice commensurable to $V_{A_{pq},b}(\Z)$. \\

We now divide $\FF_{A_{pq}}$ into two regions (which depend on $X$). We let $\FF_{A_{pq}}'$ denote the set of $h \in \FF_{A_{pq}}$ such that for some $\ell$, $\lvert b_{11}(g_\ell^{-1}v)\rvert < \widetilde{C_\ell}$ for all $v \in hG_{0}\sigma_{V}(R_{A,b}^{r_{2},\delta})(X)$. \footnote{We warn the reader that $\FF_{A_{pq}}'$ depends on $X$, although we omit this dependence from the notation for legibility.} We will also consider its complement $\FF_{A_{pq}} \setminus \FF_{A_{pq}}'$. We will refer to the portion of the integral in \eqref{eqn: averaged count} over $\FF_{A_{pq}}'$ as the \textbf{cuspidal} part of the integral and to the portion over its complement $\FF_{A_{pq}} \setminus \FF_{A_{pq}}'$ as the \textbf{main body} part. \\

The region $\FF_{A_{pq}}'$ is defined in this way because for $h \in \FF_{A_{pq}}'$, $\#(hG_{0}\sigma_{V}(R_{A,b}^{r_2,\delta})(X) \cap \mathcal{L}^{{\rm irr }})$ is not well approximated by the volume of $hG_{0}\sigma_{V}(R_{A,b}^{r_2,\delta})(X)$ since $g_\ell^{-1}hG_{0}\sigma_{V}(R_{A,b}^{r_2,\delta})(X) \cap g_\ell^{-1}\mathcal{L}^{{\rm irr }}$ lies entirely in the hypersurface $\{b_{11} = 0\}$ of $V_{A_{pq},b}$, for some $\ell$.

\subsubsection{Proof of Theorem \ref{thm: rough count} up to a cusp cutting estimate}

To prove Theorem \ref{thm: rough count}, we show that the cuspidal part of the integral in \eqref{eqn: averaged count} is negligible --- that is $o(X^{\frac{n(n+1)}{2}-1})$ --- while the integrand over the main body part of the integral is well approximated by the volume of $hG_{0}\sigma_{V}(R_{A,b}^{r_2,\delta})(X)$. \\

For the cuspidal part, there are cases depending on whether $\FF_{A_{pq}}$ can be taken compact. \footnote{$\FF_{A_{pq}}$ can be taken compact if and only if $A_{pq}$ is anisotropic over $\Q$.} 

\begin{enumerate}[1)]
\item If $A_{pq}$ is anisotropic, $\FF_{A_{pq}}$ may be chosen to be compact. This implies that $\FF_{A_{pq}}'$ is eventually empty by Remark \ref{homogeneously expanding}. Therefore, $$\int_{h \in \FF_{A_{pq}}'} \# (hG_{0}\sigma_{V}(R_{A,b}^{r_{2},\delta})(X) \cap \mathcal{L}^{{\rm irr}}) \,\, dh = 0$$ for $X$ large enough and is in particular $o(X^{\frac{n(n+1)}{2}-1})$. 
\item If $A_{pq}$ is isotropic over $\Q$, we will show in \S \ref{subsec: proof of cusp cutting estimate}, as a consequence of Lemma \ref{cusp cutting lemma}, that the following \textit{cusp cutting estimate} holds $$\int_{h \in \FF_{A_{pq}}'} \#(hG_{0} \sigma_{V}(R_{A,b}^{r_2,\delta})(X) \cap \mathcal{L}^{{\rm irr }}) \,\, dh = o(X^{\frac{n(n+1)}{2}-1}).$$ 
\end{enumerate}
Thus, granting item 2) above, we find that in all cases \[N(L; X) =\frac{1}{\sigma(r_{2})\Vol(G_{0})} \int_{h \in \FF_{A_{pq}} \setminus \FF_{A_{pq}}'} \#(hG_{0} \cdot \sigma_{V}(R_{A,b}^{r_2,\delta})(X) \cap \mathcal{L}^{{\rm irr }}) \,\, dh + o(X^{\frac{n(n+1)}{2}-1}).\] 

The average number of ``bad'' elements of $\mathcal{L}$ over the main body part is negligible. 

\begin{lem}
The integral of $\#(hG_{0} \sigma_{V}(R_{A,b}^{r_2,\delta})(X) \cap \mathcal{L} \setminus \mathcal{L}^{{\rm irr }})$ over $\FF_{A_{pq}} \setminus \FF_{A_{pq}}'$  is $o(X^{\frac{n(n+1)}{2}-1})$. 
\end{lem}
\begin{proof}
This follows
by adapting the proof of \cite[Proposition 4.6]{MR3782066} with the obvious modifications.
\end{proof} 

This allows us to change the integrand to $\#(hG_{0}\cdot \sigma_{V}(R_{A,b}^{r_2,\delta})(X) \cap \mathcal{L})$ in the main body part and apply Davenport's Lemma. 

\begin{lem}[Davenport's Lemma, \cite{MR43821} and \cite{MR166155}] \label{lem: Davenport's lemma}
Let $E \subset \mathbb{R}^n$ be a bounded semi-algebraic multiset with maximum multiplicity at most $m$, defined by $k$ algebraic inequalities each having degree at most $l$. Then the number of integral points in $E$ counted with multiplicity is \[\Vol(E)+O\left(\max_{\overline{E}} \{\Vol(\overline{E}),1\}\right)\] where $\Vol(\overline{E})$ denotes the greatest $d$-dimensional volume of a projection of $E$ onto a $d$-dimensional coordinate hyperplane for $1 \le d \le n-1$ and the implied constant depends only on $k,l,m$. \footnote{We note that a version of Davenport's lemma holds more generally for regions definable in an $o$-minimal structure, see \cite[Theorem 1.3]{MR3264671}. However, it is most commonly used, as is the case in the present article, for semi-algebraic regions.}
\end{lem}

Let $CX$ be a uniform bound on the coefficients of elements in $G_{0} \sigma_{V}(R_{A,b}^{r_{2},\delta})(X)$. Write $\Vol_{\mathcal{L}}(\cdot)$ for the Euclidean measure on $V_{A_{pq},b}(\R)$ normalised so that $\mathcal{L}$ has covolume $1$. Using Davenport's lemma, we find: 
\begin{align*}
&N(L;X) \\
&= \frac{1}{\sigma(r_{2})\Vol(G_{0})} \int_{h \in \FF_{A_{pq}} \setminus \FF_{A_{pq}}'} \Vol_{\mathcal{L}}(hG_{0}\sigma_{V}(R_{A,b}^{r_2,\delta})(X)) + O \left(\frac{\Vol_{\mathcal{L}}(hG_{0}\sigma_{V}(R_{A,b}^{r_{2},\delta})(X))}{Xw(b_{11})} \right) dh + o(X^{\frac{n(n+1)}{2}-1}) \\
&= \frac{\Vol(\FF_{A_{pq}} \setminus \FF_{A_{pq}}')\Vol_{\mathcal{L}}(G_{0}\sigma_{V}(R_{A,b}^{r_{2},\delta}(X))}{\sigma(r_{2})\Vol(G_{0})} + o(X^{\frac{n(n+1)}{2}-1}) \\
&= \frac{1}{\sigma(r_{2})}\Vol_{\mathcal{L}}(\FF_{A} \cdot R_{A,b}^{r_{2},\delta}(X))+ o(X^{\frac{n(n+1)}{2}-1}),
\end{align*}
where $w(b_{11})$ is understood for an element $h \in g_{\ell} \mathcal{S}$ written as $g = g_{\ell} ntk$ as the weight of the $t$. We have used that $\Vol(\FF_{A_{pq}}') = o(1)$ (since being in $\FF_{A_{pq}}'$ requires at least one of the $s_{i}$ to be at least $(CX/\widetilde{C}_\ell)^{\frac{1}{2m}}$), the estimates $\int_{h \in \FF_{A_{pq}} \setminus \FF_{A_{pq}}'} O \left(1/Xw(b_{11}) \right) dh =  O(X^{-1/2})$ if $n = 3$ and $O(X^{-1})$ for $n > 3$, and the Jacobian change of variable formula Proposition \ref{change of measure formula}. \\

This yields the asymptotic formula
\[N(L_{A,b} \cap V_{A,b}^{r_2,\delta}(\mathbb{Z});X) = \frac{1}{\sigma(r_2)}\frac{\Vol (\mathcal{F}_A \cdot R_{A,b}^{r_2,\delta}(X))}{\mathrm{co}\Vol(L_{A,b})} +  o(X^{\frac{n(n+1)}{2}-1}),\]
completing the proof of Theorem \ref{thm: rough count}. \\

It remains to prove the cusp cutting estimate of item 2) above.

\subsection{Proof of the cusp cutting estimate} \label{subsec: proof of cusp cutting estimate}

The goal of this subsection is to prove the \textit{cusp cutting estimate} 
\begin{equation} \label{eqn: cusp cutting estimate}
    \int_{h \in \FF_{A_{pq}}'} \#(hG_{0} \sigma_{V}(R_{A,b}^{r_2,\delta})(X) \cap \mathcal{L}^{{\rm irr }}) \,\, dh = o(X^{\frac{n(n+1)}{2}-1}).
\end{equation}

Recall that $\FF_{A_{pq}}$ is contained in a finite union of $\SO_{A_{pq}}(\mathbb{Q})$ translates of a Siegel domain $\mathcal{S}$, $\FF_{A_{pq}} \subset \cup_i g_\ell \mathcal{S}$ for $g_\ell \in \SO_{A_{pq}}(\mathbb{Q})$ and that $\widetilde{C}_\ell$ denotes the minimum absolute value of the non-zero entries of elements of $g_\ell^{-1}\sigma_{V}(V_{A,b}(\Z))$.

\begin{defn}
For $\E_1 \subset \E$, we define \[V_{A_{pq},b}(\mathbb{R})(\E_1,\ell) = \{B \in V_{A_{pq},b}(\mathbb{R}) \colon \left| b_{ij}(B) \right| < \widetilde{C}_\ell \text{ if and only if } b_{ij} \in \E_1 \}.\] 
\end{defn}
We set $g_\ell^{-1} \cdot \mathcal{L}(\E_{1},\ell) := g_\ell^{-1} \cdot \mathcal{L} \cap V_{A_{pq},b}(\mathbb{R})(\E_1,\ell)$. 

Note that $ \#(hG_{0}\sigma_{V}(R_{A,b}^{r_2,\delta})(X) \cap \mathcal{L}^{{\rm irr }}) =  \#(g_{\ell}^{-1}hG_{0}\sigma_{V}(R_{A,b}^{r_2,\delta})(X) \cap g_{\ell}^{-1}\cdot \mathcal{L}^{{\rm irr }})$. 
Therefore, by definition of $\FF_{A_{pq}}'$, it is sufficient to show that 
\begin{equation} \label{eqn: cusp integral over Siegel set}
\int_{s \in \mathcal{S}} \#(sG_{0} \sigma_{V}(R_{A,b}^{r_2,\delta})(X) \cap g_{\ell}^{-1}\mathcal{L}^{{\rm irr }}(\E_{1},\ell)) \,\, ds = O_\epsilon\left(X^{\left(\frac{n(n+1)}{2}-1\right)-1+\epsilon}\right)
\end{equation}
for all $\ell$, $g_{\ell}$ and all $\E_1 \subset \mathcal{E}$ satisfying $b_{11} \in \E_1$ in order to prove the cusp cutting estimate \eqref{eqn: cusp cutting estimate}. \\

Fix an element $g_{\ell}$. The criterion for ``bad'' elements in Theorem \ref{thm: criterion for badness} gives us a priori bounds on the coordinates $s_i$ in the integral above, as we now explain. Let $C'$ be an absolute constant such that $C'X$ bounds the absolute values of all the entries of elements $B \in t^{-1}ut G_0 \sigma_{V}(R_A^{r_2,\delta})(X)$ for $t \in T$ and $u \in N$. If $t = (s_1^{-1},\ldots,s_m^{-1},1,\ldots,1,s_m,\ldots,s_1) \in T$ and $C'X w(b_{i_0 \, n-i_0}) < \widetilde{C}_\ell$ for some $i_0 \in \{1,\ldots,m\}$, then $C'Xw(b_{ij})< \widetilde{C}_\ell$ for all $i \le i_0$ and $j \le n-i_0$ and then the {\bf \emph{rectangles}} case of the criterion for badness Theorem \ref{thm: criterion for badness} gives that $\#(t(t^{-1}ut)G_{0} \sigma_{V}(R_{A,b}^{r_2,\delta})(X) \cap g_{\ell}^{-1}\cdot\mathcal{L}^{{\rm irr }}) = 0$. Therefore, we may assume that
\begin{equation} \label{eqn: s conditions}
    s_i < (C'/\widetilde{C}_\ell)X
\end{equation}
for all $i \in \{1,\ldots,m\}$ by looking at the weights of elements above the anti-diagonal. Let $C = C'/\widetilde{C}_\ell$ and $T_X$ denote the set of $t = (s_1^{-1},\ldots,s_m^{-1},1,\ldots,1,s_m,\ldots,s_1) \in T$ which satisfy condition \eqref{eqn: s conditions}. 

By Davenport's lemma and the fact that our region contains the translate of a homogeneously expanding region in $X$, we see (by the same argument used to justify \cite[Equation (21)]{MR3719247}) that
\begin{align*}
&\int_{h \in \mathcal{S}} \#(hG_{0} \sigma_{V}(R_{A,b}^{r_2,\delta})(X) \cap g_{\ell}^{-1}\mathcal{L}^{{\rm irr }}(\E_{1},\ell)) dh \\ 
&= O\left(\int_{t \in T_X} \Vol(t G_0 \sigma_{V}(R_{A,b}^{r_2,\delta})(X) \cap V_{A_{pq},b}(\mathbb{R})(\E_1,\ell)) \prod_{i=1}^{m} s_i^{i(i+1-p-q)} d^{\times}s  \right) \\
&= O\left(X^{\left(\frac{n(n+1)}{2}-1\right)-\#\E_1} \int_{t \in T_X} \prod_{b_{ij} \not\in \E_1} w(b_{ij}) \prod_{i=1}^{m} s_i^{i(i+1-n)} d^\times s \right).
\end{align*}
Thus, we have reduced the problem to estimating the following integral. 

\begin{defn}
The active integral of $\E_1 \subset \mathcal{\E}$ is defined by \footnote{We choose the notation $\widetilde{I}$, and later on $I$, for ``integral''.} \[\widetilde{I}(\E_1,X) := X^{\left(\frac{n(n+1)}{2}-1\right)-\#\E_1}  \int_{t \in T_X} \prod_{b_{ij} \not\in \E_1} w(b_{ij}) \prod_{i=1}^{m} s_i^{i(i+1-n)} d^\times s .\]
\end{defn}

\noindent Recall, that $b_{ij} \prec b_{i_0j_0}$ when $i \le i_0$ and $j \le j_0$. Therefore, if $\E_1 \subset \E$ contains $b_{i_0j_0}$ but not $b_{ij}$, then \[\widetilde{I}\left(\E_1 \setminus \{b_{i_0j_0} \} \cup \{b_{ij}\}, X \right) \ge \widetilde{I}(\E_1,X).\] As a result, to obtain an upper bound for $\widetilde{I}(\E_1, X)$ we may assume that if $b_{i_0j_0} \in \E_1$, then $b_{ij} \in \E_1$ for all $i \le i_0$ and $j \le j_0$. We call such $\E_{1}$ left-closed and up-closed. Furthermore, suppose $\E_1$ contains any element on, or on the right of, the off anti-diagonal within the first $m$-rows. In that case, the {\bf \emph{rectangles}} case of the criterion for badness Theorem \ref{thm: criterion for badness} gives $\#(hG_{0} \sigma_{V}(R_{A,b}^{r_2,\delta})(X) \cap g_{\ell}^{-1}\mathcal{L}^{{\rm irr }}(\E_{1})) =0$. So, in that case, we are already done. Thus, it suffices to consider $\E_{1}$ containing only elements to the left of the off-antidiagonal. 

\begin{defn} Define $\E_0 \subset \E$ as the set of coordinates $b_{ij}$ such that $i \le m$ and $i \le j$ and $i+j \le n-1$. 
\end{defn}

It is easy to see that $\E_0$ has size $\#\E_0 = m(n-(m+1))$. If $|p-q|=1$, every element in $V(\mathbb{Z})(\E_0)$ is ``bad'' and it suffices to consider $\widetilde{I}(\E_1,X)$ for all $\E_1 \subsetneq \E_0$. If $|p-q|>1$, all $\E_1 \subset \E_{0}$ need to be considered. 

The product of the weights of all the coordinates is $1$. We normalise the active integral to simplify notation. 

\begin{defn}
We define for a subset $\E_1 \subset \E_{0}$ \[I(\E_1,X) = X^{-\left(\frac{n(n+1)}{2}-1\right)} \widetilde{I}(\E_1,X) = X^{-\#\E_1}  \int_{t \in T_X} \prod_{b_{ij} \in \E_1} w(b_{ij})^{-1} \prod_{i=1}^{m} s_i^{i(i+1-n)} d^\times s .\]
\end{defn}

In this way, the desired cusp cutting estimate \eqref{eqn: cusp cutting estimate} will follow from the following cusp-cutting lemma, which is the main technical result of the section.

\begin{lem}[Cusp cutting lemma] \label{cusp cutting lemma}
Let $\E_1$ be a non-empty proper subset of $\E_0$. Then we have the estimate \[I(\E_1,X) = O_\epsilon \left(X^{-1+\epsilon}\right).\]
We also have $I(\emptyset) = O(1)$ and $I(\E_0) = O\left(X^{m(2m+1-n)+\epsilon}\right)$. 
\end{lem}

In the following example, we estimate $I(\E_1, X)$ in the case $n=5, m=2$ as it serves to illustrate Lemma \ref{cusp cutting lemma} and will be a base case in the induction argument of our proof of the Lemma.

\begin{exmp}[Base case of cusp cutting induction for $|p-q|=1$] \label{cusp example}
We now do the case $n=5$, $m=2$ before cutting off the cusp in the general case. We see that torus elements $t$ act as follows 
\[t\cdot v = \begin{pmatrix}
t_1^{-2} & t_1^{-1} t_2^{-1} & t_1^{-1} & t_1^{-1} t_2 & 1 \\
t_2^{-1} t_1^{-1} & t_2^{-2} & t_2^{-1} & 1 & t_2^{-1} t_1 \\
t_1^{-1} & t_2^{-1} & 1 & t_2 & t_1 \\
t_2 t_1^{-1} & 1 & t_2 & t_2^2 & t_2 t_1 \\
1 & t_1 t_2^{-1} & t_1 & t_1 t_2 & t_1^2 
\end{pmatrix} O(X),\]
from which it is easy to read off the weights. 
The Haar measure takes the form \[dg = du \, \frac{1}{t_1^3 t_2} d^{\times}t \, dk =  du \, \frac{1}{s_1^{3} s_2^4} d^{\times}s \, dk .\]

For any subset of $\E$ containing $b_{11}$, we now want to estimate \[I(\E_1,X) = X^{-\#\E_1}  \int_{t \in T_X} \prod_{b_{ij} \not\in \E_1} w(b_{ij}) \frac{d^{\times}s}{s_1^3 s_2^4}.\]

We only need to look at proper subsets of $\E_0 = \{b_{11},b_{12}, b_{13}, b_{22}\}$ which are left-closed and up-closed. In this case we can exclude the subset $\{b_{11},b_{12}, b_{22}\}$ by the {\bf \emph{squares}} case of the criterion for badness Theorem \ref{thm: criterion for badness}. and recall that we have the bound $s_1,s_2 < CX$. Let's compute: 
\begin{alignat*}{2}
&{I}(\{b_{11}\},X) &&= X^{-1} \int_{s_1,s_2=c}^{CX} s_1^2 s_2^2 \frac{d^{\times}s}{s_1^3 s_2^4} = X^{-1} \int_{s_1,s_2=c}^{CX} \frac{d^{\times}s}{s_1 s_2^2} = O(X^{-1}) \\
&{I}(\{b_{11},b_{12}\},X) &&= X^{-2} \int_{s_1,s_2=c}^{CX} s_1^3 s_2^4 \frac{d^{\times}s}{s_1^3 s_2^4} = X^{-2} \int_{s_1,s_2=c}^{CX} d^{\times}s = O_\epsilon(X^{-2+\epsilon}) \\
&{I}(\{b_{11},b_{12},b_{13}\},X) &&= X^{-3} \int_{s_1,s_2=c}^{CX} s_1^4 s_2^5 \frac{d^{\times}s}{s_1^3 s_2^4} = X^{-3} \int_{s_1,s_2=c}^{CX} s_1s_2 d^{\times}s = O(X^{-1})  \\
&{I}(\{b_{11},b_{12},b_{22}\},X) &&= X^{-3} \int_{s_1,s_2=c}^{CX} s_1^3 s_2^6 \frac{d^{\times}s}{s_1^3 s_2^4} = X^{-3} \int_{s_1,s_2=c}^{CX} s_2^2 d^{\times}s= O_\epsilon(X^{-1+\epsilon}).
\end{alignat*}
All of them are $O_\epsilon(X^{-1+\epsilon})$, and we have thus proven Lemma \ref{cusp cutting lemma} in this case. 
\end{exmp}

\begin{proof}[Proof of Lemma \ref{cusp cutting lemma}] \footnote{
The proof of Lemma \ref{cusp cutting lemma} is inspired by the induction argument of \cite[\S 4.2]{MR3719247}. It gives much better bounds for $I(\E_1, X)$ than the ones recorded in the statement of Lemma \ref{cusp cutting lemma}, which are sufficient for our purpose.}

By the preceding discussion, we can assume that $\E_{1}$ is left-closed and up-closed. We prove Lemma \ref{cusp cutting lemma} via a combinatorial argument using induction on $m$. Recall that $n=2m+|p-q|$. The cases $|p-q|=1$ and $|p-q| > 1$ are slightly different. We handle them separately. 

To start, let us assume that $|p-q| > 1$. First, we compute $I(\E_0, X)$. 
\begin{align*}
I(\E_0,X) &= X^{-\#\E_0} \int_{t \in T_X} \prod_{b_{ij} \in \E_0} w(b_{ij})^{-1} \prod_{i=1}^{m} s_i^{i(i+1-n)} d^{\times}s. \\ 
&= X^{-m(n-(m+1))} \int_{t \in T_X} \left(t_1^{n-2+1} t_2^{n-4+2} t_3^{n-6+2} \cdots t_m^{n-2m+2} \right)   \prod_{i=1}^{m} t_i^{2i-n} d^{\times}t \\
&= X^{-m(n-(m+1))} \int_{t \in T_X} t_1t_2^2\cdots t_m^2 d^{\times}t \\
&= X^{-m(n-(m+1))} \int_{s_1,\ldots,s_m=c}^{CX} s_1 s_2^3 s_3^5 \cdots s_m^{2m-1} d^{\times}s \\
&= O\left(X^{-m(n-(m+1))+m^2}\right) \\
&= O\left(X^{m(2m+1-n)}\right) \\
&= O\left(X^{-m(|p-q|-1)}\right).
\end{align*}
We also compute $I(\emptyset,X)$ directly
\begin{equation*}
 I(\emptyset,X) =\int_{s_1, \ldots, s_n = c}^{CX} \prod_{i=1}^{m} s_i^{i(i+1-n)} d^{\times}s = O(1).
\end{equation*}
Let $\E_1'$ denote $\E_0 \setminus \E_1$. Define $I_m'(\E_1',X) :=I(\E_1,X)$. Then we have: 
\begin{align*}
I_m'(\E_1',X) &= X^{\#\E_1'-m(n-(m+1))} \int_{t \in T_X} \left( \prod_{b_{ij} \in \E_1'} w(b_{ij}) \right) t_1t_2^2\cdots t_m^2 d^{\times}t. \\
&= X^{\#\E_1'-m(n-(m+1))} \int_{s_1,\ldots,s_m = c}^{XC} \left( \prod_{b_{ij} \in \E_1'} w(b_{ij}) \right) s_1 s_2^3 \cdots s_m^{2m-1} d^{\times}s. 
\end{align*}

We have now worked out the base case of the induction. When $m = 1$, we have 
\begin{align*}
I_1(\emptyset,X) &= O\left(1\right)  \\
I_1(\{b_{11}\},X) &= X^{-1} \int_{s_1=c}^{XC} s_1^2 s_1^{2-n} d^{\times}s = O_\epsilon \left(X^{-1+\epsilon} \right) \\
I_1(\{b_{11}, \ldots, b_{1k} \},X) &= X^{-k} \int_{s_1=c}^{XC} s_1^2s_1^{k-1} s_1^{2-n} d^{\times}s = O_\epsilon\left( X^{-1+\epsilon} \right) \\
I_1(\{b_{11}, \ldots, b_{1\,n-2} \},X) &= X^{-n+2} \int_{s_1=c}^{XC} s_1 d^{\times}s = O_\epsilon\left( X^{1-|p-q|+\epsilon} \right) \\
I_1(\E_0,X) &= O_\epsilon\left(X^{1-|p-q|+\epsilon}\right).
\end{align*}
In particular, we see that when $|p-q| > 1$, all these quantities are $O_\epsilon(X^{-1+\epsilon})$. We will use this estimate in the induction step. Now, suppose that $m \ge 2$. For any decomposition $k=k_1+k_2$ we have: \[\int_c^{CX} s^k d^{\times}s \ll_{c,C} \int_{c}^{CX} s^{k_1} d^{\times}s  \int_{c}^{CX} s^{k_2} d^{\times}s.\] Consequently, we see that $I_m'(\E_1',X)$ is bounded by the product \[I_m'(\E_1',X) \le J_m(\E_2',X) \, K_m(\E_3',X),\] where $\E_2'$ consist of all the elements of $\E_1'$ in the first row, $\E_3'$ consists of the rest of the elements of $\E_1'$, and
\begin{alignat*}{2}
&J_m(\E_2',X) &&= \left(X^{\#\E_2'-(n-2)} \int_{s_1, \ldots, s_n = c}^{CX} \left( \prod_{b_{1j} \in \E_2'} w(b_{1j}) \right) s_1 s_2^2 \cdots s_m^2 d^{\times}s \right) \\
&K_m(\E_3',X) &&= \left(X^{\#\E_3'-\#\E_0+(n-2)} \int_{s_2, \ldots, s_n = c}^{CX} \left(\prod_{b_{ij} \in \E_3'} w(b_{ij}) \right) s_2 s_3^3 \cdots s_m^{2m-3} d^{\times}s \right). 
\end{alignat*}
Note that $K_m(\E_3',X) = I_{m-1}(\E_3',X)$ and we can estimate it by induction. Now, $\E_1$ is left-closed and non-empty and hence the subset $\E_2'$ is either empty or of the form $\{b_{1\,k}, \ldots, b_{1\,n-2}\}$ for $k \ge 2$. 
Now, if $\E_2' =\emptyset$: 
\begin{equation*}
J_m(\E_2',X) = X^{2-n} \int_{s_1, \ldots, s_n = c}^{CX} s_1 s_2^2 \cdots s_m^2 d^{\times}s = O_\epsilon\left(X^{2m-1-n+2}\right) = O_\epsilon\left(X^{1-|p-q|+\epsilon} \right) = O_\epsilon\left(X^{-1+\epsilon}\right).
\end{equation*} 
Now, if $k=2$, then: 
\begin{align*}
J_m(\E_2',X) &= X^{-1} \int_{s_1, \ldots, s_n = c}^{CX}  t_1^{3-n} t_2^{-1} s_1 s_2^2 \cdots s_m^2 d^{\times}s \\
&= X^{-1} \int_{s_1, \ldots, s_n = c}^{CX}  s_1^{3-n}s_2^{3-n}\cdots s_m^{3-n} s_2^{-1} \cdots s_m^{-1} s_1 s_2^2 \cdots s_m^2 d^{\times}s \\
&= X^{-1} \int_{s_1, \ldots, s_n = c}^{CX}  s_1^{3-n}s_2^{3-n}\cdots s_m^{3-n} s_1 s_2 \cdots s_m d^{\times}s \\
&= X^{-1} \int_{s_1, \ldots, s_n = c}^{CX}  s_1^{4-n}s_2^{4-n}\cdots s_m^{4-n} d^{\times}s \\
&= O_\epsilon (X^{-1+\epsilon})
\end{align*} 
Now, if $k=3$, then: 
\begin{align*}
J_m(\E_2',X) &= X^{-2} \int_{s_1, \ldots, s_n = c}^{CX}  t_1^{4-n} s_1 s_2^2 \cdots s_m^2 d^{\times}s \\
&= X^{-2} \int_{s_1, \ldots, s_n = c}^{CX}  s_1^{4-n} s_2^{4-n} \cdots s_m^{4-n} s_1 s_2^2 \cdots s_m^2 d^{\times}s \\
&= X^{-2} \int_{s_1, \ldots, s_n = c}^{CX}  s_1^{5-n} s_2^{6-n} \cdots s_m^{6-n} d^{\times}s \\
&= O_\epsilon \left(X^{-2+\epsilon}\right).
\end{align*} 
If $4 \le k \le m$, then: 
\begin{align*}
J_m(\E_2',X) &= X^{1-k} \int_{s_1, \ldots, s_n = c}^{CX} t_1^{-((n-2)-k+1)} t_k^{-1} \cdots t_m^{-1} t_m \cdots t_3 s_1 s_2^2 \cdots s_m^2 d^{\times}s \\
&= X^{1-k} \int_{s_1, \ldots, s_n = c}^{CX}  t_1^{-((n-2)-k+1)} t_3 \cdots t_{k-1} s_1 s_2^2 \cdots s_m^2 d^{\times}s \\
&= X^{1-k}  \int_{s_1, \ldots, s_n = c}^{CX}  t_1^{-((n-2)-k+1)} s_3 s_4^2 \cdots s_{k-1}^{k-3} s_1 s_2^2 \cdots s_m^2 d^{\times}s \\
&= X^{1-k}  \int_{s_1, \ldots, s_n = c}^{CX}  t_1^{-((n-2)-k+1)} s_1 s_2^2 s_3^3 s_4^4 s_{k-1}^{k-1} s_k^2 \cdots s_m^2 d^{\times}s \\
&= O_\epsilon(X^{1-k+\epsilon}).
\end{align*} 
If $m+1 \le k < m+|p-q|$, then: 
\begin{align*}
J_m(\E_2',X) &= X^{1-k} \int_{s_1, \ldots, s_n = c}^{CX} t_1^{-((n-2)-k+1)} t_m \cdots t_3 s_1 s_2^2 \cdots s_m^2 d^{\times}s \\
&= X^{1-k} \int_{s_1, \ldots, s_n = c}^{CX}  t_1^{-((n-2)-k+1)} t_3 \cdots t_{m} s_1 s_2^2 \cdots s_m^2 d^{\times}s \\
&= X^{1-k}  \int_{s_1, \ldots, s_n = c}^{CX}  t_1^{-((n-2)-k+1)} s_3 s_4^2 \cdots s_m^{m-2} s_1 s_2^2 \cdots s_m^2 d^{\times}s \\
&= X^{1-k}  \int_{s_1, \ldots, s_n = c}^{CX}  t_1^{-((n-2)-k+1)} s_1 s_2^2 s_3^3 s_4^4 \cdots s_m^m d^{\times}s \\
&= O_\epsilon(X^{1-k+\epsilon}).
\end{align*} 
If $m+|p-q| \le k \le n-2$, then: 
\begin{align*}
J_m(\E_2',X) &= X^{1-k} \int_{s_1, \ldots, s_n = c}^{CX} t_1^{-((n-2)-k+1)} t_{n-2-k+3} \cdots t_3 s_1 s_2^2 \cdots s_m^2 d^{\times}s \\
&= X^{1-k} \int_{s_1, \ldots, s_n = c}^{CX}  t_1^{-((n-2)-k+1)} t_3 \cdots t_{n-k+1} s_1 s_2^2 \cdots s_m^2 d^{\times}s \\
&= X^{1-k}  \int_{s_1, \ldots, s_n = c}^{CX}  t_1^{-((n-2)-k+1)} s_3 s_4^2 \cdots s_{n-k+1}^{n-k-1} s_{n-k+2}^{n-k-1} \cdots s_m^{n-k-1} s_1 s_2^2 \cdots s_m^2 d^{\times}s \\
&= X^{1-k}  \int_{s_1, \ldots, s_n = c}^{CX}  t_1^{-((n-2)-k+1)} s_1 s_2^2 s_3^3 s_4^4 \cdots s_{n-k+1}^{n-k+1} s_{n-k+2}^{n-k+1} \cdots s_m^{n-k+1} d^{\times}s \\
&= X^{1-k}  \int_{s_1, \ldots, s_n = c}^{CX}  s_1^{-(n-k)+1} s_2^{-(n-k)+1}\cdots s_m^{-(n-k)+1} s_1 s_2^2 s_3^3 s_4^4 \cdots s_{n-k+1}^{n-k+1} s_{n-k+2}^{n-k+1} \cdots s_m^{n-k+1} d^{\times}s \\
&= X^{1-k}  \int_{s_1, \ldots, s_n = c}^{CX}  s_1^{-(n-k)+2} s_2^{-(n-k)+3} \cdots s_{n-k-1}^{0}s_{n-k}^{1} s_{n-k+1}^{2} \cdots s_m^{2} d^{\times}s \\
&= O_\epsilon(X^{k-(n-2)-|p-q|+\epsilon}).
\end{align*} 
Therefore, in all cases we find \[J_m(\E_2',X) = O_\epsilon(X^{-1+\epsilon}).\] The lemma now follows by induction on $m$ used to bound $I'_{m-1}(\E_3',X)$ by $O_\epsilon(X^{-1+\epsilon})$. 

We now explain how to deal with the case $|p-q|=1$. In this case, we will have to use the {\bf \emph{squares}} case of the criterion for badness Theorem \ref{thm: criterion for badness}. This will guarantee that $\E_1 \neq \E_0$. Thus, the base case of the induction is not $m=1$ since $\E_{0} = \{b_{11}\}$ and thus, we can start the induction at $m=2$. Example \ref{cusp example} shows that the estimates $I(\E_1,X) = O_\epsilon(X^{-1+\epsilon})$ for $\E_1 \neq \emptyset$ and $I(\emptyset, X) = O(1)$ do hold for this base case. For $m \ge 3$, the estimates obtained in the calculations above remain valid: \[J_m(\E_2',X) = O_\epsilon(X^{-1+\epsilon})\] for $\E_2' \neq \emptyset$ and $J_m(\emptyset,X) = O(1)$.

So, we have to rewrite the induction step slightly. We do so as follows. If $\E_2'$ is non-empty, then the lemma follows by induction on $m$ used to bound $I_{m-1}'(\E_3',X)$ by $O_\epsilon(X^{\epsilon})$. If, on the other hand, $\E_2'$ is empty, then $\E_3'$ must be non-empty since $\E_1'$ is non-empty. This holds because the {\bf \emph{squares}} case of the criterion for badness Theorem \ref{thm: criterion for badness} implies that $\E_1 \neq \E_0$. If $\E_3' \neq \E_0 \setminus \{b_{1 \,1},\ldots,b_{1 \,n-2}\}$, then by induction $I'_{m-1}(\E_3',X) = O_\epsilon(X^{-1+\epsilon})$. The last outstanding case is when $\E_{1} = \{b_{1\,1}, \ldots, b_{1 \, n-2} \}$. In this case, a direct computation gives the result, completing the proof.
\end{proof}

\section{Sieving to very large and acceptable collections}\label{sec: sieves}

We obtain asymptotic formulas for $N(\mathcal{V}(\Lambda_{A,b}^{\delta}); X)$ in very large families, allowing us, in particular, to restrict the count of the previous sections to projective orbits in very large families of orders in \S \ref{sec: proof of main theorem}. We also obtain upper bounds and conditional asymptotic formulas in acceptable families. 
This section is a straightforward adaptation of \cite[\S 5]{MR3782066} to our setting.

\subsection{Acceptable and very large subfamilies of $V_{A,b}^{r_{2},\delta}(\Z)$}

Fix $A \in \mathscr{L}_{\Z}$, $1 \le b < n$, and $0 \le r_{2} \le (n-1)/2$ and let $\widehat{\Z} = \prod_{p} \Z_{p}$. \\ 

A \textbf{collection of local specifications} is a subset of the form $\Lambda_{A,b}^{\delta}  = \prod_{p} \Lambda_{A,b,p} \times V_{A,b}^{r_2,\delta}(\R) \subset V_{A,b}(\widehat{\Z}) \times V_{A,b}(\R)$, for some $r_{2}$ and $\delta \in \mathcal{T}(r_{2})$ and such that each $\Lambda_{A,b,p} \subset V_{A,b,p}(\Z_{p}) \setminus \{ \Delta=0 \}$ is non-empty, clopen, with measure $0$ boundary. We denote by $\mathcal{V}(\Lambda_{A,b}^{\delta}) \subset V_{A,b}^{r_{2},\delta}(\Z)$ the preimage of $\Lambda_{A,b}$ under the diagonal embedding $V_{A,b}^{r_{2},\delta}(\Z) \hookrightarrow V_{A,b}(\widehat{\Z}) \times V_{A,b}(\R)$. We define $N(\mathcal{V}(\Lambda_{A,b}^{\delta});X)$ as the number of absolutely irreducible $\SO_{A}(\Z)$-orbits of height at most $X$ in the subfamily $\mathcal{V}(\Lambda_{A,b}^{\delta})$. \\

We say that $\Lambda_{A,b}^{\delta}$ is {\bf very large} if for all but finitely many primes the sets $\Lambda_{A,b,p}$ contains all projective elements of $V_{A,b}(\mathbb{Z}_p)$ (i.e. with the property that $I^2 = (\delta)$ under the parametrisation \ref{integral parametrisation}). If $\Lambda_{A,b}^{\delta}$ is very large, we also say that the associated subfamily $\mathcal{V}(\Lambda_{A,b}^{\delta})$ is very large. \\ 

We say that $\Lambda_{A,b}^{\delta}$ is {\bf acceptable} if for all but finitely many primes, the set $\Lambda_{A,b,p}$ contains all elements of $V_{A,b}(\mathbb{Z}_p)$ whose discriminant is not divisible by $p^2$. If $\Lambda_{A,b}^{\delta}$ is acceptable, we also say that the associated subfamily $\mathcal{V}(\Lambda_{A,b}^{\delta})$ is acceptable.

\subsection{Sieving to very large subfamilies}

Let $p$ be a prime. Let $V_{A,b}(\mathbb{Z}_p)^{\rm proj}$ denote the set of elements $v \in V_{A,b}(\mathbb{Z}_p)$ which correspond to projective pairs $(I,\delta)$ (i.e. with the property that $I^2 = (\delta)$ under the parametrisation \ref{integral parametrisation}). Note that $V_{A,b}^{r_2,{\rm proj}} (\mathbb{Z}) = V_{A,b}^{r_2}(\mathbb{Z}) \bigcap \left( \bigcap_p V_{A,b}^{\rm proj}(\mathbb{Z}_p) \right)$. Let $W_{A,b,p}$ be the set of elements in $V_{A,b}(\mathbb{Z})$ that do not belong to $V_{A,b}^{\rm proj}(\mathbb{Z}_p)$. 

We need estimates for the number of elements in $W_{A,b,p}$ for large $p$. We have the following theorem, whose statement and proof are adaptations of \cite[Theorem 5.1]{MR3782066} to our setting.

\begin{thm}
We have
\begin{equation*} 
N \left(\cup_{p \ge M} W_{A,b,p}; X \right) = O \left(\frac{X^{\frac{n(n+1)}{2}-1}}{M^{1-\epsilon}} \right)+o\left(X^{\frac{n(n+1)}{2}} \right)
\end{equation*} 
where the implied constant is independent of $X$ and $M$.
\end{thm}
\begin{proof} One shows just as in \cite[Theorem 5.1]{MR3782066} that $W_{A,b,p} \subset V(\mathbb{Z}_p)$ is the preimage of some subset of $V_{A,b}(\mathbb{F}_p)$ under the reduction modulo $p$ map by using Nakayama's lemma. Making the necessary adjustments, the proof proceeds just as in \textit{loc. cit.}, noting that the reduction modulo $p$ of $W_{A,b,p}$ has codimension greater than $2$ in $V_{A,b}(\mathbb{F}_p)$ (being non-projective modulo $p$ and having discriminant divisible by $p$ give at least $2$ conditions).\end{proof}

\begin{thm}
Let $r_2$ be an integer such that $0 \le r_2 \le \frac{n-1}{2}$ and let $\delta \in \mathcal{T}(r_2)$. Then for a very large collection of local specifications $\Lambda_{A,b}^{\delta}$, we have 
\begin{equation*} 
N(\mathcal{V}(\Lambda_{A,b}^\delta), X) = \frac{1}{\sigma(r_2)} \Vol(\mathcal{F}_A \cdot R_{A,b}^{r_2,\delta}(X)) \prod \limits_p \Vol(\Lambda_{A,b,p}) + o\left(X^{\frac{n(n-1)}{2}-1}\right),
\end{equation*}
where the volumes of subsets of $V_{A,b}(\mathbb{R})$ are computed with respect to the Euclidean measure normalised so that $V_{A,b}(\mathbb{Z})$ has covolume $1$. The volumes of subsets of $V_{A,b}(\mathbb{Z}_p)$ are computed with respect to the Euclidean measure normalised so that $V_{A,b}(\mathbb{Z}_p)$ has measure $1$. 
\end{thm}

\subsection{Sieving to acceptable subfamilies conditional on a tail estimate} 

Let $p$ be a prime. We denote by $\mathcal{W}_{A,b,p}$ the set of elements $v \in V_{A,b}(\mathbb{Z})$ such that $p^2 \mid \Delta(v)$. 
We make the following conjecture on tail estimates. They are known for $n=3$ by a proof akin to the proof of \cite[Proposition 4.18]{BhargavaHankeShankar} and unknown for $n \ge 5$.

\begin{conj}[Conjectural tail estimates] \label{conj: conjectural tail estimate}

We have
\begin{equation*}
N(\cup_{p \ge M} \mathcal{W}_{A,b,p}; X) = O \left(\frac{X^{\frac{n(n+1)}{2}-1}}{M^{1-\epsilon}}\right) + o \left( X^{\frac{n(n+1)}{2}-1} \right)
\end{equation*}
where the implied constant is independent of $X$ and $M$.
\end{conj}

We have the following asymptotic formula, conditional on the preceding tail estimates. 

\begin{thm} \label{first asymptotic formula}
Let $r_2$ be an integer such that $0 \le r_2 \le \frac{n-1}{2}$ and let $\delta \in \mathcal{T}(r_2)$. Then for an acceptable collection of local specifications $\Lambda_{A,b}^{\delta}$, we have 
\begin{equation*}
N(\mathcal{V}(\Lambda_{A,b}^\delta); X) \le \frac{1}{\sigma(r_2)} \Vol(\mathcal{F}_A \cdot R_{A,b}^{r_2,\delta}(X)) \prod \limits_p \Vol(\Lambda_{A,b,p}) + o\left(X^{\frac{n(n-1)}{2}-1}\right),
\end{equation*}
where the volumes of subsets of $V_{A,b}(\mathbb{R})$ are computed with respect to the Euclidean measure normalised so that $V_{A,b}(\mathbb{Z})$ has covolume $1$. The volumes of subsets of $V_{A,b}(\mathbb{Z}_p)$ are computed with respect to the Euclidean measure normalised so that $V_{A,b}(\mathbb{Z}_p)$ has measure $1$. This is an equality if we assume that the tail estimates of Conjecture \ref{conj: conjectural tail estimate} hold. 
\end{thm}

\section{Local volumes and local masses}\label{sec: product of local volumes} 

In this section, we describe a change of variable formula. It will help us reduce the computation of volumes and multisets in $V_{A,b}^{r_2,\delta}(\R)$ and $V_{A,b,p}(\Z_{p})$ to calculations of integrals of local masses over $U_{1,b}^{r_2}(\R)$ and $U_{1,b}(\Z_p)$. However, first, we choose a volume form over $\Q$ for $\SO_{A}$ which will be helpful in comparing Jacobian change of variables across the different $A \in \mathscr{L}_{\Z}$. Such comparisons will be helpful in our study $2$-adic local masses in \S \ref{sec: 2-adic distribution}.  

\subsection{Choice of volume forms} \label{subsec: choice of volume forms}

Let $\omega_{V}$ and $\omega_{U}$ denote the Euclidean measures on $V_{A,b}$ and $U_{A,b}$ respectively, normalised so that $V_{A,b}(\mathbb{Z})$ and $U_{A,b}(\mathbb{Z})$ have covolume $1$. We will write $\omega_{V,b}$ and $\omega_{U,b}$ when we need to make the dependence on $b$ explicit. 

We specify our choice of measure for $\SO_{A}$. We choose $\omega_{A}$ to be the algebraic differential defined over $\Q$ generating the rank $1$ module of top-degree left-invariant differential forms on $\SO_{A}$ constructed by Tamagawa \cite[\S 7, \S 8]{MR0212025} or Hanke \cite[\S 2]{MR2155083}.

We recall the construction of this form. Let $X$ be the space of $n\times n$ matrices and $T$ be the space of $n \times n$ symmetric matrices. Write $(x_{ij})$ for $1 \le i,j \le n$ and $(t_{ij})$ for $1 \le i \le j \le n$ for the standard coordinates on $X$ and $T$. Define $\omega_{X} = \bigwedge_{i,j} dx_{ij}$ and $\omega_T =  \bigwedge_{i \le j} dt_{ij}$. Consider the map: 
\[X \xrightarrow{\quad \mathscr{F}_{A} \quad} T ,\] defined by $x \mapsto x^{t}Ax$. Let $\iota_{\SO_{A}} \colon \SO_{A} \hookrightarrow X$ denote the injection of $\SO_{A}$ into $X$. Then, there exists a differential form $\overline{\omega}_{A}$ defined over $\Q$ such that \[\omega_{X} = \overline{\omega}_{A} \wedge \mathscr{F}_{A}^{*} \omega_{T}.\] 

We define $\omega_{A} = \iota_{\SO_{A}}^{*} \overline{\omega}_A$. Then by Tamagawa \cite[\S 7, \S 8]{MR0212025} and Hanke \cite[\S 2]{MR2155083}, $\omega_{A}$ does not depend on the choice of $\overline{\omega}_A$, and $\omega_{A}$ is a non-zero top degree left-invariant algebraic differential form on $\SO_{A}$ defined over $\Q$ with the property that $\Vol(\SO_{A}(\Z_{p}))$ computed with respect to $\omega_{A}$ is equal to the local representation density at $p$ of $A$ in the Smith--Minkowski--Siegel mass formula. 

The naturality of the construction of $\omega_{A}$ and the fact that $\Vol(\SO_{A}(\Z_{p}))$ computed with respect to $\omega_{A}$ is equal to the local representation density at $p$ of $A$ in the Smith--Minkowski--Siegel mass formula allows to carry out comparison arguments for integrals on $V_{A,b}(\mathbb{Z}_{p})$ over the various $A \in \mathscr{L}_{\Z}$. Such arguments will be useful in \S \ref{sec: 2-adic distribution} at $p=2$.

\subsection{The change of measure formula}

To compute the volumes of sets and multisets in $V_{A,b}(\mathbb{R})$ and $V_{A,b}(\mathbb{Z}_p)$, we have the following change of variable formula.

\begin{prop}[Change of measure formula]\label{change of measure formula}
Let $K =\mathbb{Z}_p$, $\mathbb{R}$ or $ \mathbb{C}$. Let $| \cdot|$ denote the usual absolute value on $K$ and let $s \colon U_{1,b}(K) \rightarrow V_{A,b}(K)$ be a continuous map such that $\pi(s(f)) = f $ for each $f \in U_{1,b}$, i.e. a continuous section of $\pi$. Then there exists a rational non-zero constant $\mathcal{J}_A$, independent of $K$, $s$, and $b$, such that for any measurable function $\phi$ on $V_{A,b}(K)$, we have: 
\begin{align} \label{eqn: change of measure 1}
&\int \limits_{\SO_A(K) \cdot s(U_{1,b}(K))} \phi(v) \, \omega_{V} = \left| \mathcal{J}_A \right| \int_{f\in U_{1,b}(K)} \int_{g \in \SO_A(K)} \phi(g \cdot s(f)) \, \omega_{A}(g) \, \omega_{U}(f) \\
\label{eqn: change of measure 2}
&\int \limits_{V_{A,b}(K)} \phi(v) \omega_{V} = \left| \mathcal{J}_A \right| \int\limits_{\substack{f \in U_{1,b}(K) \\ \Delta(f) \neq 0}} \left(\sum_{v\in \frac{V_{A,b}(K) \cap \pi^{-1}(f)}{\SO_A(K)}} \frac{1}{\# {\rm Stab}_{\SO_A(\mathbb{Z}_p)}(v)} \int\limits_{g \in \SO_A(K)} \phi(g \cdot v) \, \omega_{A}(g) \right) \, \omega_{U}(f)
\end{align}
where $\frac{V_{A,b}(K) \cap \pi^{-1}(f)}{\SO_A(K)}$ denotes a set of representatives for the action of $\SO_A(\mathbb{Z}_p)$ on $V_{A,b}(\mathbb{Z}_p) \cap \pi^{-1}(f)$. 
\end{prop}
\begin{proof}
For each $b$, \cite[Remark 3.14]{MR3272925} gives the existence of a non-zero rational constant $\mathcal{J}_{A,b} \in \Q$ independent of $K$ and $s$ and verifying \eqref{eqn: change of measure 1} and \eqref{eqn: change of measure 2} (when substituted in lieu of the $\mathcal{J_A}$ there). To prove that $\mathcal{J}_{A,b}$ is independent of $b$, we recall from the proof of Proposition \cite[Proposition 3.10]{MR3272925} that $\mathcal{J}_{A,b}$ can be realised as the Jacobian change of variables of the map 
\begin{align*}
\psi_{s,b}^A \colon \SO_A(\C) \times U_{1,b}(\C) &\rightarrow V_{A,b}(\mathbb{C}) \\
(\gamma,f) &\mapsto \gamma \cdot s(f),
\end{align*}
\textit{for any} locally analytic section $s \colon U_{1,b}(\mathbb{C}) \rightarrow V_{A,b}(\mathbb{C})$ of $\pi$. Now, $s'(f(x)) := s(f(x+b))-bA$ defines a locally analytic section $s' \colon U_{1,0}(\C) \rightarrow V_{A,0}(\C)$ of $\pi$ on the $0$ slice. In particular, $\psi_{s',0}^A$ and $\psi_{s,b}^A$ have the same Jacobian change of variables since translations do not alter the Jacobian change of variables. We conclude that $\mathcal{J}_{A,b} = \mathcal{J}_{A,0}$ for all $b$ and thus that $\mathcal{J}_{A,b}$ is independent of $b$, as desired.

\end{proof}

\subsection{Local masses} 

We introduce local masses to simplify the second integral in Proposition \ref{change of measure formula}.

\begin{defn}\label{local mass formula} Let $p$ be a prime, $f \in U_{1,b}(\mathbb{Z}_p)$ and $A \in \mathscr{L}_{\mathbb{Z}}$. We define the local mass of $f$ at $p$ in $A$, $m_p(f,A)$ to be \[m_p(f,A) := \sum_{v \in \frac{V_{A,b}(\mathbb{Z}_p) \cap \pi^{-1}(f)}{\SO_A(\mathbb{Z}_p)}} \frac{1}{\# {\rm Stab}_{\SO_A(\mathbb{Z}_p)}(v)}.\]
\end{defn}

We also define an archimedean analogue of $m_{p}(A)$, which will be helpful in \S \ref{sec: proof of main theorem}. 

\begin{defn}[The archimedean mass]
Let $A \in \mathscr{L}_\mathbb{Z}$ and $r_2$ be an integer such that $0 \le r_2 \le \frac{n-1}{2}$. The archimedean mass of $A$ with respect to $r_2$ is defined to be $m_\infty(r_2,A) = \sum_{\delta \in \mathcal{T}(r_2)} \chi_A(\delta)$. 
\end{defn}
\noindent In other words, $m_\infty(r_2,A)$ counts the number of $\SO_A(\R)$ orbits in $V_{A,b}^{r_2}(\R)$.

Note that we have $m_p(f,g_p^t A g_p) = m_p(f,A)$ for $g_p \in \SL_n(\Z_p)$ and $m_\infty(r_2,g_\infty^t A g_\infty) = m_\infty(r_2,A)$ and $g_\infty \in \SL_n(\mathbb{R})$. In particular, the local masses in $A$ at a place $p$ only depend on the $\SL_n(\Z_p)$ equivalence class of $A$ and if $A_1$ and $A_2$ are unimodular integral symmetric bilinear forms in the same genus, then $m_\infty(r_2, A_1) = m_\infty(r_2, A_2)$ and $m_p(f, A_1) = m_p(f, A_2)$ for all primes $p$.

\begin{rem}
In the rest of the article, we will abuse notation and variously write $m_p(f,[g_p])$ and $m_p(f,\mathcal{G})$ for $[g_p] \in \Sym_2(\Z_p)/\SL_n(\Z_p)$ or $\mathcal{G}$ a genus of an integral quadratic forms to indicate the common value of the mass at any form taken in these equivalence classes.
\end{rem}

The following proposition summarises the main properties of local masses. 
They follow directly from \cite[Example 2.17 and Lemma 6.5]{MR3782066}, and we omit the proof. 

\begin{prop}[Properties of local masses] \label{properties of local masses}

The local masses $m_p(f, A)$ and $m_\infty(r_2, A)$ have the following properties:
\begin{enumerate}[1)] 
\item The sum of $m_p(f, A)$ over a set of representatives for the unimodular\footnote{That is having determinant $\pm 1$.} orbits of the action of $\SL_n(\mathbb{Z}_p)$ on $\Sym_2(\mathbb{Z}_p^n)$ is: \[\sum \limits_{\substack{[A] \in \frac{\Sym_2({\mathbb{Z}_p}^n)}{\SL_n(\mathbb{Z}_p)} \\ \det(A) = \factor \in \Z_p^\times / \left(\Z_p^\times \right)^2}} m_p(f, A) = \begin{cases} 2^{n-1} & \text{ if } p = 2 \\
1 & \text{ if } p \neq 2
\end{cases}.\]
\item The sum of $m_\infty(r_2, A)$ over a set of representatives for unimodular orbits of the action of $\SL_n(\mathbb{R})$ on $\Sym_2(\mathbb{R}^n)$ is: \[\sum_{\substack{[A] \in \frac{\Sym_2(\mathbb{R}^n)}{\SL_n(\mathbb{R})} \\ \det(A) = \factor \in \R^\times/(\R^\times)^2}} m_\infty(r_2, A) = 2^{r_1-1}.\]
\end{enumerate}
\end{prop}

We will refer to the sums of the local masses in Proposition \ref{properties of local masses} as the total local masses. For $p \neq 2,\infty$ there is a unique $\SL_n(\Z_p)$-equivalence class in $\Sym_2(\Z_p^n)$ of determinant $\factor$. We immediately obtain the value of the local masses at each $A \in \mathscr{L}_\Z$ for all $p \neq 2,\infty$: they are the same as the total local mass in Proposition \ref{properties of local masses} 1). 

\begin{cor}[Local masses for $p \neq 2,\infty$] For $A \in \mathscr{L}_\mathbb{Z}$  and $p \neq 2,\infty$ we have \[m_p(f,A) = 1.\]
\end{cor}

Determining how the total local mass $2^{n-1}$ and $2^{r_{1}-1}$ at $p=2$ and $p=\infty$ distributes over the many $\SL_{n}(\Z_{2})$ and $\SL_{n}(\R)$ equivalence classes of elements in $\mathscr{L}_{\Z}$ is much more delicate and occupies the following two sections. \\

Plugging Definition \ref{local mass formula} into Proposition \ref{change of measure formula} gives us the following simple form for the local volumes appearing in Theorem \ref{first asymptotic formula}. 

\begin{prop} \label{prop: formula for local volumes}

Let $S_{p,b} \subset U_{1,b}(\mathbb{Z}_p)$ be a non-empty open set whose boundary has measure $0$. Consider the set $\Lambda_{A,b,p} = V_{A,b}(\mathbb{Z}_p) \cap \pi^{-1}(S_p)$. Then we have \[\Vol( \Lambda_{A,b,p}) = \left| \mathcal{J}_A \right|_p \Vol(\SO_A(\mathbb{Z}_p)) \int_{f \in S_{p,b}} m_p(f,A) \, \omega_{U}(f).\] 
We also have \[\Vol\left( \mathcal{F}_A \cdot R_{A,b}^{r_2,\delta} (X) \right) =  \chi_A(\delta) \left| \mathcal{J}_A \right| \Vol( \mathcal{F}_A) \Vol(U_{1,b}^{r_2}(\mathbb{R})_{H < X}).\]
Here, we take volumes with respect to the measure $\omega_{A}$ defined in \S \ref{subsec: choice of volume forms}. 
\end{prop}

\section{$2$-adic mass distribution} \label{sec: 2-adic distribution}

Computing the $2$-adic and the archimedean local masses is more delicate than at odd primes. In this section, we determine the average value of $m_{2}(f,A)$ over subsets of $S_{2,b} \subset U_{1,b}(\Z_{2})$ defined by mod $2$ conditions. 

The values $m_2(f, A)$ only depend on the $\SL_n(\Z_2)$ equivalence class of $A$, and we will compute their averages for convenient representatives. There are exactly two integer-matrix quadratic forms over $\mathbb{Z}_2$ of determinant $(-1)^{\frac{n-1}{2}}$ and odd dimension $n$ up to $\SL_n(\mathbb{Z}_2)$ equivalence. They are distinguished by their Hasse--Witt invariants, but are equivalent modulo $2$, see for instance \cite[Lemmas 1, 2 and 3]{MR12640}! 

To fix ideas, choose representatives of these forms, $\mathfrak{M}_1$ and $\mathfrak{M}_{-1}$, which coincide modulo $2$ and whose subscripts indicate the value of the Hasse--Witt invariant: For $n \equiv 1 \pmod 4$ we take $\mathfrak{M_1} = \diag(1,\ldots,1,1,1,1)$ and $\mathfrak{M}_{-1} = \diag(1,1,\ldots,1,1,-1,-1)$, while for $n \equiv 3 \pmod 4$, we take $\mathfrak{M_1} = \diag(1,\ldots,1,1,1,-1)$ and $\mathfrak{M}_{-1} = \diag(1,1,\ldots,1,-1,-1,-1)$.

Let $S_{2,b} \subset U_{1,b}(\mathbb{Z}_2)$ and let $c_{2}(S_{2,b},\mathfrak{M}_1)$ and $c_{2}(S_{2,b},\mathfrak{M}_{-1})$ denote the averages of the $2$-adic masses $m_{2}(f,\mathfrak{M}_1)$ and $m_{2}(f,\mathfrak{M}_{-1})$ over $S_{2,b}$:
\[
c_2(S_{2,b},\mathfrak{M}_1) := \frac{\int_{f \in S_{2,b}} m_2(f,\mathfrak{M}_1) \, df}{\Vol(S_{2,b})} \hspace{5ex} \text{ and } \hspace{5ex}
c_2(S_{2,b},\mathfrak{M}_{-1}) := \frac{\int_{f \in S_{2,b}} m_2(f,\mathfrak{M}_{-1})}{\Vol(S_{2,b})} \, df.
\]

We obtain the precise values of these averages under the assumption that $S_{2,b}$ is defined by mod $2$ conditions in the following proposition.

\begin{prop} \label{prop: mod 2 masses average}
For $S_{2,b} \subset U_{1,b}(\Z_{2})$ defined modulo $2$, we have $$c_2(S_{2,b},\mathfrak{M}_1) + c_2(S_{2,b},\mathfrak{M}_{-1}) = 2^{n-1} \hspace{5ex} \text{ and } \hspace{5ex} c_2(n,\mathfrak{M}_1) - c_2(n,\mathfrak{M}_{-1}) = \pm_8 2^{\frac{n-1}{2}}$$
where $\pm_8$ is $+$ if $n$ is congruent to $1,3 \pmod 8$ and $-$ otherwise. In particular, we find $$c_2(S_{2,b},\mathfrak{M}_1) = \frac{1}{2}\left(2^{n-1} \pm_8 2^{\frac{n-1}{2}} \right) \hspace{5ex} \text{ and } \hspace{5ex} c_2(S_{2,b},\mathfrak{M}_{-1}) = \frac{1}{2}\left(2^{n-1} \mp_8 2^{\frac{n-1}{2}} \right).$$
\end{prop}

When $f$ is maximal, we obtain the precise values of $m_2(f,\mathfrak{M}_1)$ and $m_2(f,\mathfrak{M}_{-1})$.

\begin{prop} \label{prop: mod 2 masses pointwise}
If $f \in U_1(\Z_2)$ is maximal, we find \[m_2(f,\mathfrak{M}_1) = \frac{1}{2}\left(2^{n-1} \pm_8 2^{\frac{n-1}{2}} \right) \hspace{5 ex} \text{ and } \hspace{5ex} m_2(f,\mathfrak{M}_{-1}) = \frac{1}{2}\left(2^{n-1} \mp_8 2^{\frac{n-1}{2}} \right)\] where $\pm_8$ is $+$ if $n$ is congruent to $1,3 \pmod 8$ and $-$ otherwise. In particular, for $S_{2,b} \subset U_{1,b}(\Z_{2})$ consists of forms which are maximal at $2$ we have
$$c_2(S_{2,b},\mathfrak{M}_1) = \frac{1}{2}\left(2^{n-1} \pm_8 2^{\frac{n-1}{2}} \right) \hspace{5ex} \text{ and } \hspace{5ex} c_2(S_{2,b},\mathfrak{M}_{-1}) = \frac{1}{2}\left(2^{n-1} \mp_8 2^{\frac{n-1}{2}} \right).$$
\end{prop}

\subsection{Equality of Jacobians and proof of Proposition \ref{prop: mod 2 masses average}} 

By Proposition \ref{properties of local masses}, we have $m_2(f,\mathfrak{M}_1)+m_2(f,\mathfrak{M}_{-1}) = 2^{n-1}$ for any $f \in U_{1,b}(\mathbb{Z}_2)$. Thus \[c_2(S_{2,b},\mathfrak{M}_1) + c_2(S_{2,b},\mathfrak{M}_{-1}) = 2^{n-1}.\] In order to determine the individual values of $c_2(n,\mathfrak{M}_1)$ and $c_2(n,\mathfrak{M}_{-1})$, it is thus sufficient to find their ratio. 

Let $\Lambda_{\mathfrak{M}_i}(2) := \pi^{-1}(S_{2,b}) \cap V_{\mathfrak{M}_{i},b}(\Z_{2})$ for $i = 1,-1$ be the preimage in $V_{\mathfrak{M}_{i},b}(\Z_{2})$ of $S_{2,b}$ under the resolvent map. Recall the following expressions, which were a corollary of Proposition \ref{prop: formula for local volumes}:
\begin{align*}
\Vol(\Lambda_{\mathfrak{M}_i}(2)) &=  \left| \mathcal{J}_{\mathfrak{M}_i} \right|_2 \Vol(\SO_{\mathfrak{M}_i}(\mathbb{Z}_2)) \int_{f \in S_{2,b}} m_2(f,\mathfrak{M}_i) \, df \\ 
&= c_2(S_{2,b},\mathfrak{M}_i) \Vol(S_{2,b}) \left| \mathcal{J}_{\mathfrak{M}_i} \right|_2 \Vol(\SO_{\mathfrak{M}_i}(\mathbb{Z}_2)).
\end{align*}

Now assume that $S_{2,b} \subset U_{1,b}(\Z_{2})$ is defined by mod $2$ conditions. We now obtain the ratio of $c_2(S_{2,b},\mathfrak{M}_1)$ to $c_2(S_{2,b},\mathfrak{M}_{-1})$ by arguing in three steps:
\begin{itemize}
\item showing that $\Vol(\Lambda_{\mathfrak{M}_1}(2)) = \Vol(\Lambda_{\mathfrak{M}_{-1}}(2))$;
\item showing that $\mathcal{J}_{\mathfrak{M}_{1}} =   \mathcal{J}_{\mathfrak{M}_{-1}}$; and,
\item calculating the ratio of the volumes $\Vol(\SO_{\mathfrak{M}_1}(\mathbb{Z}_2))$ and $\Vol(\SO_{\mathfrak{M}_{-1}}(\mathbb{Z}_2))$.
\end{itemize} Together, this will yield the value of the ratio of $c_2(\mathfrak{M}_1)$ and $c_2(\mathfrak{M}_{-1})$, and hence their values. \\

We show that volumes $\Vol(\Lambda_{\mathfrak{M}_1}(2)) = \Vol(\Lambda_{\mathfrak{M}_{-1}}(2))$.

\begin{lem} [Point count]
Let $S_{2,b} \subset U_{1,b}(\mathbb{Z}_2)$ be a local condition on the space of monic polynomials at the prime $2$ defined modulo $2$. Denote by $\Lambda_{\mathfrak{M}_1,b}(2)$ and $\Lambda_{\mathfrak{M}_1,b}(2)$ the preimages in $V_{\mathfrak{M}_1,b}(\mathbb{Z}_2)$ and $V_{\mathfrak{M}_{-1},b}(\mathbb{Z}_2)$ respectively of $S_{2,b}$ under the resolvent map $\pi$. Then we have an equality of volumes  \[\Vol(\Lambda_{\mathfrak{M}_1,b}(2)) = \Vol(\Lambda_{\mathfrak{M}_{-1},b}(2)).\]
\end{lem}
\begin{proof}
The representatives $\mathfrak{M}_1$ and $\mathfrak{M}_{-1}$ are equal modulo $2$. Thus, $V_{\mathfrak{M}_1,b}(\mathbb{Z}_2)$ and $V_{\mathfrak{M}_{-1},b}(\mathbb{Z}_2)$ are equal modulo $2$. Since $\Lambda_{\mathfrak{M}_1}(2)$ and $\Lambda_{\mathfrak{M}_{-1}}(2)$ are defined by imposing congruence conditions mod $2$ on $V_{\mathfrak{M}_1,b}(\mathbb{Z}_2)$ and $V_{\mathfrak{M}_{-1},b}(\mathbb{Z}_2)$ the result follows. 
\end{proof}

We show that $\mathcal{J}_{\mathfrak{M}_{1}} = \mathcal{J}_{\mathfrak{M}_{-1}}$.

\begin{lem}[Equality of Jacobians]
Using the notation of Proposition \ref{change of measure formula}, we have $$\mathcal{J}_{\mathfrak{M}_{1}} =   \mathcal{J}_{\mathfrak{M}_{-1}} \in \Q$$ In particular, their $2$-adic valuations coincide, $\left| \mathcal{J}_{\mathfrak{M}_{1}} \right|_2 = \left| \mathcal{J}_{\mathfrak{M}_{-1}} \right|_2.$
\end{lem}
\begin{proof}
Recall from \cite[Proposition 3.10]{MR3272925}, that the proof of Proposition \ref{change of measure formula} over $K = \mathbb{Z}_p, \mathbb{R}$, or $\mathbb{C}$ proceeds by proving that over $\mathbb{C}$ the identity 
\begin{equation*} 
\int \limits_{\SO_A(\mathbb{C}) \cdot s(U_{1,b}(\mathbb{C}))} \phi(v) \, \omega_{V} = \left| \mathcal{J}_A \right| \int_{f\in U_{1,b}(\mathbb{C})} \int_{\gamma \in \SO_A(\mathbb{C})} \phi(\gamma \cdot s(f)) \, \omega(\gamma) \, \omega_{U}(f)
\end{equation*}
holds for some non-zero rational number $\mathcal{J}_A \in \Q^\times$. The principle of permanence of identities then gives the result for $K = \mathbb{Z}_p, \mathbb{R}$, or $\mathbb{C}$ with the same $\mathcal{J}_A$. 

$\mathcal{J}_A$ can be realized as follows. Consider the map $\psi_s^A \colon \SO_A(\mathbb{C}) \times U_{1,b}(\mathbb{C}) \rightarrow V_b(\mathbb{C})$ given by $\psi_s^A(\gamma,f) = \gamma \cdot s(f)$ for any locally analytic section $s \colon U_{1,b}(\mathbb{C}) \rightarrow V_{A,b}(\mathbb{C})$. Then $\mathcal{J}_{A}$ satisfies \[\mathcal{J}_{A} \omega_{A} \wedge \omega_{U} = (\psi_s^A)^*(\omega_{V}).\] 

Thus, the lemma will follow from comparing the Jacobian change of variables of the maps $\psi_s^{\mathfrak{M}_1}$ and $\psi_s^{\mathfrak{M}_{-1}}$. Now, fix a matrix $g \in \SL_n^{\pm}(\mathbb{C})$ such that $\mathfrak{M}_{-1} = g^t \mathfrak{M}_{1} g$. Consider the map $\sigma_g \colon \SO_{\mathfrak{M}_1}(\mathbb{C}) \rightarrow \SO_{\mathfrak{M}_{-1}}(\mathbb{C})$
defined by $\sigma_g(h) = g^{-1} h g$. Now, fix an analytic section $s_1 \colon U_{1,b}(\mathbb{C}) \rightarrow V_{\mathfrak{M}_1,b}(\mathbb{C})$ and define the analytic section $s_{-1} \colon U_{1,b}(\mathbb{C}) \rightarrow V_{\mathfrak{M}_{-1},b}(\mathbb{C})$ to be $s_{-1} := g \cdot s_1$. The following diagram commutes
\begin{center}
\begin{tikzcd} [row sep = large, column sep = large]
\SO_{\mathfrak{M}_1}(\mathbb{C}) \times U_{1,b}(\mathbb{C}) \arrow[r, "\psi_s^{\mathfrak{M}_1}" above] \arrow[d,"\sigma_g \times {\rm id}" left ]& V_{\mathfrak{M}_{1},b}(\mathbb{C}) \arrow[d, "g \, \cdot " right] \\
\SO_{\mathfrak{M}_{-1}}(\mathbb{C}) \times U_{1,b}(\mathbb{C}) \arrow[r, "\psi_{s_{-1}}^{\mathfrak{M}_{-1}}" below]& V_{\mathfrak{M}_{-1},b}(\mathbb{C})
\end{tikzcd}.
\end{center}
Computing the pullback of the form $\omega_{V}$ from the lower right-hand corner of the diagram to the upper left-hand corner using the two different paths and the fact $(g \cdot)^{*}\omega_{V} = \omega_{V}$ gives
\[\mathcal{J}_{\mathfrak{M}_{-1}} (\sigma_{g}^{*}\omega_{\mathfrak{M}_{-1}}) \wedge \omega_{U} = \mathcal{J}_{\mathfrak{M}_{1}} \omega_{\mathfrak{M}_{1}} \wedge \omega_{U} .\]
It is thus sufficient to prove that $\sigma_{g}^{*}\omega_{\mathfrak{M}_{-1}} = \frac{1}{\det(g)^{n+1}}\omega_{\mathfrak{M}_{1}}$ (since $n$ is odd and $\det g = \pm 1$). Use the notation introduced in \S \ref{sec: product of local volumes} to define the volume forms $\omega_{A}$ and notice that $\sigma_{g}$ is really the restriction of a map $\sigma_{g} \colon X \rightarrow X$ defined by $\sigma_g(h) = g^{-1} h g$. Then it is easy to check that $\sigma_{g}^{*} \omega_{X} = \omega_{X}$ and that the following diagram commutes: 
\begin{center}
\begin{tikzcd}
X \arrow[dd, "\sigma_g"] \arrow[rrd, "g^t \mathscr{F}_{\mathfrak{M}_1}g"] &  &   \\
                                                                          &  & T \\
X \arrow[rru, "\mathscr{F}_{\mathfrak{M}_{-1}}"']                         &  &  
\end{tikzcd}
\end{center}
By definition, we have $\omega_{X} = \overline{\omega}_{\mathfrak{M}_{-1}} \wedge \mathscr{F}_{\mathfrak{M}_{-1}}^{*}\omega_{T}$. Pulling back this equation via $\sigma_{g}$ and using the diagram above gives
\begin{align*}
\omega_{X} = \sigma_{g}^{*}(\omega_{X}) &= (\sigma_{g}^{*}\overline{\omega}_{\mathfrak{M}_{-1}}) \wedge (\sigma_{g}^{*}\mathscr{F}_{\mathfrak{M}_{-1}}^{*}\omega_{T}) \\
&= (\sigma_{g}^{*}\overline{\omega}_{\mathfrak{M}_{-1}}) \wedge (\mathscr{F}_{\mathfrak{M}_{1}}^{*} (g^{t} (\cdot) g)^{*}\omega_{T}) \\ 
&= (\sigma_{g}^{*}\overline{\omega}_{\mathfrak{M}_{-1}}) \wedge (\det g)^{n+1} \mathscr{F}_{\mathfrak{M}_{1}}^{*} \omega_{T}
\end{align*} where the third line follows from Hanke \cite[proof of Lemma 2.2]{MR2155083}. Pulling back $(\det g)^{n+1} \sigma_{g}^{*}\overline{\omega}_{\mathfrak{M}_{-1}}$ via $\iota_{\SO_{\mathfrak{M}_{1}}}$ and using independence of $\omega_{\mathfrak{M}_{1}}$ on choices completes the proof. 
\end{proof}

We find the ratio of $\Vol(\SO_{\mathfrak{M}_1}(\mathbb{Z}_2))$ and $\Vol(\SO_{\mathfrak{M}_{-1}}(\mathbb{Z}_2))$ computed with respect to $\omega_{\mathfrak{M}_{1}}$ and $\omega_{\mathfrak{M}_{-1}}$ by using the values of the local representation densities at $p=2$ from the Smith--Minkowski--Siegel mass formula.

\begin{prop} We have
\begin{align*}
\frac{c_2\left(n,\mathfrak{M}_1\right)}{c_2 \left(n,\mathfrak{M}_{-1}\right)} = \frac{\Vol(\Lambda_{\mathfrak{M}_1,b}(2)) \Bigg(\left| \mathcal{J}_{\mathfrak{M}_{-1}} \right|_2 \Vol(\SO_{\mathfrak{M}_{-1}}(\mathbb{Z}_2)) \Bigg)}{\Vol(\Lambda_{\mathfrak{M}_{-1},b}(2)) \Bigg(\left| \mathcal{J}_{\mathfrak{M}_1} \right|_2 \Vol(\SO_{\mathfrak{M}_1}(\mathbb{Z}_2)) \Bigg)} &= \frac{\Vol(\SO_{\mathfrak{M}_{-1}}(\mathbb{Z}_2))}{\Vol(\SO_{\mathfrak{M}_1}(\mathbb{Z}_2))} \\
&=\frac{2^{n-1} \pm_8 2^{\frac{n-1}{2}}}{2^{n-1} \mp_8 2^{\frac{n-1}{2}}},
\end{align*}
where $\pm_8$ is $+$ if $n$ is congruent to $1, 3 \pmod 8$ and $-$ otherwise.
\end{prop}

\begin{proof}
The first two equalities above follow directly from the preceding two lemmas. 

The value of the ratio of the volumes of $\SO_{\mathfrak{M}_{-1}}(\mathbb{Z}_2)$ and $\SO_{\mathfrak{M}_{1}}(\mathbb{Z}_2)$ is equal to the ratio of the $2$-adic representation densities of the lattices defined by $\mathfrak{M}_{1}$ and $\mathfrak{M}_{-1}$ respectively. 

The ratio of the $2$-adic densities of the lattices defined by $\mathfrak{M}_{1}$ and $\mathfrak{M}_{-1}$, is in this case equal to the ratio of the diagonal factors in the language of \cite[\S 4]{MR965484}, which are denoted by $M_2(\mathfrak{M}_{1})$ and $M_2(\mathfrak{M}_{-1})$ respectively in \textit{loc. cit.}. \footnote{Note that this is distinct from the $2$-adic mass! We use this notation only in this proof and only to follow \cite{MR965484} closely.} There are general formulae for computing the value of the diagonal factors at every prime, which are a bit subtle in the case of $p=2$. To find the values, we need to determine the freeness, oddity, and octane values of the forms $\mathfrak{M}_{1}$ and $\mathfrak{M}_{-1}$.  

$\mathfrak{M}_{1}$ and $\mathfrak{M}_{-1}$ are already in $2$-adic Jordan form. Note that $\mathfrak{M}_{1}$ is free, odd, and has octane value $1 \pmod 8 $ if $n \equiv 1, 3 \pmod 8$ and $5 \pmod 8$ if $n \equiv 5,7 \pmod 8$. On the other hand, $\mathfrak{M}_{-1}$ is free, odd, but has octane value $5 \pmod 8 $ if $n \equiv 1, 3 \pmod 8$ and $1 \pmod 8$ if $n \equiv 5,7 \pmod 8$. Applying the formulae for the diagonal factor
\cite[\S 5 Equation (5)]{MR965484}, we find
\[M_2(\mathfrak{M}_{1}) = \frac{1}{2 \prod_{i = 1}^{\frac{n-1}{2}-1}(1-2^{-2i})} \frac{1}{1 \mp_8 2^{-\frac{n-1}{2}}} \hspace{3ex} \text{ and } \hspace{3ex} M_2(\mathfrak{M}_{-1}) = \frac{1}{2\prod_{i = 1}^{\frac{n-1}{2}-1}(1-2^{-2i})} \frac{1}{1 \pm_8 2^{-\frac{n-1}{2}}}, \]
completing the proof of the proposition.
\end{proof}

\subsection{Quadratic refinement $\eta$ and proof of Proposition \ref{prop: mod 2 masses pointwise}} 

Let $f \in U_1(\Z_2)$ be maximal and have $m+1$ different irreducible factors. 

The following fact can be gathered from \cite[\S 6.2]{MR3156850}: There exists an even quadratic form $$\eta' \colon (R_f^\times/(R_f^\times)^2)_{N \equiv 1} \rightarrow \F_2$$ such that 1) $\eta'$ has a maximal isotropic subspace of dimension $m + (n-1)/2$, and 2) the zeros of $\eta'$ in $(R_f^\times/(R_f^\times)^2)_{N \equiv 1}$ correspond to orbits of $\SL_n(\Z_2)$ on $V(\Z_2) \cap \pi^{-1}(f)$ whose first component is equivalent to $\mathfrak{M}_1$ for $n \equiv 1,3 \pmod{8}$ and $\mathfrak{M}_{-1}$ for $n \equiv 5,7 \pmod{8}$. Indeed, in the notation of \textit{loc. cit.}, $\eta'$ is the restriction of $\eta$ to $(R_f^\times/(R_f^\times)^2)_{N \equiv 1}$ and $J(f)/2J(f)$ is the maximal isotropic subspace. 

\begin{proof}[Proof of Proposition \ref{prop: mod 2 masses pointwise}]
It thus remains to find the number of zeroes of this even quadratic form $\eta'$. 

Since $\eta'$ is even, it has the form $\eta'(x_0,\cdots,x_R) = x_0 x_1 + \ldots + x_{t-1} x_t$ with respect to some $\F_2$ basis of $(R_f^\times/(R_f^\times)^2)_{N \equiv 1}$ \cite[Theorem 2.1]{MR1680530}, where $R = \dim_{\F_2} (R_f^\times/(R_f^\times)^2)_{N \equiv 1} - 1 = m + n - 2$. The number of zeroes of $\eta'$ is then equal to $2^{R-t}(2^t + 2^{\frac{t-1}{2}})$ \cite[Theorem 2.7]{MR1680530}. 

We can figure $t$ out by computing the dimension of the maximal isotropic subspace of $\eta'$ in two ways: $m+(n-1)/2 = (t+1)/2 + (R+1)-(t+1)$. This implies that $t = 2(R+1 - m - (n-1)/2)) - 1 = n-2$. 

Because the stabilizer of any element in $V(\Z_2) \cap \pi^{-1}(f)$ has size $|R_f^\times[2]|_{N \equiv 1} = 2^m$, it follows that
$$m_2(f,\mathfrak{M}_1) = \frac{1}{2^m}\cdot 2^m(2^{n-2} \pm_8 2^{\frac{n-3}{2}}) = \frac{1}{2} (2^{n-1} \pm_8 2^\frac{n-1}{2})$$
and, consequently, $m_2(f,\mathfrak{M}_{-1}) = \frac{1}{2} (2^{n-1} \mp_8 2^\frac{n-1}{2})$, because of Proposition \ref{properties of local masses} 1). \end{proof}

\section{Archimedean mass distribution}\label{sec: archimedean distribution}

It now remains to determine the value of the archimedean local mass $m_{\infty}(r_2,A) = \sum_{\delta \in \mathcal{T}(r_{2})} \chi_{A}(\delta)$ for each $A \in \mathscr{L}_\Z$. This mass depends only on the real signature of $A$, and we obtain its precise value in the following proposition.

\begin{prop}[Distribution of the archimedean mass] \label{prop: distribution of archimedean mass}
Let $A \in \mathscr{L}_{\mathbb{Z}}$ and $0 \le r_2 \le (n-1)/2$ and $n = r_1+2r_2$. Suppose that $A$ has $q$ negative eigenvalues. Then the archimedean mass of $A$ with respect to $r_2$ is given by \[m_\infty(r_2,A) = \sum_{\delta \in \mathcal{T}(r_2)} \chi_A(\delta) = {r_1 \choose q-r_2}.\] In particular, if $q < r_2$ then $\chi_A(\delta)=0$ for all $\delta \in \T(r_2)$.
\end{prop}

Recall from \S \ref{sec: product of local volumes} that $m_\infty(r_2,A)$ only depends on the $\SL_n(\R)$-orbit of $A$. In view of our application of the classification theorem for genera of integral symmetric bilinear forms in organising the calculation \S \ref{sec: proof of main theorem}, it will be useful to know the value the sums \[c_{\infty,1}(r_2) := \sum\limits_{\substack{[q_{\infty}] \in \frac{\Sym_2(\R^n)}{\SL_n(\R)} \\ \det(q_{\infty}) = \factor \\ e_\infty(q_{\infty}) = 1}}  m_{\infty}(r_2,q_{\infty}) \hspace{5 ex} \text{ and } \hspace{5 ex} c_{\infty,-1}(r_2) := \sum_{\substack{[q_{\infty}] \in \frac{\Sym_2(\R^n)}{\SL_n(\R)} \\ \det(q_{\infty}) = \factor \\ e_\infty(q_{\infty}) = -1}}  m_{\infty}(r_2,q_{\infty})\] of the masses $m_{\infty}(r_{2},\cdot)$ over all $\SL_n(\R)$-equivalence classes of real symmetric matrices of determinant $\factor$ which have Hasse--Witt symbol of $1$ and $-1$ respectively.

\begin{prop}\label{prop: sum archimedean masses}

The quantities $c_{\infty,1}(r_2)$ and $c_{\infty,-1}(r_2)$ satisfy the equations $$c_{\infty,1}(r_2) + c_{\infty,-1}(r_2) = 2^{r_1-1} \hspace{5ex} \text{ and } \hspace{5ex} c_{\infty,1}(r_2) - c_{\infty,-1}(r_2) = \pm_8 2^{\frac{r_1-1}{2}},$$
where $\pm_8$ is $+$ if $n$ is congruent to $1,3 \pmod{8}$ and $-$ otherwise. In particular, 
$$c_{\infty,1}(r_2) = \frac{1}{2}\left(2^{r_1-1} \pm_8 2^{\frac{r_1-1}{2}} \right) \hspace{5ex} \text{ and } \hspace{5ex} c_{\infty,-1}(r_2) = \frac{1}{2}\left(2^{r_1-1} \mp_8 2^{\frac{r_1-1}{2}} \right).$$ 

Let $\delta_{\gg 0} := (11\cdots 1 1) \in \T(r_2)$. If $n$ is congruent to $1,3 \pmod 8$, then every $A \in \mathscr{L}_\Z$ with $\chi_A(\delta_{\gg 0}) = 1$ has $e_\infty(A) = 1$. If $n$ is congruent to $5,7 \pmod 8$ then all such $A$ have $e_{\infty}(A) = -1$.
\end{prop}

\subsection{Proof of Proposition \ref{prop: distribution of archimedean mass}} 

The idea is to use the alternative description of the $\SO_A(\mathbb{R})$ orbits on $V_A(\mathbb{R}) \bigcap \pi^{-1}(f)$ in terms of $\SO_{W_A}$-orbits of self-adjoint operators on the bilinear space $W_A$ of rank $n$ with associated Gram matrix $A$ from Lemma \ref{self-adjoint operators}.  \footnote{$\W_{A}$ is simply defined as the space $\R^{n}$ with bilinear form given in standard coordinates by the matrix $A$.} The spectral theorem gives a classification of these orbits in which it is easy to count. 

\begin{proof}[Proof of Proposition \ref{prop: distribution of archimedean mass}] \footnote{This extends the reasoning of Bhargava--Gross \cite[\S 6.3]{MR3156850} from the totally split case to the case of general $A \in \GL_n(\R)$.}
Let $\W_A$ be the bilinear space of rank $n$ with associated Gram matrix $A$. $\W_A$ has signature $(n-q,q)$. Fix an $f \in U_1(\R)$ with $r_1$ real roots and $r_2$ pairs of complex roots. By Lemma \ref{self-adjoint operators}, $m_\infty(r_2,A)$ is the number of $\SO_{\W_A}(\R)$-conjugacy classes of self-adjoint linear operators with characteristic polynomial $f$. Suppose that $T$ is a self-adjoint operator on $\W_A$ with characteristic polynomial $f$. By the spectral theorem: This operator has $r_1$ distinct eigenspaces of dimension $1$, $r_2$ distinct $ T$-stable subspaces of dimension $2$ (corresponding to pairs of complex conjugate eigenvalues), and the bilinear space $\W_A$ decomposes as an orthogonal direct sum of these subspaces. Since each of the $2$-dimensional eigenspaces has signature $(1,1)$, we see that $n-q \ge r_2$ and $q \ge r_2$, and that the number of $1$-dimensional eigenspaces of $T$ which are negative definite is $q - r_2$. This subset of $q-r_2$ negative definite $1$-dimensional eigenspaces for $T$ determines its $\SO_{\W_A}(\R)$-orbit. Conversely, such a choice can represent any $\SO_{\W_A}(\R)$-conjugacy class of self-adjoint linear operators with characteristic polynomial $f$. We conclude that $m_\infty(r_2,A) = {r_1 \choose q-r_2}$. \end{proof}

\subsection{Proof of Proposition \ref{prop: sum archimedean masses}}

The proof is to sum Proposition \ref{prop: distribution of archimedean mass}. To this end, we will need identities concerning sums of binomial coefficients.

\begin{lem}\label{binomial sum}
Let $\ell$ be an odd number. Then 
\[\sum_{k=0}^{\frac{\ell-1}{2}}  {\ell \choose 2k} = 2^{\ell-1} \hspace{5ex}  \text{and} \hspace{5ex} \sum_{k=0}^{\frac{\ell-1}{2}} (-1)^k {\ell \choose 2k} = \pm_8' 2^{\frac{\ell-1}{2}} \] where $\pm_8'$ is $+$ if $\ell$ is congruent to $1,7 \pmod 8$ and $-$ otherwise. Furthermore, we have \[\sum_{k=0}^{\frac{\ell-1}{2}}  {\ell \choose 2k+1} = 2^{\ell-1} \hspace{5ex}  \text{and} \hspace{5ex} \sum_{k=0}^{\frac{\ell-1}{2}} (-1)^k {\ell \choose 2k+1} = \pm_8 2^{\frac{\ell-1}{2}}\] where $\pm_8$ is $+$ if $\ell$ is congruent to $1,3 \pmod 8$ and $-$ otherwise.  
\end{lem}
\begin{proof}
    These identities follow from the binomial theorem. Indeed, for the first one, note that $2^\ell = (1+1)^\ell = \sum_{k=0}^{\frac{\ell-1}{2}} {\ell \choose 2k} + \sum_{k=0}^{\frac{\ell-1}{2}} {\ell \choose 2k+1}$ while $0 = (1-1)^\ell = \sum_{k=0}^{\frac{\ell-1}{2}} {\ell \choose 2k} - \sum_{k=0}^{\frac{\ell-1}{2}} {\ell \choose 2k+1}$. Thus $\sum_{k=0}^{\frac{\ell-1}{2}}  {\ell \choose 2k} = 2^{\ell-1}$. For the second one, examine $(1+i)^\ell$ in the complex plane. For notational clarity, let us write $z = (1+i)^\ell$. On the one hand, $z = \sum_{k=0}^\ell {\ell \choose k} i^k =  ( \sum_{k=0}^{\frac{\ell-1}{2}} (-1)^k {\ell \choose 2k} ) + i( \sum_{k=0}^{\frac{\ell-1}{2}} (-1)^{k}{\ell \choose 2k+1}).$ On the other hand, $z = 2^{\frac{\ell}{2}}e^{i\frac{\ell\pi}{4}}$ since $1+i = \sqrt{2}e^{i\frac{\pi}{4}}$ and thus we find $\mathfrak{Re}(z) = \pm \mathfrak{Im}(z)$ and $2^\ell = \mathfrak{Re}(z)^2 + \mathfrak{Im}(z)^2$. Therefore, $\sum_{k=0}^{\frac{\ell-1}{2}} (-1)^k {\ell \choose 2k} = \pm 2^{\frac{\ell-1}{2}}$. Note that $\mathfrak{Re}(z)>0$ if and only if $\ell \equiv 0,1,7 \pmod 8$, while $\mathfrak{Im}(z) > 0 $ if and only if $\ell \equiv 1,2,3 \pmod 8$. Thus, $\sum_{k=0}^{\frac{\ell-1}{2}} (-1)^k {\ell \choose 2k} = \pm_8' 2^{\frac{\ell-1}{2}}$ and $\sum_{k=0}^{\frac{\ell-1}{2}} (-1)^k {\ell \choose 2k+1} = \pm_8 2^{\frac{\ell-1}{2}}$ as desired.
\end{proof}

Being careful with signs, we obtain the value of $c_{\infty,1}(r_2)$ and $c_{\infty,-1}(r_2)$.

\begin{proof}[Proof of Proposition \ref{prop: sum archimedean masses}]
The first two statements follow at once from Proposition \ref{prop: distribution of archimedean mass} and Lemma \ref{binomial sum}. The last statement follows from the fact that the distinguished orbit has first component equal to matrix with $1$'s on the anti-diagonal by \cite[\S 4.1]{MR3156850} or by a direct computation from Theorem \ref{integral parametrisation} using the basis $\langle1,\theta,\ldots,\theta^{\frac{n-1}{2}}, \zeta_{\frac{n-1}{2}+1},\ldots,\zeta_{n-1} \rangle$ of $I_f((n-3)/2)$.
\end{proof}

\section{{Proof of Theorem \ref{main theorem} and Corollary \ref{main corollary}}} \label{sec: proof of main theorem}

We prove Theorem \ref{main theorem} by putting together the elements developed in the previous sections and briefly explain how to obtain Corollary \ref{main corollary}.

\subsection{Proof of Theorem \ref{main theorem}}

Let $\mathscr{L}_{\mathbb{Z}}$ be a set of representatives for $\SL_n(\Z)$ classes of integral $n$-ary quadratic forms of determinant $\factor$ and $\mathcal{G}_{\mathbb{Z}}$ be the set of genera of integral $n$-ary quadratic forms of determinant $\factor$. $\mathcal{G}_{\mathbb{Z}}$ partitions $\mathscr{L}_{\mathbb{Z}}$. 

Let $\Sigma$ denote either an acceptable family of monogenised fields or a very large family of monogenised orders with local conditions at $2$ given modulo $2$, associated to the collection $S = \bigsqcup_{0 \le b < n} \prod_{p} S_{b,p}$ of local specifications. For each $A,b,\delta$, let $\Lambda_{A,b}^{\delta}$ be the corresponding collection of projective local specifications on $V_{A,b}^{r_{2},\delta}$ associated to $\prod_{p} S_{b,p}$ via the resolvent map $\pi$. Then $\V(\Lambda_{A,b}^{\delta})$ is very large (resp. acceptable) if $\Sigma$ is very large (resp. acceptable). Proving Theorem \ref{main theorem} amounts to estimating the ratio
\[
\frac{\sum\limits_{\substack{\mathcal{O} \in \Sigma, \\ H_{\text{naive}}(\mathcal{O}) < X}} 2^{r_1+r_2-1} \left| \Cl(\mathcal{O})[2] \right|-\left| \mathcal{I}_2(\mathcal{O}) \right|}{\sum\limits_{\substack{\mathcal{O} \in \Sigma \\ H_{\text{naive}}(\mathcal{O}) < X}} 1}.
\]
By the remark at the end of \S \ref{subsec: absolutely irreducible elements and N(L;X)} and the subsequent sections, we find:
\begin{flalign*}
&= \frac{\sum\limits_{0 \le b <n} \sum\limits_{\delta \in \mathcal{T}(r_2)} \sum\limits_{A \in \mathscr{L}_{\mathbb{Z}}} N(\mathcal{V}(\Lambda_{A,b}^\delta),X)}{\left(\sum\limits_{0 \le b <n} \Vol(U_{1,b}^{r_2}(\mathbb{R}))_{<X})  \prod\limits_p \Vol(S_{b,p})\right) + o\left(X^{\frac{n(n-1)}{2}-1}\right)} \\
&\lesssim \,\,\,\,\, \frac{\sum\limits_{0 \le b <n} \sum\limits_{\delta \in \mathcal{T}(r_2)} \sum\limits_{A \in \mathscr{L}_{\mathbb{Z}}} \frac{1}{\sigma(r_2)} \chi_A(\delta) \left| \mathcal{J}_A \right| \Vol( \mathcal{F}_A) \Vol(U_{1,b}^{r_2}(\mathbb{R})_{H < X}) \prod \limits_p \Vol(\Lambda_{A,b,p}) + o\left(X^{\frac{n(n-1)}{2}-1}\right)}{\left(\sum\limits_{0 \le b <n} \Vol(U_{1,b}^{r_2}(\mathbb{R}))_{<X})  \prod\limits_p \Vol(S_{b,p})\right) + o\left(X^{\frac{n(n-1)}{2}-1}\right)}. &
\end{flalign*}
where $\,\,\, \lesssim \,\,\,$ denotes an inequality which becomes an equality if $\Sigma$ is very large or if $\Sigma$ is acceptable and the tail estimate of Conjecture \ref{conj: conjectural tail estimate} holds. Note that $\Vol( \Lambda_{A,b,p}) = \left| \mathcal{J}_A \right|_p \Vol(\SO_A(\mathbb{Z}_p)) \int_{f \in S_{p,b}} m_p(f,A) \, df$ by \S \ref{sec: product of local volumes}. Furthermore, note that the quantity $\int_{f \in S_{b,2}} m_2(f,A) df /\Vol(S_{b,2})$ is independent of $b$ by Proposition \ref{prop: mod 2 masses average} and Proposition \ref{prop: mod 2 masses pointwise}. We suppress the $b$ and write it as $c_2(S_2,A)$. Expanding and simplifying, we find:
\begin{flalign*}
   = \sum_{\delta \in \mathcal{T}(r_2)} \sum_{A \in \mathscr{L}_{\mathbb{Z}}} \frac{1}{\sigma(r_2)}  \chi_A(\delta) \Vol(\mathcal{F}_A)\prod_p \Vol(\SO_A(\mathbb{Z}_p)) c_2(S_{2},A) + o(1) &&
\end{flalign*}

We now break up the collection $\mathscr{L}_\mathbb{Z}$ into genera and sum over the forms in each genus separately before summing over the distinct genera. Since both $\chi_{A}(\delta)$ and $c_2(S_2, A)$ are constant over the forms in a single genus, they factor out of the inner sum, and we may replace $A$ by its genus in their notation.
\begin{flalign*}
&= \sum_{\delta \in \mathcal{T}(r_2)} \sum_{\mathcal{G} \in \mathcal{G}_{\mathbb{Z}}} \sum_{A \in \mathcal{G} \cap \mathscr{L}_{\mathbb{Z}}} \frac{1}{\sigma(r_2)}  \chi_A(\delta) \Vol(\mathcal{F}_A)\prod_p \Vol(\SO_A(\mathbb{Z}_p))  c_2(S_2,A) + o(1)\\
&= \sum_{\delta \in \mathcal{T}(r_2)} \sum_{\mathcal{G} \in \mathcal{G}_{\mathbb{Z}}} \frac{1}{\sigma(r_2)} \chi_{\mathcal{G}}(\delta) c_2(S_2,\mathcal{G}) \left(\sum_{A \in \mathcal{G} \cap \mathscr{L}_{\mathbb{Z}}}  \Vol(\mathcal{F}_A)\prod_p \Vol(\SO_A(\mathbb{Z}_p)) \right) + o(1) &&
\end{flalign*}
Now, the innermost sum gives the Tamagawa number of the special orthogonal group of $A$ and is always equal to $2$, independent of $A$, as seen in \cite{MR0213362} or \cite[\S 1]{MR1131433}. We denote it by $\tau(\SO)$.
\begin{flalign*}
    = \tau(\SO) \sum_{\delta \in \mathcal{T}(r_2)} \sum_{\mathcal{G} \in \mathcal{G}_{\mathbb{Z}}}  \frac{1}{\sigma(r_2)} \chi_{\mathcal{G}}(\delta) c_2(S_2,\mathcal{G}) + o(1) &&
\end{flalign*}
At this point, we simplify the sum using the fact that $\sigma(r_2) = \frac{1}{2^{r_1+r_2-1}}$, the value of the $2$-adic mass, the value of the infinite mass, and the classification of genera of integral unimodular quadratic forms from \S \ref{subsec: background on quadratic forms}.
\begin{flalign*}
    &= \frac{\tau(\SO)}{2^{r_1+r_2-1}} \sum_{\delta \in \mathcal{T}(r_2)} \sum_{\mathcal{G} \in \mathcal{G}_{\mathbb{Z}}} \chi_{\mathcal{G}}(\delta) c_2(S_2,\mathcal{G}) + o(1) \\
&= \frac{\tau(\SO)}{2^{r_1+r_2-1}} \sum_{\mathcal{G} \in \mathcal{G}_{\mathbb{Z}}} \sum_{\delta \in \mathcal{T}(r_2)}  \chi_{\mathcal{G}}(\delta) c_2(S_2,\mathcal{G}) + o(1)\\
&= \frac{\tau(\SO)}{2^{r_1+r_2-1}} \sum_{\mathcal{G} \in \mathcal{G}_{\mathbb{Z}}}  \left( c_2(S_2,\mathcal{G}) \sum_{\delta \in \mathcal{T}(r_2)}  \chi_{\mathcal{G}}(\delta) \right) + o(1)  \\
&=\frac{\tau(\SO)}{2^{r_1+r_2-1}} \sum_{\substack{([q_2],[q_\infty]) \in \frac{\Sym_2(\mathbb{Z}_2^n)}{\SL_n(\mathbb{Z}_2)} \times \frac{\Sym_2(\R^n)}{\SL_n(\R)} \\  \det(q_2)= \det(q_\infty) = \factor \\ e_2(q_2)e_\infty(q_\infty) = 1}}  c_{2}(S_{2},q_2) m_{\infty}(r_{2},q_{\infty}) + o(1)  \\
&= \frac{\tau(\SO)}{2^{r_1+r_2-1}}  \left[ \left( c_2(S_2,\mathfrak{M}_1) \sum_{\substack{[q_{\infty}] \in \frac{\Sym_2(\R^n)}{\SL_n(\R)} \\ \det(q_{\infty}) = \factor \\ e_\infty(q_{\infty}) = 1}}  m_{\infty}(r_2,q_{\infty}) \right)  + \left( c_2(S_2\mathfrak{M}_{-1}) \sum_{\substack{[q_{\infty}] \in \frac{\Sym_2(\R^n)}{\SL_n(\R)} \\ \det(q_{\infty}) = \factor \\ e_\infty(q_{\infty}) = -1}}  m_{\infty}(r_2,q_{\infty}) \right) \right]+ o(1)  \\
&= \frac{\tau(\SO)}{2^{r_1+r_2-1}} \Big( c_2(S_2,\mathfrak{M}_1)c_{\infty,1}(r_2)  + c_2(S_2,\mathfrak{M}_{-1})c_{\infty,-1}(r_2) \Big) + o(1) &&
\end{flalign*}
Now, these values were obtained in \S \ref{sec: 2-adic distribution} and \S \ref{sec: archimedean distribution}. We substitute them, set $\tau(\SO)=2$ and simplify.
\begin{flalign*}
&= \frac{1}{2^{r_1+r_2-1}} \cdot 2 \cdot \frac{1}{4} \left( \left(2^{n-1}\pm_8 2^{\frac{n-1}{2}}\right)\left(2^{r_1} \pm_8 2^{\frac{r_1-1}{2}}\right) + \left(2^{n-1}\mp_8 2^{\frac{n-1}{2}} \right) \left(2^{r_1-1}\mp_8 2^{\frac{r_1-1}{2}} \right) \right) + o(1) \\
&= \frac{1}{2^{r_1+r_2-1}} \cdot 2 \cdot \frac{1}{4} \left( \left(2^{n-1}-2^{\frac{n-1}{2}}\right)\left(2^{r_1}-2^{\frac{r_1-1}{2}}\right) + \left(2^{n-1}+2^{\frac{n-1}{2}} \right) \left(2^{r_1-1}+2^{\frac{r_1-1}{2}} \right) \right) +o(1)\\
&= \frac{1}{2^{r_1+r_2-1}} \cdot 2 \cdot \frac{1}{4} \left( 2 \cdot \left(2^{n+r_1-2}+2^{\frac{n+r_1-2}{2}} \right) \right) + o(1)\\
&= 2^{r_1+r_2-1}+1 + o(1). &&
\end{flalign*}

This demonstrates Theorem \ref{main theorem} for ${\rm Avg}(\SS;\Cl(\cdot)[2])$. To prove Theorem \ref{main theorem} for ${\rm Avg}(\SS;\Cl(\cdot)^{+}[2])$, the setup is the same as above, except that we now sum over $\delta_{\gg 0} = (11\cdots 1) \in \mathbb{R}^{r_1} \times \mathbb{C}^{r_2}$ instead of the entire collection $\mathcal{T}(r_2)$. Otherwise, using the statement about the Hasse--Witt invariant of forms $A$ with $\chi_A(\delta_{\gg 0}) \neq 0$ from Proposition \ref{prop: distribution of archimedean mass} the computation is analogous and we find:  
\[
\frac{\sum\limits_{\substack{\mathcal{O} \in \Sigma \\ H_{\text{naive}}(\mathcal{O}) < X}} 2^{r_2} \left| \Cl^+(\mathcal{O})[2] \right| - \left| \I_2(\mathcal{O}) \right|}{\sum\limits_{\substack{\mathcal{O} \in \Sigma \\ H_{\text{naive}}(\mathcal{O}) < X}} 1} \,\,\, \lesssim \,\,\, 2^{r_2}+\frac{2^{r_2}}{2^{\frac{n-1}{2}}} + o(1).
\]

Thus we obtain the following bounds/conditional averages for the class group and the narrow class group of acceptable families $\SS \subset \mathfrak{R}_{1,\max}^{r_1,r_2}$ of monogenised fields: \[{\rm Avg} \left(\Sigma, \Cl(\cdot)[2] \right) \,\,\, \lesssim \,\,\, 2^{r_2}+\frac{2^{r_2}}{2^{\frac{n-1}{2}}} \hspace{5ex} \text{ and } \hspace{5ex}  {\rm Avg} \left(\Sigma, \Cl(\cdot)[2] \right)\,\,\, \lesssim \,\,\, 1 + \frac{1}{2^{\frac{n-1}{2}}}+\frac{1}{2^{r_2}},\] thereby completing the proof of Theorem \ref{main theorem}.

\subsection{Proof of Corollary \ref{main corollary}}

Items 1) and 3) follow from Theorem \ref{main theorem} in the same way as \cite[Corollary 6.7]{MR3782066} follows from \cite[Theorem 2 and Theorem 3]{MR3782066}. Items 2) and 4) follow from items 1) and 3) in the same way as \cite[Theorem 6.8]{MR3782066} follows from \cite[Corollary 6.7]{MR3782066}.

\section{Averages for $N$-monogenised fields} \label{sec: averages for $N$ monogenised fields}

In this section, we explain how to prove Theorem \ref{thm: main N monogenic theorem} by using the work done in \S 2 -- \S10 in the context of rings associated with binary $n$-ic forms having any fixed leading coefficient. We begin by clarifying the definition of collections of local conditions (we require a specific ``coherence condition''). 

\subsection{Notation and coherent collections of local conditions} Fix an odd degree $n$, a leading coefficient $0 \neq N \in \Z$, and a signature $(r_{1},r_{2})$. 

We recall from the introduction that we denote by $\mathfrak{R}_{N,\max}^{r_1,r_2}$ the family of equivalence classes of rings $R_f$ corresponding to binary $n$-ic forms $f \in \Sym_n(\Z^2)$ with $r_1$ real roots and $2r_2$ complex roots, leading coefficient $N$, and with the property that: 1) $f$ is primitive; 2) $R_f$ is the maximal order in its fraction field $K_f$. 

Collections of local conditions are defined almost in the same way as before: They are subsets $S = \bigsqcup_{0 \le b < n} \prod_{p} S_{b,p} \subset \bigsqcup_{0 \le b < Nn} U_{1,b}(\widehat{\Z})$ with the property that each $S_{b,p}$ is non-empty, clopen, has boundary with measure $0$, except that we also require such collections to be ``coherent'' across the different $b$'s. By this, we mean that they must be the preimage of a subset of $F_{1}(\Z) \setminus U_{N}(\Z)$ under the quotient map $\bigsqcup_{0 \le b < Nn} U_{N,b}(\widehat{\Z}) \rightarrow F_{1}(\widehat{\Z}) \setminus U_{N}(\widehat{\Z})$.

We then have the notion of an \emph{acceptable family} of $N$-monogenised fields $\SS \subset \mathfrak{R}_{N,\max}^{r_1,r_2}$ associated to an \emph{acceptable collection} of local specifications $(\SS_p)_p$. For $p|N$, the \emph{even ramification density} of $\SS$ at $p$, $\square_p(\SS)$, is the density in $\Sigma_p$ of elements of $f_p \in \Sigma_p$ with the property that $f_p(x,1)$ is a square mod $p$ up to multiplication by $(\Z/p)^\times$, that is $f_p(x,1) = u g_p(x,1)^2$ for some polynomial $g_p(x,1)$ and some unit $u \in (\Z/p)^\times$.

\subsection{Proof of Theorem \ref{thm: main N monogenic theorem}} We now prove Theorem \ref{thm: main N monogenic theorem}. Fix $0 \le b <Nn$. By the coherence condition imposed on families of local conditions, it will be sufficient to prove the formulae of Theorem \ref{thm: main N monogenic theorem} for forms having a fixed $b$. The material in \S \ref{sec: parametrisations} had no restriction on the form being monic, and so works here. We recall the main parametrisation, due to Bhargava for \cite{MR2081442} for $n=3$ and Wood \cite{MR3187931} for general $n$. 

\begin{thm}[{\cite[Theorem 6.3]{MR3187931}}] \label{Wood parametrisation 2}
Let $n$ be odd and $T$ a PID. Take a non-degenerate primitive binary $n$-ic form $f \in U(T)$. We have a bijection between $\SL_n(T)$-orbits of pairs $(A,B) \in V(T)$ with $\pi(A,B) = f$ and equivalence classes of pairs $(I,\delta)$ where $I \subset K_f$ is a fractional ideal of $R_f$ and $\delta \in K_f^\times$ such that $I^2 \subset (\delta) I_{f}(n-3)$ as ideals and $f_{0}^{n-3}N(I)^2 = N(\delta)$. The classes $(I,\delta)$ and $(I',\delta')$ are equivalent if there exists a $\kappa \in K_f^\times$ with the property that $I = \kappa I'$ and $\delta = \kappa^2 \delta$. Furthermore, the forms $A$ and $B$ have simple descriptions. They are, respectively, the bilinear forms $\check{\zeta}_{n-1}(xy/\delta)$ and $\check{\zeta}_{n-2}(xy/\delta)$ on $I_{f}(n-3)$ expressed in a common basis for $I$, where we write $ \{ \check{\zeta}_{0}, \check{\zeta}_{1}, \ldots, \check{\zeta}_{n-1} \}$ for the dual basis of $R_{f}$. 
\end{thm}

The matrix $A$ has a natural interpretation. 

\begin{prop}\label{Euler}
Let $D$ be a field of characteristic $0$, $\Z_p$ or $\Z$ with fraction field $K$. Let $f$ be a non-degenerate polynomial of degree $n$ with coefficients in $D$ and leading coefficient equal to $f_0$ and let $K_{f} := K[\theta]/f(\theta,1)$. If $(A,B)$ is the pair corresponding to $(I,\alpha)$, then $A$ is the matrix representation of the bilinear form $\Tr \left(\frac{xy}{\alpha f'(\theta)} \right)$ in some basis of $I$. 
\end{prop} 
\begin{proof}
The proof follows a computation of Euler. We follow Serre's presentation in \cite[pg. 56]{MR554237}. Write $f(X) = f_0(X-x_1)\cdots (X-x_n)$ over the algebraic closure of $T$ or its fraction field and express, using an expansion into partial fractions, $\frac{1}{f(X)} = \sum_{k=1}^n \frac{1}{f'(x_k)(X-x_k)}$. Expand both sides as a power series in $\frac{1}{X}$ to find: 
\[\frac{1}{f_0 X^n} \prod_{k=1}^n \sum_{j \ge 0} \left(\frac{x_k}{X} \right)^j = \sum_{k=1}^n \sum_{j \ge 0} \frac{1}{f'(x_k)X} \left( \frac{x_k}{X} \right)^j.\]
Note that the lowest degree term on the left-hand side is $\frac{1}{f_0X^n}$. Equate terms on both sides to find:
\[\Tr \left( \frac{\theta^k}{f'(\theta)} \right) = \begin{cases}
\frac{1}{f_0} &\text{ if } k=n-1 \\
0 & \text{ if } k < n-1
\end{cases}.
\]
This shows that $\Tr \left(\frac{xy}{\alpha f'(\theta)} \right) = \check{\zeta}_{n-1}\left( \frac{xy}{\alpha} \right)$, completing the proof. 
\end{proof}

Since maximality is a local condition and $\mathfrak{R}_{N,\max}^{r_{1},r_{2}}$ consists of maximal rings, we obtain the following local restrictions on the matrices $A$ that appear in our count. 

\begin{prop} \label{prop: one p block}
Let $f \in U_N(\Z_p)$ be maximal at $p$ and suppose that $p^k \parallel N$. Then, for any $(A, B) \in V(\Z_p)$ with $\pi(A, B) = f$, the Jordan decomposition of $A$ at $p$ contains exactly one block of scale $p^0$ of dimension $n-1$ and exactly one block of scale $p^k$ of dimensions $1$. 
\end{prop}
\begin{proof}
$\dim_{\F_p}\textrm{rad}(\overline{A}) \ge 2$ implies that $\pi(A,B)$ is not maximal at $p$. 
\end{proof}

\begin{rem} \label{rem: diagonal at 2}
    We furthermore note that by \cite[Theorem 80]{Siadeven}, if the resolvent of a pair $(A, B)$ over $\Z_{2}$ is maximal and unramified modulo $2$, then the mod $2$ reduction of $A$ is diagonalisable. 
\end{rem}

Let $\mathscr{L}^{N}_{\Z}$ be the set of $\SL_{n}(\Z)$ classes of integral symmetric matrices of determinant $\factor N$ with the property that for all $p$, the mod $p$ reduction has rank at least $n - 1$ and the mod $2$ reduction is diagonalizable.

Let $A \in \mathscr{L}_{\Z}$, $0 \le b < Nn$, $(r_{1},r_{2})$ be a signature, and $\delta \in \mathcal{T}(r_{2})$.  Determining the averages then reduces to counting integral points in $V_{A,b}^{\delta,r_{2}}(\R)$ of bounded height, subject to local conditions. The work of \S \ref{sec: parametrisations} - \S \ref{sec: product of local volumes} immediately yields the following theorem. 

\begin{thm} \label{asymptotic formula}
Let $r_2$ be an integer such that $0 \le r_2 \le \frac{n-1}{2}$ and let $\delta \in \mathcal{T}(r_2)$. Then for an acceptable collection of local specifications $\Lambda_{A,b}^{\delta}$, we have 
\begin{equation*}
N(\mathcal{V}(\Lambda_{A,b}^\delta), X) \le \frac{1}{\sigma(r_2)} \Vol(\mathcal{F}_A \cdot R_{A,b}^{r_2,\delta}(X)) \prod \limits_p \Vol(\Lambda_{A,b,p}) + o\left(X^{\frac{n(n-1)}{2}-1}\right),
\end{equation*}
where the volumes of subsets of $V_{A,b}(\mathbb{R})$ are computed with respect to the Euclidean measure normalised so that $V_{A,b}(\mathbb{Z})$ has covolume $1$. The volumes of subsets of $V_{A,b}(\mathbb{Z}_p)$ are computed with respect to the Euclidean measure normalised so that $V_{A,b}(\mathbb{Z}_p)$ has measure $1$. This is an equality assuming a tail estimate of the form of Conjecture \ref{conj: conjectural tail estimate} holds.
\end{thm}

We define the local masses in the same manner as before. By the classification of Theorem \ref{genera computation}, the final computation will depend on the number of local orbits whose associated $A$ has a given Hasse--Witt symbol.  

\begin{lem}[Dichotomy in the behaviour of Hasse--Witt symbol] \label{dichotomy}
Let $N = m^2k$ where $k$ is squarefree. Let $f \in U_{N}(\Z_p)$ and suppose that $R_f$ is maximal at $p$. Recall that the set of $\SL_n(\Z_p)$-orbits in $\pi^{-1}(f)$ is in bijection with $(R_f^\times/(R_f^\times)^2)_{N \equiv 1}$. Define $A_\delta := \Tr \left( xy/\delta f'(\theta) \right)$ for $\delta \in (R_f^\times/(R_f^\times)^2)_{N \equiv 1}$. Let $p$ be an odd prime.

\begin{enumerate}[1)]

\item If $p \nmid k$, then $e_p(A_{\delta}) = 1$ for all $\delta \in (R_f^\times/(R_f^\times)^2)_{N \equiv 1}$.

\item If $p \mid k$, the following dichotomy holds: 
\begin{enumerate}[a)]
\item If $f$ is evenly ramified, $e_p(A_{\delta})= \left(\factor,  N\right)_p \left(-1,(-1)^{{(n-1)/2 \choose 2}} \right)_p$  for all $ \delta \in (R_f^\times/(R_f^\times)^2)_{N \equiv 1}$.
\item Otherwise, $e_p(A_{\delta})$ is equal to $1$ and $-1$ equally often.  
\end{enumerate} 
\end{enumerate}
\end{lem}
\begin{proof}
\textit{1)} Suppose $p \nmid k$ and $p  \neq 2$ is odd. Then the $p$-adic valuation of $N$ is even, $2 \mid v_p(N)$. By inspecting canonical forms (see \cite[Chapter 8 Theorem 3.1.]{MR522835}), we find that all symmetric bilinear forms over $\Z_p$ of determinant $\factor N$ which are either unimodular or whose Jordan decomposition contains exactly one block of scale $p^0$ of dimension $n-1$ and exactly one block of scale $p^{v_p(N)}$ of dimensions $1$ have Hasse--Witt symbol equal to $1$. In particular, under the hypotheses of 1), $e_p(A_{\delta}) = 1$ for all $\delta \in (R_f^\times/(R_f^\times)^2)_{N \equiv 1}$ as required. 

\textit{2)} Suppose that $p \mid k$ and $p  \neq 2$ is odd. Then $v_p(N)$ is odd.
Using Hensel's lemma, we can factor $f(x,y)$ over $\Z_p$ as $f(x,y)= h_1(x,y)h_2(x,y)$ with $h_1(x,y)$ having leading coefficient $N$, $h_2$ having leading coefficient $1$, $h_1(x,y)$ and $h_2(x,y)$ co-prime modulo $p$, and such that $h_1(x,y)$ reduces to a unit times $y^{\deg(h_1)}$ modulo $p$, see \cite[proof of Proposition A.2]{2207.05592}. In addition, since  $h_1(x,y)$ and $h_2(x,y)$ are co-prime modulo $p$ we find $R_f \cong R_{h_1} \times R_{h_2}$, see \textit{loc. cit.}.  In particular,  we may decompose $R_f^\times/(R_f^\times)^2 \cong \left( R_{h_1}^\times /(R_{h_1}^\times)^2 \right) \times \left(R_{h_2}^\times / (R_{h_2}^\times)^2 \right)$ whence $A_\delta = A_{\delta_1} \perp \widetilde{A_{\delta_2}}$ for each $\delta = (\delta_1,\delta_2) \in R_f^\times/(R_f^\times)^2$.  Proposition \ref{Euler}, then gives that $\det(A_1) \in p\Z_p$ and $\det(\widetilde{A_{\delta_2}}) \in \Z_p^\times$.  We can now begin dealing with our cases. 

\textit{a)} Suppose that $f$ is evenly ramified and write $\Z_p^\times/(\Z_p^{\times})^2 = \{1,r\}$. Then the norm map $N \colon R_{h_2}^\times/(R_{h_2}^\times)^2 \rightarrow \Z_p^\times/(\Z_p^{\times})^2$ has image $\{1\}$. In particular, $(R_f^\times/(R_f^\times)^2)_{N \equiv 1} \cong (R_{h_1}^\times/(R_{h_1}^\times)^2)_{N \equiv 1} \times (R_{h_2}^\times/(R_{h_2}^\times)^2)_{N \equiv 1} \cong \{1 \} \times (R_{h_2}^\times/(R_{h_2}^\times)^2)_{N \equiv 1}$.  We can then express $A_{(1,\delta_2)} = A_1 \perp \widetilde{A_{\delta_2}}$, with $\det(A_1) \in p\Z_p$ and $\det(\widetilde{A_{\delta_2}}) \in \Z_p^\times$, using this decomposition.  Computing the Hasse--Witt symbol of $A_{(1,\delta_2)}$ we find that it is constant across all $(1,\delta_2) \in (R_f^\times/(R_f^\times)^2)_{N \equiv 1}$.  Indeed:  
\begin{align*}
e_p(A_{(1,\delta_2)}) &= e_p(A_1) e_p(\widetilde{A_{\delta_2}}) ( \det A_1, \det\widetilde{A_{\delta_2}})_p \\ 
&= e_p(A_1) e_p(\widetilde{A_1}) ( \det A_1,  N(\delta_2) \det\widetilde{A_1})_p \\ 
&=  e_p(A_1) e_p(\widetilde{A_1}) ( \det A_1, \det\widetilde{A_1})_p \\ 
&= e_p(A_{(1,1)})
\end{align*}
where we used the fact that $\det\widetilde{A_{\delta_2}} = N(\delta_2) \det\widetilde{A_1}$ which is immediate from Proposition \ref{Euler}. An explicit computation with the basis $\langle1,\theta,\ldots,\theta^{\frac{n-1}{2}}, \zeta_{\frac{n-1}{2}+1},\ldots,\zeta_{n-1} \rangle$ shows that $A_{(1,1)}$ is $\SL_n(\Q_p)$-equivalent to \[\begin{psmallmatrix}
& & & & & & 1 \\ 
& & & & & \iddots & \\ 
& & & & 1 & & \\ 
& & & N & & & \\  
& & 1 & & & &  \\ 
& \iddots & & & & & \\ 
1 & & & & & &  
\end{psmallmatrix}\]
which has Hasse--Witt symbol $\left(\factor,  N\right)_p \left(-1,(-1)^{{(n-1)/2 \choose 2}}\right)_p$. 

\textit{b)} It suffices to find a bijection $\Xi \colon (R_f^\times/(R_f^\times)^2)_{N \equiv 1} \rightarrow (R_f^\times/(R_f^\times)^2)_{N \equiv 1}$  with $e_p(A_{\Xi(\delta)}) = - e_p(A_\delta)$, i.e. which reverses the Hasse--Witt symbol of $A_{(\cdot)}$.  First, decompose $R_f^\times/(R_f^\times)^2 \cong \left( R_{h_1}^\times /(R_{h_1}^\times)^2 \right) \times \left(R_{h_2}^\times / (R_{h_2}^\times)^2 \right)$.  Next, fix $r \in \Z_p^\times$ be such that $r \not\in (\Z_p^\times)^2$ and find a $\delta^\star \in R_{h_2}^\times / (R_{h_2}^\times)^2$ such that $N(\delta^\star) = \frac{1}{r^n}$ (such a $\delta^\star$ exists since $f$ is not evenly ramified). We define $\Xi$ as follows: 
\[
\delta = (\delta_1,\delta_2) \mapsto (\delta_1,\delta^\star \delta_2) \mapsto (A_{\delta_1} \perp A_{\delta^\star\delta_2},B_{\delta_1} \perp B_{\delta^\star\delta_2}) \mapsto r (A_{\delta_1} \perp A_{\delta^\star\delta_2},B_{\delta_1} \perp B_{\delta^\star\delta_2}) \mapsto \Xi(\delta)
\] where the second and fourth arrow come from the orbit parametrisation and the third arrow denotes the multiplication of the elements of the orbit by the $r \in \Z_p$ we fixed earlier.  By \cite[Theorem 1.3]{MR3187931},  $(A_{\delta_1} \perp A_{\delta^\star\delta_2},B_{\delta_1} \perp B_{\delta^\star\delta_2})$ lies in $\pi^{-1}(\frac{1}{r^n} f)$, and thus $r (A_{\delta_1} \perp A_{\delta^\star\delta_2},B_{\delta_1} \perp B_{\delta^\star\delta_2})$ lies in $\pi^{-1}(f)$.  This map is a bijection since it is the composition of bijections. 

Since $n$ and $p$ are odd,  the third arrow is readily seen to preserve the Hasse--Witt symbol of $A_{(\cdot)}$.  Thus, it only remains to compare the Hasse--Witt symbols of $A_{\delta_1} \perp A_{\delta^\star\delta_2}$ and $A_{\delta_1} \perp A_{\delta_2}$. Noting that $\det A_{\delta^\star\delta_2} = N(\delta^\star) \det A_{\delta_2}$, we compute to find on the one hand: 
\[e_p(A_{\delta_1} \perp A_{\delta_2}) = e_p(A_{\delta_1}) e_p(A_{\delta_2}) (\det A_{\delta_1}, \det A_{\delta_2})_p\]
and on the other
\begin{align*}
e_p(A_{\delta_1} \perp A_{\delta^\star\delta_2}) &= e_p(A_{\delta_1}) e_p(A_{\delta^\star\delta_2})  (\det A_{\delta_1}, N(\delta^\star)\det A_{\delta_2})_p \\
&= e_p(A_{\delta_1}) e_p(A_{\delta^\star\delta_2})  \left(-(\det A_{\delta_1}, \det A_{\delta_2})_p \right) \\
&= - e_p(A_{\delta_1}) e_p(A_{\delta_2}) (\det A_{\delta_1}, \det A_{\delta_2})_p \\
&= - e_p(A_{\delta_1} \perp A_{\delta_2}) 
\end{align*}
where the third equality follows since $N(\delta^\star) = 1/r^n$ is not a square in $\Z_p^\times$,  $\det A_{\delta_1}$ has odd $p$-adic valuation, and the Hilbert symbol for odd primes $p$ satisfies $(\alpha p,  r \beta)_p = - (\alpha p,  \beta)_p$ for all $\alpha, \beta \in \Q_p^\times$ and $r$ a non-square in $\Z_p^\times$.  The fourth line follows since $e_p(A_{\delta^\star\delta_2}) = 1 = e_p(A_{\delta_2})$ since both $A_{\delta^\star\delta_2}$ and $A_{\delta_2}$ are unimodular $\Z_p$ bilinear forms for $p \neq 2$. Therefore, the bijection $\Xi$ we have constructed reverses the Hasse--Witt symbol for $A_{(\cdot)}$ as required.   
\end{proof}

By the classification of quadratic forms, for odd primes $p$, there is a unique $\SL_{n}(\Z_{p})$ class of $\Z_{p}$-integral quadratic forms having a given determinant indivisible by $p$. If we ask for the given determinant to be divisible by $p$ and the mod $p$ reduction to be of rank $n-1$, there are precisely two such classes. Suppose an even power of $p$ exactly divides the given determinant. In that case, both classes have Hasse--Witt symbol $1$, while if an odd power of $p$ exactly divides the given determinant, they are distinguished by their Hasse--Witt invariant. When $p = 2$, there can be many forms with a determinant divisible by $2$ and with a diagonalisable mod $2$ reduction. 

Because the final computation depends only on the distribution of the total masses, which are the same as in \S \ref{sec: product of local volumes}, across forms of a certain Hasse--Witt symbol, it suffices to obtain the value of the following averages. For $S_{b,p}  \subset U_{A,b}(\Z_{p})$, define
\[c_{p}(S_{b,p},+1) := \frac{\int\limits_{S_{b,p}}{ \sum\limits_{\substack{A \in \frac{\Sym_2(\Z_{p}^n)}{\SL_n(\Z_{p})}, \\ \det(A)= \factor N, \\ e_p(A) = +1 }}} m_{p}(f,A)}{\Vol(S_{b,p})}.\]
We define $c_{p}(S_{b,p},-1)$ similarly, replacing the condition $e_p(A) = + 1$ in the sum by $e_p(A) = -1$. 

The exact values for $c_{p}(S_{b,p},\pm 1)$ in the cases of interest are given in the following theorem. 

\begin{thm} 
For each prime $p \mid 2N$, let $S_{b,p} \subset U_{A,b}(\Z_{p})$ be a local condition at $p$ consisting only of maximal forms. Furthermore, suppose that $S_{2,b}$ is defined by local conditions mod $2$ and only consists of forms which are unramified at $2$. Then: 

\begin{itemize}
    \item If $p$ is odd and $p \mid N$ to an even power, we have $$c_{p}(S_{b,p},+1) = 1 \hspace{5ex} \text{ and } \hspace{5ex} c_{p}(S_{b,p},-1) = 0.$$ 
    \item If $p$ is odd and $p \mid N$ to an odd power, we have $$c_{p}(S_{b,p},\pm 1) =  \frac{1}{2} \pm \frac{1}{2} \left(\factor,  N \right)_p \left(-1,(-1)^{{(n-1)/2 \choose 2}}\right)_p\square_p(S_{p,b})$$ where $\square_p(S_{p,b})$ denotes the even ramification density of $S_{p,b}$. 
    \item If $2$ divides $N$ to an even power, $$c_{2}(S_{b,2}, 1) = \frac{1}{2}\left(2^{n-1} \pm_8 2^{\frac{n-1}{2}} \right) \hspace{5ex} \text{ and } \hspace{5ex} c_{2}(S_{b,2}, 1) = \frac{1}{2}\left(2^{n-1} \mp_8 2^{\frac{n-1}{2}} \right)$$ where $\pm_8$ is $+$ if $n$ is congruent to $1,3 \pmod 8$ and $-$ otherwise. 
    \item If $2$ divides $N$ to an odd power, $$c_{2}(S_{b,2}, \pm 1) = \frac{1}{2} 2^{n-1}.$$  
\end{itemize}
\end{thm}
\begin{proof}
For odd $p$, this follows directly from Lemma \ref{dichotomy}. 

For $p = 2$, Remark \ref{rem: diagonal at 2} shows that all the forms $A$ associated to $f \in S_{2,b}$ are diagonal because $S_{2,b}$ have the rank $n-1$ modulo $2$. In particular, they are all equal modulo $2$. 

An argument similar to the one in \S \ref{sec: 2-adic distribution} shows that it suffices to look at the ratios between the reciprocals of $\Vol(\SO_{A}(\Z_{2}))$ for $A$ ranging in a set of canonical forms having determinant $\factor N$ and rank $n-1$ modulo $2$. Using \cite[Theorem 1 and Lemmas 1, 2 and 3]{MR12640} to find such a set of canonical representatives and \cite[\S 5]{MR965484} to compute the volumes of their respective special orthogonal groups gives these ratios and the final result.
\end{proof}

The archimedean local masses are the same as before. Using the following lemma of Hanke and doing a mass calculation similar to the one done in \S \ref{sec: proof of main theorem} yields Theorem \ref{thm: main N monogenic theorem}.

\begin{lem}[{\cite[Lemma 4.17]{1108.3580}}]
Let $\mathbb{T}$ be a finite set and $X_{i}, Y_{i}$ be variables indexed by the elements of $\mathbb{T}$. Then
\[\sum_{\substack{(\epsilon_{i}) \in \{\pm 1\}^{\mathbb{T}} \\ \prod_{i \in \mathbb{T}} \epsilon_{i} = 1}} \left(\prod_{i \in \mathbb{T}} X_{i}+\epsilon_{i} Y_{i} \right) = 2^{\mathbb{T}-1} \left(\prod_{i \in \mathbb{T}} X_{i} + \prod_{i \in \mathbb{T}} Y_{i}\right). \]
\end{lem}

\noindent This gives us the following bounds/conditional averages for acceptable families $\SS \subset \mathfrak{R}_{N,\max}^{r_1,r_2}$:
\[{\rm Avg}(\SS,\Cl(\cdot)[2]) \,\, \lesssim \,\,  1+ \frac{1+ \prod_{p|k}\square_p(\SS)}{2^{r_1+r_2-1}} \hspace{5ex} \text{ and } \hspace{5ex} {\rm Avg}(\SS,\Cl^+(\cdot)[2]) \,\, \lesssim \,\, 1 + \frac{\prod_{p|k}\square_p(\SS)}{2^{\frac{n-1}{2}}}+\frac{1}{2^{r_2}},\]
which completes the proof of Theorem \ref{thm: main N monogenic theorem}. \\

It seems plausible that we could improve Lemma \ref{dichotomy} to cover the case $p = 2$ by cleverly applying the formula for the second Stiefel--Whitney classes of twisted trace forms \cite[Th\'eor\`eme 1']{SerreWittInvariant84}.
Presumably, this would also give a different proof of Proposition \ref{prop: mod 2 masses pointwise}. \\

The proof of Theorem \ref{thm: main N monogenic theorem} we obtained in this section offers a new perspective on the intricate local mass calculations of \cite{BhargavaHankeShankar} where class field theory played a central role. In the approach taken in this section, class field theory emerges through the formula for the Hasse--Witt invariant of orthogonal direct sums of quadratic forms. At its core, this formula comes from the bilinearity of the Hilbert symbol, a product of local class field theory.

\section*{Index of notation}
\addcontentsline{toc}{section}{Index of Notation}

Here is an index of the notation used, roughly arranged in the chronological order in which they first appear in the text.

\begin{longtable}{lp{13.5cm}}
$\widehat{\Z}$ &the profinite integers $\prod_p \Z_p$. \\
$U(T)$ &space of binary $n$-ic forms with coefficients in $T$. \\
$\Delta(\cdot)$ &the discriminant.\\
$\mathfrak{R}_1^n$ &isomorphism classes of monogenised $n$-ic rings. \\
$\mathfrak{R}_{N,\max}^n$ &isomorphism classes of $N$-monogenised $n$-ic fields. \\ 
$\Cl(\O)$ & class group of the order $\O$. \\ 
$\Cl^+(\O)$ & narrow class group of $\O$. \\
$\mathcal{I}_2(\O)$ & $2$-torsion subgroup of the ideal group of the order $\O$. \\ 
$H_{\text{naive}}$ &the naive height on binary form or monic polynomials. \\
$H_{\text{roots}}$ & box-height on the roots of monic polynomials. \\ 
$e_p(\cdot)$ & the Hasse--Witt invariant of a quadratic form.\\
$\mathscr{L}_\Z$ & set of representatives of equivalence classes $n$-ary integer-matrix quadratic forms of \\
& determinant $\factor$ under the action of $\SL_n(\mathbb{Z})$. \\
$\mathcal{G}_{\Z,N}$ & collection of genera of integral quadratic forms of discriminant $\factor N$. \\
$H(\O)$ &the set of pairs $(I,\delta)$ consisting of a fractional ideal $I \subset \mathcal{O}$ and an element $\delta \in K^\times$, where $K$ is the field of fraction of $\O$, such that $I^2 \subset (\delta)$, $N(I)^2 = N(\delta)$ and such that the ideal $I$ is projective. It is a group under component-wise multiplication. \\ 
$H^+(\O)$ &the subgroup of $H(\O)$ consisting of pairs $(I,\delta)$ with $\delta$ having the property that its image under any real embedding of $K$ is positive.  \\
$V(T)$ & $T^2 \otimes \Sym_2 (T^n)$ the space of pairs of symmetric $n \times n$ matrices with coefficients in $T$. This space appears variously decorated in the text. For example $V_{A,b}^{r_2,\delta}(\R)$ denotes the space of pairs $(A,B)$ whose resolvent polynomial has subleading coefficient equal to $b$, $r_2$ pairs of conjugate complex roots, and is in the real orbit labelled $\theta \in \mathcal{T}(r_2)$ in \S \ref{subsec: decorations on V}.\\
$R_f$ & the ring $R_f$ attached to a binary form $f$. Defined in \S \ref{sec: parametrisations}. \\
$N(L;X)$ &for $L \subset V_{A,b}^{r_2,\delta}(\mathbb{Z}) := V_{A,b}^{r_2,\delta} (\mathbb{\mathbb{R}}) \cap V_{A,b}(\mathbb{Z})$ an $\SO_A(\mathbb{Z})$-invariant set, the number of absolutely irreducible $\SO_A$ orbits in $L$ of height at most $X$.  \\ 
$\T(r_2)$ &the set of tuples in $\{\pm 1\}^{r_1} \times \{1 \}^{r_2}$ with an even number of $-1$ entries. Elements of $\mathcal{T}(r_2)$ are denoted by $\delta$.\\ 
$\mathcal{W}_A$ & the bilinear space of rank $n$ with associated Gram matrix $A$.\\
$R_{A,b}^{r_2,\delta}$ &fundamental set for the action of $\SO_A(\mathbb{R})$ on $V_{A,b}^{r_2,\delta}(\mathbb{R})$ satisfying the properties of Proposition \ref{prop: good fundamental domain}. \\
$\FF_A$ &fundamental domain for the action of $\SO_{A}(\Z)$ on $\SO_{A}(\R)$. \\
$\sigma(r_2)$ &shorthand for $2^{r_1+r_2-1}$, the size of the stabilizer in $\SO_A(\R)$ of elements in $V_{A,b}^{r_2}(\R)$. \\
$\chi_{A,b}^{r_2}(\delta)$ &indicator symbol recording whether $V_{A,b}^{r_2,\delta}(\mathbb{R})$ is empty or not. \\
$L_{A,b}$ & the restriction of lattice in $L \subset V_A(\Z)$ to $V_{A,b}(\R)$. \\
$A_{pq}$ &canonical form for $A$ over $\Q$ chosen in \S \ref{subsec: choice of FA and haar measure on SOA}, $g_A^tAg_A = A_{pq}$ for some $g_A \in \SL_n(\Q)$.  \\
$NTK$ & the Iwasawa decomposition of $\SO_{A_{pq}}(\R)$. \\
$\sigma_V$ &for $K = \mathbb{R}$ or $\mathbb{Q}$, the map $\sigma_V \colon V_{A,b}^{r_2,\delta}(K) \rightarrow V_{A_{pq},b}^{r_2,\delta}(K)$ given by $\sigma_V(A,B) = (A_{pq},g_A^tBg_A)$. \\
$\sigma_A$ &for $K = \mathbb{R}$ or $\mathbb{Q}$, the map $\sigma_A \colon \SO_A(K) \rightarrow \SO_{A_{pq}}(K)$ given by  $\sigma_A(h) = g_A^{-1} h g_A$. \\
$\mathcal{L}$ &the subset $\sigma_{V}(L \cap  V_{A_,b}^{r_2,\delta}(\mathbb{R}))$ of $\sigma_V(V_{A,b}^{r_2,\delta}(\mathbb{Z}))$. \\
$\mathcal{L}^{\text{irr}}$ &the subset $\sigma_{V}(L^{\text{irr}} \cap  V_{A_,b}^{r_2,\delta}(\mathbb{R}))$ of $\sigma_V(V_{A,b}^{r_2,\delta}(\mathbb{Z}))$. \\
$\Gamma$ &the subgroup $\sigma_A(\SO_{A}(\mathbb{Z}))$ of $\SO_{A_{pq}}(\mathbb{R})$. \\
$\FF_{A_{pq}}$ &fundamental domain for the action of $\Gamma$ on $\SO_{A_{pq}}(\mathbb{R})$ contained in a finite union of $\SO_{A_{pq}}(\mathbb{Q})$ translates of a Siegel domain $\mathcal{S}$, $\FF_{A_{pq}} \subset \cup_\ell g_\ell \mathcal{S}$ for $g_\ell \in \SO_{A_{pq}}(\mathbb{Q})$. \\
$\mathcal{S}$ &a Siegel domain in $\SO_{A_{pq}}(\R)$ used in the definition of $\FF_{A_{pq}}$. \\
$g_\ell$ &element of finite list of elements in $\in \SO_{A_{pq}}(\mathbb{Q})$ used in the definition of $\FF_{A_{pq}}$. \\
$\widetilde{C}_\ell$ &the minimum absolute value of the non-zero entries of elements of $g_\ell^{-1} \sigma_{V}(V_{A,b}(\Z))$. \\
$\FF_{A_{pq}}'$ &the set of $h \in \FF_{A_{pq}}$ such that for some $\ell$, $\lvert b_{11}(g_\ell^{-1}v)\rvert < \widetilde{C_\ell}$ for all $v \in hG_{0}\sigma_{V}(R_{A,b}^{r_{2},\delta})(X)$. We used it to define the cuspidal part of integral \eqref{eqn: averaged count}. \\
$\E$ &the coordinates on $V_{A_{pq},b}$ given by the matrix entries $\{b_{ij}\}$ \\ &for $1 \le i \le j \le n$ and $(i,j) \neq ((n-1)/2,(n-1)/2)$. \\
$w(b_{ij})$ &the weight $w(b_{ij})$ of an element $b_{ij} \in \mathcal{E}$. \\
$\widetilde{I}(\E_1)$ &the active integral of a subset $\E_1 \subset \mathcal{\E}$ defined in \S \ref{subsec: proof of cusp cutting estimate}. \\
$I(\E_1)$ &normalisation of the active integral defined in \S \ref{subsec: proof of cusp cutting estimate}. \\
$\Lambda_{A,b}^\delta$ &a collection of local specifications $\Lambda_{A,b}^{\delta}  = \prod_{p} \Lambda_{A,b,p} \times V_{A,b}^{r_2,\delta}(\R) \subset V_{A,b}(\widehat{\Z}) \times V_{A,b}(\R)$, for some $r_{2}$ and $\delta \in \mathcal{T}(r_{2})$. \\
$\V(\Lambda_{A,b}^\delta)$ &the preimage of $\Lambda_{A,b}$ under the diagonal embedding $V_{A,b}^{r_{2},\delta}(\Z) \hookrightarrow V_{A,b}(\widehat{\Z}) \times V_{A,b}(\R)$. \\
$\omega_U$ &the Euclidean measure on $U_{A,b}$ normalised so that $U_{A,b}(\mathbb{Z})$ has covolume $1$.  \\
$\omega_V$ &the Euclidean measure on $V_{A,b}$ normalised so that $V_{A,b}(\mathbb{Z})$ has covolume $1$.  \\
$\omega_A$ &choice of volume form on $\SO_A$ made in \S \ref{subsec: choice of volume forms} following Tamagawa \cite[\S 7, \S 8]{MR0212025} or Hanke \cite[\S 2]{MR2155083}. \\
$m_p(f,A)$ & the local mass of $f$ at $p$ in $A$. \\
$\J_A$ &the Jacobian in the change of measure formula Proposition \ref{change of measure formula}. It is a rational number which does not depend on $b$. \\
$m_\infty(r_2,A)$ & the archimedean mass of $f$ in $A$. \\
$\mathfrak{M}_{\pm 1}$ & representatives of the equivalence classes of symmetric bilinear forms of determinant $\factor$ which coincide modulo $2$. The $\pm 1$ subscript indicates the value of their respective Hasse--Witt invariant. \\
$\pm_8$ &is $+$ if $n$ is congruent to $1, 3 \pmod 8$ and $-$ otherwise. \\
$\mp_8$ &is the negative of $\pm_8$. \\
$c_2(S_{2,b}, \mathfrak{M}_{\pm 1})$ & average of $m_2(f,\mathfrak{M}_{\pm 1})$ over $S_{2,b} \subset U_{1}(\Z_2)$. \\
$\eta'$ & restriction of the quadratic form $\eta$ of \cite[\S 6.2]{MR3156850} to $R_f^\times/(R_f^\times)^2$ for $f \in U_1(\Z_2)$ maximal.\\
$c_{\infty,\pm 1}$ & sum of the masses $m_{\infty}(r_{2},\cdot)$ over all $\SL_n(\R)$-equivalence classes of real symmetric matrices of determinant $\factor$ which have Hasse--Witt symbol of $1$ and $-1$ respectively.\\
$\delta_{\gg 0}$ & the element $(11\cdots 1 1) \in \T(r_2)$. \\
$\lesssim$ &an inequality which becomes an equality if $\Sigma$ is very large or if $\Sigma$ is acceptable and the tail estimate of Conjecture \ref{conj: conjectural tail estimate} holds. \\
$\tau(\cdot)$ & the Tamagawa number. \\
$c_p(S_{b,p},\pm 1)$ &average of the sum of masses at $p$ of $f$ in equivalence classes of forms having Hasse--Witt symbol $\pm 1$ respectively over $S_{b,p} \subset U_{N}(\Z_p)$. \\
\end{longtable}

\bibliographystyle{amsalpha}
\bibliography{bibfile}

\end{document}